\documentclass{amsart} 

\usepackage{amssymb,amscd}
\usepackage[dvipsnames]{xcolor}
\usepackage{amsxtra}
\usepackage{soul}
\usepackage[mathscr]{eucal}
\usepackage{graphics}
\usepackage[final]{ifdraft}
\ifoptionfinal{}{\pagecolor[HTML]{FDF6E3}}
\usepackage{hyperref}

\usepackage[final]{showkeys}
\definecolor{labelkey}{rgb}{0.4,0.75,0.75}

\numberwithin{equation}{section}

\newtheorem{theorem}{Theorem}[section]
\newtheorem{proposition}[theorem]{Proposition}
\newtheorem{lemma}[theorem]{Lemma}
\newtheorem{corollary}[theorem]{Corollary}

\theoremstyle{definition}

\newtheorem{definition}[theorem]{Definition}
\newtheorem{example}[theorem]{Example}

\renewcommand{\AA}{\mathbb A} 
\newcommand{\absolute}[1]{{\left | {#1} \right |}}
\newcommand{\norm}[1]{{\left \| {#1} \right \|}}

\newcommand{\h}{\textsf{\textup{h}}}
\newcommand{\Hh}{\textsf{\textup{H}}}

\newcommand{\od}{\operatorname{d}}
\newcommand{\Id}{\operatorname{Id}}

\newcommand{\GL}{\operatorname{GL}}
\newcommand{\vol}{\operatorname{vol}}

\newcommand{\diag}{\mathrm{diag}}
\newcommand{\uni}{\mathrm{uni}}
\newcommand{\base}{\mathrm{base}}
\newcommand{\irred}{{\mathrm{irred}}}
\newcommand{\posfact}{{\mathrm{pos}}}
\newcommand{\peqref}[1]{(\ensuremath{\ref{#1}'})}
\newcommand{\ptag}[1]{\tag{$\ref{#1}'$}}

\newcommand{\be}{\mathbf{e}}

\newcommand{\bm}{\mathbf{m}}
\newcommand{\bn}{\mathbf{n}}

\newcommand{\n}{\mathbf{n}}

\newcommand{\CC}{\mathbb{C}}
\newcommand{\RR}{\mathbb{R}}

\newcommand{\KK}{\mathbb{K}}

\newcommand{\LL}{\mathbb{L}}
\newcommand{\NN}{\mathbb{N}}
\newcommand{\QQ}{\mathbb Q}
\newcommand{\TT}{\mathbb{T}}
\newcommand{\ZZ}{\mathbb{Z}}
\newcommand{\zd}{\ZZ^{d}}
\newcommand{\Zd}{\ZZ^{d}}

\newcommand{\cA}{\mathcal{A}}
\newcommand{\cB}{\mathcal{B}}
\newcommand{\cC}{\mathcal{C}}

\newcommand{\cP}{\mathcal{P}}

\makeatletter
\@namedef{subjclassname@2020}{%
  \textup{2020} Mathematics Subject Classification}
\makeatother

\begin{document}

\title[Rigidity properties for commuting automorphisms]{Rigidity properties
for commuting automorphisms on tori and solenoids}

\author{Manfred Einsiedler and Elon Lindenstrauss}

\address{Manfred Einsiedler, Department of Mathematics, ETH Z\"urich,
R\"amistrasse 101, CH-8092 Z\"urich, Switzerland}

\address{Elon Lindenstrauss, Einstein Institute of Mathematics, The Hebrew University of Jerusalem, Jerusalem 91904, Israel}

\thanks{M.E.~acknolewdges the support by SNF (grant no.~200020-178958) E.L.~supported by ERC 2020 grant HomDyn (grant no.~833423). Both authors thank the Hausdorff Research Institute for
Mathematics at the Universit\"at Bonn for its hospitality during the trimester program ``Dynamics: Topology and Numbers''.}




\begin{abstract}
Assuming positive entropy we prove a measure rigidity
theorem for higher rank actions on tori and solenoids
by commuting automorphisms.
We also apply this result to obtain a complete classification of disjointness and
measurable factors for these actions.
\end{abstract}

\dedicatory{To the memory of Anatole Katok}

\maketitle
\tableofcontents

\section{Introduction and main results}\label{sec: Intro}

The map $T_p:x \mapsto px$ on $\TT=\RR/\ZZ$ has many closed
invariant sets and many invariant measures. Furstenberg initiated
the study of jointly invariant sets in his seminal paper
\cite{Furstenberg-disjointness-1967}. A set $A \subseteq \TT$ is
called {\em jointly invariant} under $T_p$ and $T_q$ if $T_p
(A)\subseteq A$ and $T_q(A)\subseteq A$. Furstenberg proved that if
$p$ and $q$ are multiplicatively independent integers then any
closed jointly invariant set is either finite or all of $\TT$.

Furstenberg also raised the question what the jointly invariant
measures are: which probability measures $\mu$ on $\TT$ satisfy
$(T_p)_*\mu=(T_q)_*\mu=\mu$. The obvious ones are the Lebesgue
measure, atomic measures supported on finite invariant sets, and
(non-ergodic) convex combinations of these.

\label{defn: irreducibility}
In the following a {\em solenoid} $X$ is a compact, connected,
abelian group whose Pontryagin dual $\widehat{X}$ can be embedded
into a finite-dimensional vector space over $\QQ$. The simplest
example is a finite-dimensional torus. A $\zd$-action $\alpha$ by
automorphisms of a solenoid $X$ is called {\em irreducible} if
there is no proper infinite closed subgroup which is invariant
under $\alpha$, and {\em totally irreducible} if there is no
finite index subgroup $\Lambda \subseteq\zd$ and no proper infinite
closed subgroup $Y \subseteq X$ which is invariant under the induced
action $\alpha_\Lambda$. A $\zd$-action is {\em virtually cyclic}
if there exists $\n \in\zd$ such that for every element $\bm \in
\Lambda$ of a finite index subgroup $\Lambda \subseteq\zd$ there
exists some $k \in \ZZ$ with $\alpha^\bm=\alpha^{k\n}$.

We briefly summarize the history of this problem. The topological
generalization of Furstenberg's result to higher dimensions was
given by Berend
\cite{Berend-invariant-tori,Berend-invariant-groups}: An action on
a torus or solenoid has no proper, infinite, closed, and invariant
subsets if and only if it is totally irreducible, not
virtually-cyclic, and contains a hyperbolic element. 

The first partial result for the measure problem on $\TT$ was
given by Lyons \cite{Lyons-2-and-3} under a strong additional
assumption. Rudolph \cite{Rudolph-2-and-3} weakened this
assumption considerably, and proved the following theorem.

\begin{theorem}\cite[Thm.~4.9]{Rudolph-2-and-3}   \label{thm: Rudolph}
Let $p, q \geq 2$ be relatively prime positive integers, and let
$\mu$ be a $T_p$, $T_q$-invariant and ergodic measure on $\TT$. Then either
$\mu=m_\TT$ is the Lebesgue measure on $\TT$, or the entropy of $T_p$ and
$T_q$ is zero.
\end{theorem}

Johnson \cite{Johnson-invariant-measures} lifted the relative
primality assumption, showing it is enough to assume that $p$ and
$q$ are multiplicatively independent. 
Feldman \cite{Feldman-generalization}, Parry \cite{Parry-2-3}, and Host
\cite{Host-normal-numbers} have found different proofs of this
theorem, but positive entropy remains a crucial assumption.

Anatole Katok and Spatzier
\cite{Katok-Spatzier,Katok-Spatzier-corrections} obtained the
first analogous results for actions on higher dimensional tori and
homogeneous spaces. However, their method required either an
additional ergodicity assumption on the measure (satisfied for
example if every one parameter subgroup of the suspension acts
ergodically), or that the action is totally non-symplectic (TNS).
A careful and readable account of these results has been
written by Kalinin and Anatole Katok \cite{Kalinin-Katok-Seattle}, which
also fixed some minor inaccuracies. The following 
theorem (already proven in the announcement \cite{EL-ERA03})
gives a full generalization of the result of Rudolph and Johnson
to actions on higher-dimensional solenoids.

\begin{theorem}\cite[Thm.~1.1]{EL-ERA03}\label{thm: easy-case}
Let $\alpha$ be a totally irreducible, not virtually-cyclic
$\zd$-action by automorphisms of a solenoid $X$.
Let $\mu$ be an $\alpha$-invariant and ergodic probability 
measure. Then either $\mu=m_X$
is the Haar measure of $X$, or the entropy $\h_\mu(\alpha^\n)=0$
vanishes for all $\n \in\zd$.
\end{theorem}

\subsection{The general positive entropy measure rigidity theorem}

Without total
irreducibility the Haar measure of the group is no longer the only
measure with positive entropy. Thus our main theorem below is
(necessarily) longer in its formulation than
Theorem \ref{thm: easy-case}. It strengthens e.g.\ 
\cite[Thm.~3.1]{Kalinin-Katok-Seattle} which has a similar conclusion but stronger assumptions.

\begin{theorem}[Positive entropy rigidity theorem]\label{thm: main}
Let $\alpha$ be a $\Zd$-action ($d \geq 2$) by automorphisms of
a solenoid $X$.
Suppose $\alpha$ has no virtually cyclic factors, and let
$\mu$ be an $\alpha$-invariant and ergodic probability measure on $X$.
Then there exists a subgroup $\Lambda \subseteq\Zd$ of finite index and a decomposition
$\mu=\frac{1}{J}(\mu_1+\ldots+\mu_J)$ of $\mu$ into mutually singular measures
with the following properties for every $j=1,\ldots,J$.
\begin{enumerate}
\item The measure $\mu_j$ is $\alpha_\Lambda$-ergodic, where
$\alpha_\Lambda$ is the restriction of $\alpha$ to $\Lambda$.
\item There exists an $\alpha_\Lambda$-invariant closed subgroup
$G_j$ such that $\mu_j$ is invariant under translation with
elements in $G_j$, i.e.\ $\mu_j(B)=\mu_j(B+g)$ for all $g \in G_j$
and every measurable set $B\subseteq X$.
\item For $\n \in\zd$ 
$\alpha^\n_*\mu_j=\mu_k$ for some $k\in\{1,\ldots,J\}$ and $\alpha^\n(G_j)=G_k$.
\item The measure $\mu_j$ induces a measure on the factor $X/G_j$
with $\h_{\mu_j}(\alpha^{\n}_{X/G_j})=0$ for any $\n \in \Lambda$.
(Here $\alpha_{X/G_j}$ denotes the action induced on $X/G_j$).
\end{enumerate}
\end{theorem}

We remark that in the topological category there is a big gap between our understanding of the totally irreducible case and the general case of $\ZZ ^ d$-actions by automorphisms on a solenoid. In the totally irreducible case Berend \cite{Berend-invariant-groups} gave an if-and-only-if condition for a $\ZZ ^ d$-action to have the property that every orbit is either finite or dense, and the same methods could be pushed further to give a complete classification of closed invariant subsets for a totally irreducible $\ZZ ^ d$-action on the solenoid; for $\ZZ ^ d$-action on tori this is due to Z. Wang \cite[Thm 1.10]{Wang-nonhyperbolic} and his proof certainly works also for solenoids though this does not seem to have been written (a special case, with a very nice application, can be found in Manner's paper \cite{Manners-pyjama}). In the non-irreducible case, orbit closures, and closed invariant sets are much less understood. We refer to \cite{Lindenstrauss-Wang} by Z. Wang and the second named author for some results in this direction and more details.

The proofs of Theorem \ref{thm: easy-case} and Theorem \ref{thm:
main} follow the outline of Rudolph's proof of Theorem \ref{thm:
Rudolph}. One of the main ingredients there was the observation
that $\h_\mu(T_p)/ \log p = \h_\mu(T_q)/\log q$ 
(and a relativized version of this equality).
This follows from the particularly simple geometry of this system 
where both $T_p$ and $T_q$ expand the
one-dimensional space $\TT$ with fixed factors. There is no simple geometrical 
reason why such an equality should be true for more complicated $\ZZ ^ d$-actions 
on solenoids, and indeed is easily seen to fail in the reducible case.
However, somewhat surprisingly, such an equality {\em is} true for 
irreducible $\ZZ ^ d$-actions, even though this is true from subtler reasons (see Theorem~\ref{thm: shape} below).

In the following two subsections we also apply Theorem \ref{thm: main} to obtain new
information about the measurable structure, with respect to the
Haar measure, of algebraic $\zd$-actions on tori and
solenoids.

\subsection{Characterization of disjointness}

Let $\alpha_1$ and $\alpha_2$ be two measure-preserving
$\zd$-actions on the probability spaces $(X_1,\cB_{X_1},\mu_1)$
and $(X_2,\cB_{X_2},\mu_2)$. A {\em joining} between $\alpha_1$
and $\alpha_2$ is an $\alpha_1\times\alpha_2$-invariant
probability measure $\nu$ on $X_1\times X_2$, which projects to
$\mu_1$ and $\mu_2$ under the projection maps $\pi_1$ and $\pi_2$. 
In other words we require $\nu(\alpha_1^{\n}\times\alpha_2^{\n}(C))=\nu(C)$
for $\n\in\ZZ^d$ and $C\in\cB_{X_1\times X_2}$,
$\nu(A\times X_2)=\mu_1(A)$ for $A\in\cB_{X_1}$,
and also $\nu(X_1\times B)=\mu_2(B)$ for $B\in\cB_{X_2}$. The product measure
$\mu_1\times\mu_2$ is always a joining, called the {\em trivial
joining}. If the trivial joining is the only joining, the two
actions are {\em disjoint}. This implies 
that the two actions are measurably non-isomorphic. In fact if
they are disjoint there is no nontrivial common factor of the two
systems, see for instance Section \ref{sec: inv_algebra} where we
recall the construction of the relatively independent joining over
a common factor.

Let now $\alpha_j$ be measure preserving $\zd$-actions on
$(X_j,\cB_{X_j},\mu_j)$ for $j=1,\ldots,r$. A {\em joining between
$\alpha_j$ for $j=1,\ldots,r$} is a measure $\nu$ on
$\prod_{j=1}^rX_j$ which projects to $\mu_j$ under the coordinate projections $\pi_j$ for
$j=1,\ldots,r$, and is invariant under the $\zd$-action
$\alpha_1\times\cdots\times\alpha_r$. The product measure is the
{\em trivial joining}, and the $\zd$-actions are {\em mutually
disjoint} if the trivial joining is the only joining.

Suppose now $\alpha_1$ and $\alpha_2$ are actions by automorphisms
on solenoids $X_1$ and $X_2$ respectively. We will classify disjointness
with respect to the Haar measures $m_{X_j}$
on the group $X_j$ for $j=1,2$.
If $\varphi:X_1\rightarrow X_2$ is a surjective continuous
homomorphism and satisfies
$\alpha_2^\n\circ\varphi=\varphi\circ\alpha_1^\n$ for all
$\n\in\zd$, we say $\varphi$ is an {\em algebraic factor map}. If
$\alpha_1$ and $\alpha_2$ are both finite-to-one factors of each
other by algebraic factor maps, we say they are {\em algebraically
weakly isomorphic}. Equivalently, $\alpha_1$ and $\alpha_2$ are algebraically weakly isomorphic
if they have a common finite-to-one algebraic factor.

The following generalizes a theorem of Kalinin and Anatole Katok
\cite[Thm.~3.1]{Kalinin-Katok} resp.\ Kalinin and Spatzier
\cite[Thm.~4.7]{Kalinin-Spatzier}, where the main difference
is that we do not assume that the actions are totally
non-symplectic or hyperbolic.

\begin{corollary}[Classification of disjointness]\label{thm: joining}
 If $\alpha_1$ and $\alpha_2$ are totally irreducible and not virtually cyclic,
 then they are not disjoint (with respect to the Haar measures)
 if and only if there exists a finite index subgroup
 $\Lambda\subseteq\zd$ for which $\alpha_{1,\Lambda}$ and $\alpha_{2,\Lambda}$ are
 algebraically weakly isomorphic.

 More generally, let $\alpha_j$ be $\zd$-actions on solenoids (not necessarily irreducible)
 without virtually-cyclic factors for $j=1,\ldots,r$. Then they
 are not mutually disjoint if and only if there exist indices
 $i, j\in\{1,\ldots,r\}$ with $i\neq j$,
 a finite index subgroup $\Lambda\subseteq\zd$, and
 a nontrivial $\Lambda$-action $\beta$ on a 
 solenoid $Y$ which is an algebraic factor of
 $\alpha_{i,\Lambda}$ and $\alpha_{j,\Lambda}$.
\end{corollary}

\subsection{Algebraicity of factors}
Anatole Katok, Svetlana Katok and Schmidt
\cite[Thm.~5.6]{Katok-Katok-Schmidt} studied measurable factor
maps between $\zd$-actions by automorphisms of tori. Our second
application gives an extension of this by characterizing 
the structure of measurable factors resp.\ invariant
$\sigma$-algebras. We start by giving two algebraic constructions
that give invariant $\sigma$-algebras.
\begin{itemize}
 \item If $X'\subseteq X $ is a closed $\alpha $-invariant subgroup and
$\pi:X\rightarrow X/X' $ denotes the canonical projection map,
then the preimage $\cA=\pi^{-1}\cB_{X/X'}$ of the Borel $\sigma
$-algebra $\cB_{X/X'} $ of $X/X'$ is $\alpha $-invariant.
 \item If $\Gamma $ is a finite group of affine automorphisms that is
normalized by $\alpha $, then the $\sigma $-algebra $\cB_X^\Gamma$
of $\Gamma$-invariant Borel subsets of $X$ is $\alpha $-invariant.
\end{itemize}

\begin{corollary}[Algebraicity of measurable factors]\label{thm: algebra}
 Let $\alpha$ be a $\zd$-action by automorphisms of the
 solenoid $X$ without virtually cyclic factors,
 and let $\cA\subseteq\cB_X$ be an invariant $\sigma$-algebra.
 Then there exists a closed $\alpha$-invariant subgroup $X'\subseteq
 X$ and a finite group $\Gamma$ of affine automorphisms of $X/X'$
 that is normalized by the action $\alpha_{X/X'}$ induced by $\alpha$ on $X/X'$
 such that
 \[
  \cA=\pi^{-1}\left(\cB_{X/X'}^\Gamma\right)\ \operatorname{modulo}m_X.
 \]
\end{corollary}

 In other words, the corollary states that
 every measurable factor of $\alpha$ arises by a 
 combination of the
 two algebraic constructions given above.

In the irreducible case the theorem
gives that every nontrivial measurable factor of $\alpha$ is a 
quotient of $X$ by the action of a finite affine group.
The simplest examples of such groups are
finite translation groups. However, more complicated examples are also 
possible; for example, let $w \in X$ be any
$\alpha$-fixed point. Then the action of $G=\{ \Id,-\Id+w \}$ on $X$
commutes with $\alpha$.

The proof of Corollary \ref{thm: algebra} uses the relatively
independent joining of the Haar measure with itself over the
factor $\cA$, which gives an invariant measure on $X \times X$
analyzable by Theorem \ref{thm: main}. This is similar to the
proof of isomorphism rigidity in \cite{Katok-Katok-Schmidt}, which
followed a suggestion by Thouvenot.

We will discuss further corollaries
towards factors in Section \ref{sec: inv_algebra}.

\subsection{Remarks and acknowledgements}
The results of this paper were obtained in 2002 and announced in \cite{EL-ERA03}; indeed this was the first result we worked on together. Since then there was always another newer result that we wanted to write, and we never seemed to have the time to finally write down the general case of the results announced in \cite{EL-ERA03}. One important ingredient in this work is the product structure for coarse Lyapunov foliations developed around that time by Anatole Katok and the first author.

The ideas behind the proof of Theorem~\ref{thm: main} were used by Zhiren Wang to prove his strong measure classification result for invariant measures on nilmanifolds \cite{Wang-nilmanifolds}. 
Actions by automorphisms on nilmanifolds generalize actions on tori which are covered by the results of this paper; solenoids are more general than tori, but more importantly in that paper Zhiren Wang does not allow for zero entropy factors, as we do here. Hence the results of this paper are (to the best of our knowledge) ``new'' in the sense that they have not appeared in print before.
We thank Zhiren for encouraging us to write down the complete proof of \cite{EL-ERA03} and for his willingness to help us do so.
We also would like to thank the anonymous referee
and Manuel Luethi for their comments.

\section{Actions on adelic solenoids}\label{sec: solenoids}

\subsection{Adeles, local and global fields}
We review some basic facts and definitions regarding local fields, global fields, and the adeles. A general reference to these topics is Weil's classical book \cite[Ch. I-IV]{Weil}; note that Weil calls what is now commonly referred to as global fields $\AA$-fields. Throughout this paper, the term {\em local field} will denote a
locally compact field of characteristic zero\footnote{The terminology of global and local fields was introduced to incorporate both the positive and zero characteristic cases on an equal footing, but dynamically there are rather fundamental differences (see e.g. \cite{Kitchens-Schmidt, Einsiedler-bad-measures}) and we restrict ourselves in this paper to the zero characteristic case.}; these include $\RR$
and $\CC$ as well as finite extensions of the field of
$p$-adic numbers $\QQ_p$. Let $\KK$ be a local field and let
$\lambda_\KK$ be the Haar measure on $\KK$. 
We define $\delta(\KK)$ as the degree of the field
extension $\KK$ over the closure of $\QQ$ in $\KK$, which can be isomorphic to either $\RR$ if $\KK$ is Archimedean or $\QQ_p$ for some prime $p$ otherwise (to make the notations more consistent, we will also write $\QQ_\infty$ for $\RR$). 
Local fields come equipped with an absolute value $|\cdot|_\KK$, which we will always normalize to coincide with the usual absolute value on $\RR$ or $\QQ_p$.
We note that in any of these cases we have
\begin{equation}\label{eq: norm_norm}
\lambda_\KK(aC)=|a|_\KK^{\delta(\KK)}\lambda_\KK(C)
\end{equation}
for any measurable set $C \subseteq\KK$. 

We recall that a {\em global field} $\KK$ is a finite field extension of $\QQ$. 
We will denote the completions of $\KK$ by $\KK_\sigma$,
where $\sigma$ stands for the (Archimedean or non-Archimedean) \emph{place} --- i.e.\ an equivalence class of absolute values. We choose the representative to coincide with either $\absolute\cdot_\infty$ of $\absolute\cdot_p$ on $\QQ$. 
We recall that $\KK_\sigma$ is a local field and will use the abbreviation  $\absolute\cdot_\sigma=\absolute\cdot_{\KK_\sigma}$ 
for the norms satisfying \eqref{eq: norm_norm} 
on $\KK_\sigma$.
If $\absolute{\cdot}_\sigma$ coincides with $\absolute{\cdot}_p$ on $\QQ$ then we say that $\sigma$ lies over $p$; if $\absolute{\cdot}_\sigma$ is Archimedean we say that $ \sigma$ is an infinite place of $\KK$.

For a global field $\KK$ the ring of adeles $\AA_\KK$ over $\KK$
is defined as the restricted direct product of all completions of $\KK$
with respect to the maximal compact subrings for all non-Archimedean completions, in other words $(t_\sigma)_\sigma \in \AA_\KK$ if $t_\sigma \in \KK_\sigma$ for all places $\sigma$ of $\KK$ and except for finitely many $\sigma$ (an exceptional set that is assumed to include all infinite places) we have that in fact $t_\sigma$ lies in the maximal compact subring $\mathcal{O}_{\KK,\sigma} < \KK_\sigma$.
In the special case $\KK=\QQ$ this takes the form 
\[
\AA=\textstyle{\RR\times\prod_p'\QQ_p}=\RR\times\bigcup_S\left(\prod_{p\in S}\QQ_p\times\prod_{p\notin S}\ZZ_p\right),
\]
where the union runs over all finite subsets $S$ of the primes. 
The general case of the ring of adeles $\AA_\KK$
over a global field $\KK$ is defined similarly,
but can also be obtained via
\begin{equation}\label{eq:adelictensor}
 \AA_\KK=\AA\otimes_\QQ\KK.
\end{equation}
We shall identify $\KK_\sigma$ with the corresponding subring in
$\AA_\KK$.
Using a basis of $\KK$ over $\QQ$ we obtain an additive group
isomorphism (indeed, an isomorphism of vector spaces over $\QQ$) 
\begin{equation}\label{eq:adelicisomorphism}
    \AA_\KK=\AA\otimes_\QQ\KK\cong \AA^{[\KK:\QQ]}.
\end{equation}

We recall moreover that $\QQ$ diagonally embedded into $\AA$ is discrete and
cocompact and that the Pontryagin dual $\widehat{\AA}$ of $\AA$ 
can be identified with $\AA$
itself. Finally the isomorphism between 
$\widehat{\AA}$  and $\AA$ can be chosen 
so that the annihilator of $\QQ$ is $\QQ$ itself, which implies that
the Pontryagin dual of $\QQ$ can be identified with $\AA/\QQ$.
This extends similarly to global fields, see e.g. \cite[pp.~64--69]{Weil}. 

\subsection{Adelic actions}\label{adelicaction}
For us the adelic setup gives a concrete language to discuss actions on general solenoids. 
We note however, that for automorphisms on tori it suffices to consider all Archimedean
places of $\KK$ and for irreducible actions it would suffice to consider only
finitely many places (see also \cite{EL-ERA03} for the latter). 

Indeed, let us fix a dimension $m\geq 1$, a rank $d\geq 1$, and $d$ commuting
matrices $A_1,\ldots,A_d\in\GL_m(\QQ)$. We
use them to define a linear representation 
$\widetilde{\alpha}$ of $\ZZ^d$
on $\QQ^m$. Using the matrices in the same way 
as within vector spaces this extends to
an action of $\ZZ^d$ by group automorphisms on $\AA^m$, which we will also denote by $\widetilde{\alpha}$. Finally we take the quotient
by the discrete cocompact invariant subgroup $\QQ^m$ and obtain 
an action $\alpha$ of $\ZZ^d$
by automorphisms on the solenoid
\[
X_m=\AA^m/\QQ^m.
\]
We will refer
to $X_m$ as an {\em adelic solenoid}
and to this action as the {\em adelic action} on $X_m$ defined
by the matrices (or equivalently the linear maps) $A_1,\ldots,A_d$. 

Since every group automorphism of $\QQ^m$
is in fact $\QQ$-linear and defined by an invertible matrix in $\GL_m(\QQ)$,
it follows from Pontryagin duality that every action of $\ZZ^d$ by automorphisms on $X_m$
can be defined this way. We will explain this step in a more general form 
in Section \ref{ssec:extension}.

We say that a closed subgroup $Y<X_m$ of an adelic 
solenoid $X_m=\widehat{\QQ}^m$ 
is \emph{adelic} if it is a linear subspace over $\QQ$ 
(i.e.\ $\QQ Y\subseteq Y$). Since this notion will be useful
for us we wish to study it briefly in the following lemma.

\begin{lemma}\label{lem:adelic}
 Let $m\geq 1$ and let $Y\leq X_m$
 be a closed subgroup. Then the following
 conditions are equivalent.
 \begin{enumerate}
     \item $Y\leq X_m$ is an adelic subgroup.
     \item The annihilator $Y^\perp\leq\QQ^m$ is a $\QQ$-linear subspace.
     \item There exists a $\QQ$-linear subspace $V\leq\QQ^m$
     so that $Y$ is the image of $\AA\otimes_\QQ V\leq\AA^m$
     modulo $\QQ^m$.
 \end{enumerate}
\end{lemma}

\begin{proof}
 The equivalence of (1) and (2) follows from Pontryagin
 duality. Indeed $aY=Y$ for $a\in\ZZ\setminus\{0\}$ (and then also $a\in\QQ\setminus\{0\}$)
 is equivalent to $a(Y^\perp)=Y^\perp$ (since $(aY)^\perp=a^{-1}Y^\perp$). 
 
 Suppose now $V<\QQ^m$ is a linear subspace as in (3).
 Then $\AA\otimes_\QQ V$ is clearly invariant under $\QQ$
 and hence defines modulo $\QQ^m$ an adelic subgroup.
 
 Finally assume that $Y$ is adelic as in (1) (and 
 equivalently (2)). Let $W=Y^\perp<\QQ^m$ so that $W$
 is a linear subspace and $Y=W^\perp$ 
 by Pontryagin duality. By \cite[Ch.~IV]{Weil} there exists a 
 character $\chi_0\in\widehat{\AA}$
 so that the isomorphism $\widehat{\AA}\cong\AA$ is induced
 by the definition $\langle a,b\rangle=\chi_0(ab)$ for all $a,b\in\AA$
 and with this isomorphism we have $\QQ^\perp=\QQ$.
 Moreover, this also gives $\widehat{\AA^m}\cong\AA^m$ using the pairing
 \[
  \bigl\langle (a_1,\ldots,a_m),(b_1,\ldots,b_m)\bigr\rangle
  = \chi_0\bigl(a_1b_1+\cdots+a_mb_m\bigr)
 \]
 for all $(a_1,\ldots,a_m),(b_1,\ldots,b_m)\in\AA^m$.
 Since $W<\QQ^m$ is a linear subspace we may apply a linear isomorphism 
 $A\in\operatorname{GL}_m(\QQ)$ so that $W_1=A(W)$ is precisely the span of the first $k$
 standard basis vectors. Applying the inverse of the dual (transpose) linear automorphism
 to $Y$ this shows that $Y_1=(A^t)^{-1}(Y)$ satisfies that $Y_1^\perp=W_1$. 
 Now let $(a_1,\ldots,a_m)\in Y_1$. Hence we have $\chi_0(a_jb)=1$ for all 
 $b\in\QQ$ and $j=1,\ldots,k$. However, this gives by the properties of $\chi_0$
 that $a_j\in\QQ$ for $j=1,\ldots,k$. It follows that $Y_1=\QQ^m+\AA\otimes_\QQ V_1$,
 where $V_1$ is the linear hull of the last $m-k$ basis vectors. Applying $A^t$
 to this claim gives the description of $Y$ as in (3).
\end{proof}

\subsection{Irreducible adelic actions}\label{irreducible section}

We say that an adelic action on $X_m$ is {\em $\AA$-irreducible} 
if the associated linear representation
of $\ZZ^d$ on $\QQ^m$ is irreducible over $\QQ$, i.e.\ if there does not exist
a rational nontrivial proper invariant subspace. Note, however, that $\AA$-irreduciblity does not coincide with the notion of irreducibility defined on~p.~\pageref{defn: irreducibility}. In fact
an adelic action is never irreducible, but it will be convenient to 
study $\AA$-irreducible adelic actions as basic building blocks
of other adelic actions.

We note that given a global field $\KK$ 
and $d\geq 1$ elements $\zeta_1,\ldots,\zeta_d\in\KK$ 
we may consider multiplication by these elements
as a $\QQ$-linear map on the vector space $\KK$ over $\QQ$
to define an adelic action of $\ZZ^d$ on $\AA_\KK/\KK$.
Using a fixed basis of $\KK$ over $\QQ$
we may identify $\KK$ with $\QQ^m$
and multiplication by $\zeta_1,\ldots,\zeta_d$
with certain matrices $A_1,\ldots,A_d$. 
In this way our discussions of Section \ref{adelicaction}
also apply to the multiplication maps
by $\zeta_1,\ldots,\zeta_d$ on $\KK$.
The point of the following proposition
is that every $\AA$-irreducible action of $\ZZ^d$
arises in this way from a global number field
and $d$ of its elements.

\begin{proposition}[Diagonalization of $\AA$-irreducible action]\label{adelicirred}
 Let $m,d\geq 1$ and let $\alpha$ be an  $\AA$-irreducible adelic action of $\ZZ^d$ on $X_m=(\AA/\QQ)^m$. 
 Then there exists a global field $\KK$ of degree $m$ over $\QQ$ and 
 $d$ nonzero elements $\zeta_1,\ldots,\zeta_d\in\KK^\times$
 so that $\alpha$ is isomorphic to the action on $\AA_\KK/\KK$ 
 generated by the maps $a\in\KK\mapsto\zeta_j a\in\KK$ for $j=1,\ldots,d$. 
More explicitly, this action on $\AA_\KK/\KK$ (which as an additive topological group is isomorphic to $X_m$)
can be given as follows:
 \[
   \widetilde{\alpha}^\n: (v_\sigma)_\sigma\in\AA_\KK\mapsto 
      \bigl(\underbrace{\zeta_{1,\sigma}^{n_1}\cdots \zeta_{d,\sigma}^{n_d}}_{=\zeta_{\n,\sigma}}v_\sigma\bigr)_\sigma,
 \]
 where $\zeta_{1,\sigma},\ldots,\zeta_{d,\sigma}\in\KK_\sigma$ and $\zeta_{\n,\sigma}$
 denote the image of $\zeta_1,\ldots,\zeta_d$ respectively of $\zeta_\n=\zeta_1^{n_1}\cdots\zeta_d^{n_d}$
 in the completion $\KK_\sigma$.
\end{proposition}

\begin{proof}
 Let $\zeta_j=\widetilde{\alpha}^{\mathbf{e}_j}\in\GL_m(\QQ)$ for
 $j=1,\ldots,d$ be the matrices that define the action $\widetilde{\alpha}$
 on $\QQ^m$ and $\AA^m$ associated to $\alpha$. 
 
 We define $\KK=\QQ[\zeta_1,\ldots,\zeta_d]\subseteq\GL_m(\QQ)$ to
 be the ring of polynomial expressions $f$ in the
 matrices $\zeta_1,\ldots,\zeta_d$ and with rational coefficients. 
 We note that Lemma~\ref{lem:adelic} implies that
 $\QQ^m$ has no proper rational subspaces invariant under $\KK$. 
 Since $\zeta_1,\ldots,\zeta_d$ commute,
 it follows that any such polynomial expression $f\in\KK$ is either zero or is
 invertible (as an element of $\GL_m(\QQ)$). 
 In particular we have that $\KK$ is an integral domain. 
 As it is also a finite dimensional algebra over $\QQ$ it follows
 that $\KK$ is field extension of $\QQ$. Once more
 because $\QQ^m$ has no proper invariant subspaces it also follows that
 $\varphi: a\in\KK\mapsto a(\mathbf{e}_1)\in\QQ^m$ must be surjective. By 
 definition the kernel $\ker(\varphi)$ is an ideal, which implies
 that $\varphi$ is injective since $\KK$ is a field. 
 It follows that $\varphi$ is a linear isomorphism. 
 
 To summarize we have found a global field $\KK$ and 
 elements $\zeta_1,\ldots,\zeta_d\in\KK^\times$ so that up to a linear isomorphism
 our linear representation $\widetilde{\alpha}^\n$ on $\QQ^m$ is defined 
 for every $\n\in\Zd$ by 
 multiplication by $\zeta_\n=\zeta_1^{n_1}\cdots \zeta_d^{n_d}$ on the vector space $\KK$. 
 
 To obtain the adelic action we tensorize with $\AA$. On one hand 
 for the action on $\QQ^m$ this gives
 the action of $\ZZ^d$ on $\AA^m$ we started with. On the other hand we may 
 tensorize the linear isomorphism between $\QQ^m$ and $\KK$
 with $\AA$ to obtain the group isomorphism
 \[
  \AA^m=\QQ^m\otimes_\QQ\AA\cong\KK\otimes_\QQ\AA\cong
  \bigl(\KK\otimes_\QQ\RR\bigr)\times\textstyle{\prod'_p\bigl(\KK\otimes_\QQ\QQ_p\bigr)}.
 \]
 Now notice that we can identify $\KK$ with the quotient $\QQ[x]/(p(x))$ for
 some irreducible polynomial $p(x)\in\QQ[x]$, which implies that 
 $\KK\otimes\RR$ is isomorphic $\RR[x]/(p(x))$. Since the irreducible factors
 of $p(x)\in\RR[x]$ correspond precisely to the roots of $p(x)$ (all appearing
 with multiplicity one) and hence also 
 to the Galois embeddings of $\KK$ into $\CC$ it follows that $\KK\otimes\RR$
 is as a ring isomorphic to $\prod_{\sigma|\infty}\KK_\sigma$, where the product runs over all places of $\KK$ lying above $\infty$, i.e. over
 all Archimedean completions of $\KK$.
 
 This argument applies similarly for the tensor product with $\QQ_p$ 
 so that $\KK\otimes_\QQ\QQ_p$ is isomorphic as a ring to the product
 $\prod_{\sigma|p}\KK_\sigma$ and $\sigma$ denotes here all places of $\KK$ above $p$, see also \cite[p.~56]{Weil}. Applying this argument at all places of $\QQ$
 we obtain that $\AA^m$ is isomorphic to $\AA_\KK$. 

 Application of $\widetilde{\alpha}^{\mathbf{e}_j}$ corresponds under this
 isomorphism from $\AA^m$ to $\AA_\KK=\prod_{\sigma}'\KK_\sigma$ to multiplication
 by the image of $\zeta_j$ in the factor $\KK_\sigma$ for every place $\sigma$ of $\KK$.
 This gives the proposition.
\end{proof}

Let us write $\delta(\sigma)=\delta(\KK_\sigma)\in\NN$ for any place $\sigma$
of $\KK$. The following product formula is a crucial ingredient in our proof.

\begin{proposition}
 [Product formula]\label{prop: localprod}
Let $\alpha$ be an $\AA$-irreducible adelic $\zd$-action as in Proposition \ref{adelicirred}. 
Then we have
\begin{equation}\label{eq: prodform}
\prod_{\sigma}|a|_{\sigma}^{\delta(\sigma)}= 1\mbox{ for
every }a \in \KK\setminus\{0\}
\end{equation}
and this applies in particular to $a=\zeta_\n$ for every $\n\in\ZZ^d$.
\end{proposition}

We note that one way to obtain this result is precisely to interpret
the product as the modular character for the automorphism
defined by multiplication by $a$ on the compact
group $\widehat{\KK}=\AA_\KK/\KK$
(cf.\ \eqref{eq: norm_norm}). 
We refer to \cite[p.~75]{Weil} for a proof along these lines.

\subsection{A filtration by $\AA$-irreducible adelic actions}

The following lemma reveals an advantage of adelic actions 
by connecting structural questions concerning $\alpha$ 
to linear algebra on the dual.

\begin{lemma}[Decomposition into $\AA$-irreducible factors]\label{factorlist}
 Let $m,d\geq 1$ and let $\alpha$ be an adelic $\ZZ^d$-action on $X_m$.
 Then there exist closed $\alpha$-invariant adelic subgroups
 \[
   Y_0=\{0\}<Y_1<\cdots< Y_r=X_m
 \]
 such that the action induced by $\alpha$ on $Y_{j}/Y_{j-1}$
 is an $\AA$-irreducible adelic $\ZZ^d$-action for all $j=1,\ldots,r$. 
\end{lemma}

We will refer to the $\AA$-irreducible adelic actions appearing in Lemma \ref{factorlist}
as the $\AA$-irreducible factors associated to $\alpha$. 

\begin{proof}
 By Pontryagin duality we may consider instead of $\alpha$ the linear 
 representation $\widehat{\alpha}$ on $\QQ^m$. Let $V_1\subseteq\QQ^m$
 be a nontrivial subspace that is invariant 
 under $\widehat{\alpha}$ and of minimal dimension.
 Note that this implies that the restriction of $\widehat{\alpha}$ to $V_1$
 is irreducible over $\QQ$. If $V_1\neq \QQ^m$ we let $V_2$
 be a subspace that is invariant under $\widehat{\alpha}$, strictly contains $V_1$
 and is among these of minimal dimension. Once more this implies that $V_2/V_1$
 is irreducible over $\QQ$ (for the representation induced by $\widehat{\alpha}$).
 
 Continuing like this we obtain a partial flag 
 \begin{equation}\label{eq:partialflag}
  V_0=\{0\}<V_1<V_2<\cdots<V_r=\QQ^m
 \end{equation}
 consisting of $\widehat{\alpha}$-invariant subspaces so that $V_{j}/V_{j-1}$
 is irreducible over $\QQ$. Applying Pontryagin duality (and reversing
 the indexing) this gives the lemma.
\end{proof}

\section{Leafwise measures, invariant foliations, and entropy}\label{sec: leaf-entropy} 

We briefly recall the main properties of leafwise measures. These have been
introduced in the context of higher rank rigidity theorems
(under the name of conditional measures for foliations) by Anatole Katok and Spatzier 
in \cite{Katok-Spatzier} 
and have since become an essential tool for all
of the theorems in the area. Implicitly leafwise measures appear already in the proof of Rudolph's Theorem in \cite{Rudolph-2-and-3}. A~general reference for this section is \cite[\S6-7]{EL-Pisa}.

\subsection{Leafwise measures}\label{sec:leafwiseXtilde}
Given a quotient $X=G/\Gamma$ of a locally compact abelian group $G$ by a lattice
$\Gamma<G$ and a closed subgroup $V<G$ with $V\cap\Gamma=\{0\}$
we consider the foliation of $X$
into $V$-orbits. Let $\pi_X$ denote the natural projection $G\to X$.
As we will reduce our main theorem to the adelic case (Theorem~\ref{thm:main2}) we consider the case that $G$ is  $\AA^m$ (though everything we say below is equally valid for $G=\RR^m$, or a finite product of local fields).  We note that the metric on $\AA^m$
is chosen so that the balls $B_r^V=B_r(0)\cap V$
have compact closures for all $r>0$.

For our purposes it will be important to work with an extension of $X$ ---
a product $\widetilde{X}=X\times\Omega$ 
of $X$ with an arbitrary 
compact metric space $\Omega$. 
We let $V$ act on $\widetilde{X}$ 
by translation on the first coordinate and trivially on the second coordinate, obtaining in this way a foliation of $\widetilde{X}$ into $V$ orbits.
This foliation of $\widetilde X$ into $V$-orbits does not admit in general a Borel cross section, and typically one cannot find a countably generated $\sigma$-algebra on $\widetilde X$ whose atoms coincide a.e. with $V$-orbits. Given a probability measure $\mu$ on $\widetilde X$, the foliation into $V$-orbits gives rise to a system of \emph{leafwise measures} on $\widetilde X$: a Borel measurable map $ x \mapsto \mu _ x ^ V$ from a subset of full measure $\widetilde X  ' \subset \widetilde X$ to locally finite (possibly infinite) measures on $V$.
We say that a leafwise measure $\mu_x^V$ is \emph{trivial} if
it is a multiple of the Dirac measure at the identity; we say that the system of leafwise measures is trivial if it is trivial at a.e.\ point. 
We also note that almost surely $0$ belongs to the support of $\mu_x^V$.

The system of leafwise measure satisfies the following \emph{compatibility condition}: for any $v\in V$ and $x \in \widetilde{X}'$ so that $x+v$ is also in ${\widetilde{X}}'$ 
\begin{equation}\label{eq:shifting}
 \bigl(\mu_{x+v}^V+v\bigr)\propto\mu_x^V.   
\end{equation}
Here and in the following we write $\nu\propto\nu'$ for two measures
$\nu,\nu'$ if there exists $c>0$ with $\nu=c\nu'$.

One way to characterize the leafwise measures is through the notion of subordinate $\sigma$-algebras:

\begin{definition}[Subordinate $\sigma$-algebras]\label{def: subordinate}
A $\sigma$-algebra $\cA$ of Borel subsets of $\widetilde X$ is {\em subordinate to
$V$} if $\cA$ is countably generated, for every $x \in \widetilde X$
the atom $[x]_\cA$ of $x$ with respect to $\cA$ is contained
in the leaf $x+V$, and for a.e.\ $x$
\[
 x+B_\epsilon^V\subseteq [x]_\cA\subseteq x+B_\rho^V\mbox{ for some }
\epsilon>0\mbox{ and }\rho>0.
\]
For these $x$ we define
the \emph{shape} $S_x$ of the atom $[x]_{\cA}$ (from the point of view 
of $x$) as the subset of $V$ satisfying $x+S_x=[x]_{\cA}$. 
\end{definition}

Let $\mathcal{A}$ be a countably generated $\sigma$-algebra 
on $\widetilde{X}$ that is subordinate to $V$.
Then the leafwise measures for $\mu$ with respect to $V$ 
and the conditional measures of $\mu$ 
with respect to $\cA$ satisfy that for 
a.e.\ $x\in \widetilde{X}$ the conditional measure  
$\mu_x^{\cA}$ at $x$ with respect to $\cA$ equals the normalized
push forward $x+\bigl(\mu_x^V|_{S_x}\bigr)$ of the 
restriction $\mu_x^V|_{S_x}$
under the addition map $v \mapsto x+v$ from $V \to \widetilde X$.

A slightly subtler but important feature of the system of leafwise measures is that while they may (and typically are) infinite measures, they have certain apriori restrictions on how fast they grow:
There exists a concrete functions $f_V$ 
on $V$ (only depending on $V$) that is integrable with respect to $\mu_x^V$ for every $x$ in a set of full measure (that we may as well assume already contains the conull set $\widetilde{X}'$), see \cite[Thm.~6.30]{EL-Pisa}. In fact,
$f_V$ can be chosen with very mild (polynomial-like) decay properties.
 
In particular, this implies that if $x \in {\widetilde{X}}'$ and $v \in V$ satisfies
\begin{equation}\label{eq:affineinvariance1}
\bigl(\mu_x^V+v\bigr)\propto\mu_x^V
\end{equation}
this in fact implies the formally stronger conclusion that $\mu_x^V$ is translation invariant by the same $v$, i.e.
\begin{equation}\label{eq:affineinvariance2}
\bigl(\mu_x^V+v\bigr)=\mu_x^V,
\end{equation}
for otherwise $\mu_x^V$ would have exponential growth,
which would contradict the poly\-nomial-like growth condition.
Finally we recall the following:

\begin{proposition}[{\cite[Lem.~3.2]{Einsiedler-Lindenstrauss-low-entropy}, \cite[\S3]{Lindenstrauss-Quantum}}]\label{prop: subspace}
 Let $X=G/\Gamma$ and let $W <V <G$ be closed subgroups. Let $\mu$ be a probability measure on $X$. Suppose that for $\mu$-a.e. $x$, the measure $\mu_x^V$ is supported on $W$. Then, identifying locally finite measures on $W$ with locally finite measures on $V$ supported on $W$ in the obvious way, we have that for $\mu$-a.e. $x$, \ $\mu_x^W \propto \mu_x^V$.  
\end{proposition}

\subsection{Entropy and leafwise measures}\label{sec:entropyleafwise}
Suppose now $T:G\to G$ is a group automorphism 
preserving $\Gamma$ and $V$. 
Recall that we have restricted ourselves without loss of generality to the case  where $X$ is the Pontryagin dual to   a $m$-dimensional vector space $L$ over $\QQ $. 
Fixing a choice of basis in $L$ we can restrict our attention to the case of $G=\AA^m$, $\Gamma=\QQ^m$, and hence the automorphism $T$ is defined by a rational matrix $\mathbf T \in \GL_m(\QQ)$. We note however that some of the definitions below, e.g.\ the ``sufficiently fine'' condition in~\eqref{eq: sufficiently fine}, depend on the choice of basis used to give the isomorphism $L\cong\QQ^m$ (which induces an isomorphism $\Gamma\cong\QQ^m$).

We denote the resulting automorphism of $X=G/\Gamma$ also by $T$, and consider an extension $\widetilde{T}=T\times T_\Omega:\widetilde{X}\to\widetilde{X}$ 
with $T_\Omega:\Omega\to\Omega$ measurable.
Furthermore, suppose the probability measure $\mu$
on $\widetilde{X}$ is invariant under $\widetilde{T}$. 
Then the characterizing properties of the leafwise measures $\mu_x^V$
imply the \emph{equivariance formula}
\begin{equation}\label{eq:conjugacy}
    \mu_{\widetilde{T}x}^V\propto T_*\bigl(\mu_x^V\bigr)
\end{equation}
for a.e.\ $x\in  {\widetilde{X}}'$, see e.g.~\cite[Lemma 7.16]{EL-Pisa}. 

We say that a closed subgroup $V<\AA^m$
is \emph{$S$-linear} where $S$ is a finite set of places of 
$\QQ$ (i.e.\ a set of prime numbers or infinity) if for 
each $\sigma\in S$ there is a subspace $V_\sigma<\QQ_\sigma^m$ so that 
$V$ is the direct product of the $V_\sigma$
for $\sigma \in S$. 
Below we will frequently use \emph{stable horospherical subgroup}
\[
U_T^-=\bigl\{a \in \AA^m: T^n a \to 0 \text{ as $n \to \infty$}\bigr\}
\]
for $T$ and the \emph{unstable horospherical subgroup} $U_T^+=U_{T^{-1}}^-$. 
If 
\begin{equation}\label{eq:define S}
    S = \bigl\{ p: \mathbf T \not\in \GL (m, \ZZ _ p) \bigr\} \cup \left\{ \infty \right\}
\end{equation}
then the horospherical subgroups $U_T^-$ and $U_T^+$
are $S$-linear for this~$S$.

\subsection{Increasing subordinate $\sigma$-algebras and entropy}\label{sec:increasing sigma algebra}
We continue with the notations of \S\ref{sec:entropyleafwise}. 
We say that a countably generated $\sigma$-algebra $\cA$ on $\widetilde{X}$
is {\em increasing} with respect to $\widetilde T$ if
$\cA \subseteq\widetilde T\cA$ (modulo $\mu$), i.e.\ the atom $[\widetilde T x]_{\cA}$ a.s.\ contains $\widetilde T ([x]_{\cA})$.

Let $\cP$ be a finite partition of $\widetilde X$, which we identify 
with the corresponding finite algebra of sets.
For any $\epsilon > 0$ let 
\[
 \partial ^ V _ \epsilon \cP = 
\left\{ \widetilde x \in \widetilde X: \widetilde x + B _ \epsilon ^ V 
\not \subseteq [\widetilde x] _\cP \right\}.
\]
The following lemma follows quickly from monotonicity
of the function $r\in[0,\infty)\mapsto\mu(B_r(x))$ for $x\in\widetilde{X}$
(and its almost sure differentiability) together
with compactness of $\widetilde{X}$.

\begin{lemma}[{\cite[Lemma~7.27]{EL-Pisa}}]\label{lemma:thinboundary}
For any probability measure $\widetilde \mu$ on $\widetilde X$,
there exists a finite
partition $\cP$ of $ \widetilde X$ into arbitrarily small sets such that for some fixed $C$ and  for
every $\epsilon>0$
\begin{equation}\label{eq: bound-on-boundary}
\mu\bigl(\partial _ \epsilon ^ V\cP\bigr)<C \epsilon
\end{equation}
\end{lemma}

For more details see \cite[\S7]{EL-Pisa}.
A partition $\cP$ satisfying the conclusion of the above lemma will be said to have \emph{thin boundaries}.
We will assume throughout that any finite partition $\cP $ of $X$ we will consider below is \emph{sufficiently fine} in the sense that 
\begin{equation}\label{eq: sufficiently fine}
P-P \subset \pi_X \left (\prod_ {v \in S} B _ {\QQ _ v ^ m} (r _ v) \times \prod_ {v \not \in S} \ZZ _ v ^ m \right ) \qquad\text{for every $P \in \cP$}
,\end{equation}
with $r _ v = 0.1 \max(\norm {\mathbf T}_v,\norm {\mathbf T^{-1}}_v)^{-1}$ (with respect to the operator norm on $\GL (\QQ _ v)$). 

\medskip

For any $\sigma$-algebra $\cA$ and $-\infty \leq k_0<k_1 \leq \infty$ set
\[
\cA^{k_0}=\widetilde T^{{-k_0}}\cA\mbox{ and
}\cA^{(k_0,k_1)}=\bigvee_{k_0 \leq i \leq k_  1} \widetilde T^{{-i}}\cA
\]
(for $k_0$ or $k_1 = \pm \infty$ strict inequality instead of $\leq$ should be used).

An easy Borel-Cantelli argument gives that if $\mathcal P$ is a sufficiently fine finite partition of $X$ with small boundaries in the sense of \eqref{eq:  bound-on-boundary} and \eqref{eq: sufficiently fine} then
\begin{equation*}
\mathcal{C} _ {\mathcal{P}} = \mathcal P ^{(0,\infty)}
\end{equation*}
is a countably generated $\sigma$-algebra satisfying one of the two conditions\footnote {With a bit more care, using a countable partition $\mathcal{P}$ with finite entropy, one can get that $\mathcal{P} ^ {({0,\infty})}$ is actually $U _ T ^ -$-subordinate; we achieve a similar goal by a cruder approach below.} required by Definition~\ref{def: subordinate} for $V = U _ T ^ -$, namely for a.e.~$x$  it holds that there is an $\epsilon > 0$ so that $x + B _ {U _ T ^ -} (\epsilon) \subset [x]_{\mathcal{P}}$.

A modified version of this increasing $\sigma$-algebra $\mathcal{P} ^ {(0,{\infty})}$ can be used to  construct for any  $S$-linear $T$-normalized subgroup $V<U_T^-$ of $\AA^m$ an increasing $V$ subordinate $\sigma$-algebra $\mathcal{C} _ V$ on $\widetilde{X}=X\times\Omega$
so that moreover $ \mathcal{C} _ V = \tilde T ^{-1} \mathcal{C} _ V \vee \mathcal P$. Indeed, first we construct starting from a (sufficiently fine, with small boundaries) finite partition $\mathcal P$ on $X$ a $\sigma$-algebra $\mathcal{P} _V$ on $\widetilde{X}$
as follows\label{page of P_V}: for  each $P \in \mathcal{P}$, lift it to a subset $\widetilde P \subset \AA ^ m$ contained in a translate of $\prod_ {v \in S} B _ {\QQ _ v ^ m} (r _ v) \times \prod_ {v \not \in S} \ZZ _ v ^ m$ (up to translation by an element of $\QQ ^ m$ this lift is uniquely defined).
Now take the countably generated $\sigma$-ring $\widetilde{ \mathcal{Q}} _ P$ of subsets $C$ of $\widetilde P$ with the property that if $x \in C$ then $(x +V) \cap \widetilde P \subseteq C$. Since $\pi _ X$ is a bijection from $\widetilde P$ to $P$ the image $\mathcal{Q} _ P$ of $ \widetilde {\mathcal{Q}} _ P$ in $X$ is a countably generated $\sigma$-ring, and now we define $\mathcal{P} _ V$ to be the $\sigma$-algebra of subsets of $X$ generated by the $\sigma$-rings $\mathcal{Q} _ P\times\mathcal{B}_\Omega$ for all $P \in \mathcal P$. Then
$\cC_V=\mathcal{P} _ V ^ {({0,\infty})}$ is a $T$-increasing, countably generated $\sigma$-algebra subordinate to $V$ satisfying $\cC_V= \widetilde T^{-1} \cC_V \vee \cP$. Note also that by the way $\mathcal{C} _ V$ is defined, there is a fixed $\rho > 0$ so that $[x] _ {\mathcal C _ V} \subseteq x + B _ \rho ^ V$ for \emph{all} $x \in \widetilde{X}$.
For more details the reader is referred again to \cite[\S7]{EL-Pisa}.

For $V\leq U^-_T$ as above $\cC_V$ as above, we define the {\em entropy contribution
of $V$}  to be
\begin{equation}
    \label{eq: definition of entropy contribution}
\h_{\widetilde \mu}(\widetilde T,V)=\Hh_{\widetilde \mu}(\cC_V\mid \widetilde T^{-1}\cC_V).
\end{equation}
We also need the conditional form of this definition: if $\mathcal Y$ is a $\widetilde T$-invariant $\sigma$-algebra then the \emph{entropy contribution of $V$ conditional on $\mathcal Y$} is defined to be
\begin{equation*}
h _ {\widetilde \mu} (\widetilde T, V \mid \mathcal Y) = H _ {\widetilde \mu} (\mathcal C _ V \mid {\widetilde T} ^{-1} \mathcal C _ V \vee \mathcal Y) \ptag{eq: definition of entropy contribution}
\end{equation*}
where as usual we will identify a factor $Y$ of $\widetilde X$ with the corresponding $\widetilde T$-invariant  $\sigma$-algebra $\mathcal Y$ of subsets of $\widetilde X$. Formally the conditional entropy contribution is included in the previous case, replacing $\Omega$ by $\Omega\times Y$, but notationally it will be useful to allow additional explicit conditioning. 
The following propositions shows that --- as implied by the notation --- $\h_{\widetilde \mu}(\widetilde T,V)$ does not depend on the choice of $\cP$ and $\cC_V$.

\begin{proposition}\label{prop: contribution}
Let $V\leq U^-_T$ be an $S$-linear subgroup normalized by $T$, and let
$\cC_V$ be as above. Then 
\begin{equation}\label{eq:entropygrowthrate}
    \vol(\widetilde T,V,\widetilde x)=\lim_{|N| \rightarrow \infty}\frac{1}{N}
\log{\widetilde \mu}_{\widetilde x}^V\left(T^{-N}(B_1^V(0))\right)
\end{equation}
exists a.e.\ and
$\h_{\widetilde \mu}(\widetilde T,V)=\int\vol(\widetilde T,V,\widetilde x)\operatorname{d}\!{\widetilde \mu}(\widetilde x)$. Moreover, outside a set of $\widetilde\mu$-measure zero, $\vol(\widetilde T,V,x)=0$ if and only if $\widetilde\mu_{x}^V$ is trivial.
Furthermore, $\h_{\widetilde \mu}(\widetilde T,V) \leq h_{\widetilde \mu}(\widetilde T \mid \Omega)$, with equality holding for $V=U_T^-$. In particular $h_\mu(\widetilde T\mid \Omega)>0$ if and only if $\widetilde\mu_x^{U_T^-}$ is not a.e.\ trivial. 
\end{proposition}

By the remark above, this proposition also covers the case of entropy contributions conditional on a factor. This proposition is essentially well known, and is e.g.\ heavily used 
by Ledrappier and Young in \cite{Ledrappier-Young-I,Ledrappier-Young-II} (though we are using a version of these results relative to the factor $\Omega$ of $\widetilde{X}$).
For proof we refer the reader to \cite[\S7]{EL-Pisa} where an exposition in the spirit of this paper can be found\footnote{In \cite[\S~7]{EL-Pisa} it is assumed that the acting group acts in a semisimple way on the leaves, which does not necessarily hold in our case as we are explicitly allowing actions with non-trivial Jordan form. However, the arguments of \cite[\S~7]{EL-Pisa} can be easily modified to handle this situation; we leave the details to the readers.}.

Assuming that $\mu$ is invariant and ergodic under a $\ZZ^d$-action $\widetilde{\alpha}$ 
with $\widetilde T=\widetilde \alpha^\n$ for some $\n\in\ZZ^d$ 
the value of the limit in~\eqref{eq:entropygrowthrate}
defines an invariant function for $\widetilde\alpha$, hence is a.e.\ constant, and so 
equals $h_\mu(\widetilde T, V\mid \Omega)$ for a.e.\ $x\in \widetilde{X}$. 
In particular, assuming $\widetilde\alpha$ is ergodic the following three statements are equivalent: (i) $h_\mu(\widetilde T\mid \Omega)>0$ \ (ii) $\mu_x^{U_T^-}$ is nontrivial 
a.e.\ (iii) $\mu_x^{U_T^-}$ is nontrivial on a set of positive measure.

The entropy contributions for $V < U^+_T$ (also denoted $ h _ {\widetilde \mu} (\widetilde T, V )$) are defined similarly, and satisfy that \[ h _ {\widetilde \mu} (\widetilde T, V ) =  h _ {\widetilde \mu} (\widetilde T^{-1}, V ).\]
\medskip

Let $0 \to L _ 1 \to L \to K \to 0$ be an exact sequence of finite dimensional vector space over $\QQ$ and let $0 \to Y \to X \to X _ 1 \to 0$ be the corresponding dual exact sequence of adelic solenoids. Let $T : X \to X$ be the dual map to a linear map in $\GL (L)$ fixing $L _ 1$. Then $T$ also induces a map $T _ 1: X _ 1 \to X _ 1$. Let $V _ {Y, \AA} < \AA ^ m $ be the rational subspace projecting modulo $\QQ ^ m$ to $Y$.
Let $T _ \Omega: \Omega \to \Omega$ be a continuous map on the compact metric space $\Omega$, and denote $\widetilde X = X  \times \Omega$, $\widetilde T = T \times T _ \Omega$, $\widetilde X _ 1 = X _ 1 \times \Omega$, etc. We let $\pi$ denote both the projection from $\AA^m \to \AA^m/V _ {Y, \AA}$ as well as the corresponding projection $X \to X_1$.

\begin{proposition} [cf.~{\cite[Prop 6.4]{Einsiedler-Lindenstrauss-joinings-2}} or {\cite[Prop. 3.1]{Einsiedler-Lindenstrauss-joinings}}]\label{entropy inequality proposition}
With the notations above, let $\widetilde \mu$ be a $\widetilde T$-invariant measure on $X$, let $V$ be a $S$-linear subgroup of $U _ T ^ - <\AA^m$, and let $V_1$ be a $S$-linear subgroup of $U _ {T _ 1}^{-}<\AA^m/V _ {Y, \AA} $ so that $V \leq \pi^{-1}(V_1)$. Let $\widetilde \mu _ 1 = \pi _ {*} \widetilde \mu$. Then
\begin{equation}\label{relative entropy inequality}
h _ {\widetilde \mu} (\widetilde T, V) \leq h _ {\widetilde \mu _ 1} (\widetilde T _ 1, V _ 1) + h _ {\widetilde \mu} (\widetilde T, V \cap V _ {Y, \AA})
,\end{equation}
with equality holding for $V=U_T^-$ and $V_1=U_{T _ 1} ^{-}$.
\end{proposition}

\begin{proof}
 See \cite[Prop.~6.4]{Einsiedler-Lindenstrauss-joinings-2} (while the setting is a bit different, the proof there works verbatim also in our setting).
\end{proof}

Note that by definition $h _ {\widetilde \mu} (\widetilde T, V) = h _ {\widetilde \mu} (\widetilde T, V\mid \Omega)$ and similarly for the other terms in \eqref{relative entropy inequality}.

\section{An adelic version of the positive entropy theorem}

We show in this section that it suffices to prove the following more 
special version of Theorem~\ref{thm: main}.

\begin{theorem}[Adelic theorem]\label{thm:main2}
Let $m\geq 1$, $d\geq 2$,
and let $\alpha$ be an adelic $\ZZ^d$-action on $X_m$
without virtually cyclic factors. 
We suppose furthermore
that the action satisfies that every adelic subgroup of $X_m$
that is invariant under the restriction of $\alpha$ to 
a finite index subgroup of $\ZZ^d$ is actually invariant under $\ZZ^d$.
Let $\mu$ be an $\alpha$-invariant and ergodic probability measure
on $X_m$. Then there exists an adelic subgroup $G<X_m$ so that $\mu$
is invariant under translation by elements of $G$
and $h_\mu(\alpha_{X/G}^\n)=0$ for all $\n\in\ZZ^d$. 
\end{theorem}

We note that the reader interested in the 
heart of the argument may skip most of this
section, which is dedicated to the reduction of Theorem \ref{thm: main}
to Theorem \ref{thm:main2}, and instead continue with Section \ref{sec:standing}. 

\subsection{Extension to adelic action}\label{ssec:extension}
Suppose that $\alpha$ is a $\ZZ^d$-action by automorphisms on a solenoid $X$
as in Theorem \ref{thm: main}. 
By definition this means that the Pontryagin dual $\widehat{X}$ is isomorphic
to a subgroup $V\subseteq\QQ^m$ for some $m\in\NN$. We may assume that $m$
is minimal, that $\widehat{X}=V$, and by
applying some linear automorphism if necessary, we may also assume
that the standard basis vectors of $\QQ^m$ belong to~$V$.
By Pontryagin duality we also have that $X$ is isomorphic to the quotient
of $X_m=\widehat{\QQ^m}=\AA^m/\QQ^m$ modulo the annihilator $K=V^\perp<X_m$ of $V$.

Moreover, for every $\n\in\ZZ^d$ 
the dual $\widehat{\alpha}^\n$ of the automorphism $\alpha^{\n}$ is an automorphism of $V\subseteq\QQ^m$.
Using the standard basis of $\QQ^m$ (contained in $V$ by the above assumption) 
we can represent this dual automorphism
by a rational matrix and extend it to linear automorphism of $\QQ^m$ 
(since $V\otimes_\QQ\QQ$ can be identified with $\QQ^m$).
This shows that the dual action extends to a linear 
representation of $\ZZ^d$ 
on $\QQ^m$. 

We take the transpose of this representation (i.e.\ 
of each of the matrices defining $\widehat{\alpha}^{\mathbf{e}_j}$ for $j=1,\ldots,d$)
to define an action $\widetilde{\alpha}$ of $\ZZ^d$ by automorphisms
on $\QQ^m$ and $\AA^m$. 
As discussed in Section \ref{adelicaction}
this defines an adelic action of $\ZZ^d$
by automorphisms of $X_m=\AA^m/\QQ^m$. 
Moreover, since $V\subseteq\QQ^m$ is invariant under $\widehat{\alpha}$
the annihilator $K=V^\perp$ is a closed invariant subgroup 
for $\widetilde{\alpha}$ and the 
induced action on $X_m/K$ is isomorphic to the original action on $X$.

Extending the above discussion the following lemma allows us to switch our attention
to the setting of an adelic $\ZZ^d$-action $\alpha$ on $X_m$ for some $m\in\NN$. 

\begin{lemma}\label{lemma:firstreduction}
 Suppose Theorem \ref{thm: main} holds for adelic actions,
 then it also holds for all actions of $\ZZ^d$ by automorphisms
 on solenoids.
\end{lemma}

\begin{proof}
 Let $\alpha$ be a $\ZZ^d$-action by automorphisms on a solenoid $X$
 and let $\mu$ be an $\alpha$-invariant and ergodic probability measure on $X$.
 Applying the above discussion we can construct an adelic action 
 (again denoted by $\alpha$)
 on $X_m$ for some $m\geq 1$ and an invariant compact subgroup $K$
 such that the action on $X_m/K$ is isomorphic to the original action. 

Next we can define a probability measure $\mu_K$ on $X_m$ that is invariant
under $K$ and modulo $K$ equals $\mu$.
This describes $\mu_K$ uniquely. By invariance of $K$ under $\alpha$
and uniqueness this measure is also $\alpha$-invariant.
If it is not ergodic with respect to $\alpha$ we may consider
an ergodic component $\tilde{\mu}$ of $\mu_K$. 
Due to ergodicity of $\mu$ almost surely the ergodic components
will project to $\mu$. Let $\tilde\mu$ be one such ergodic component.

By our assumption (in Lemma~\ref{lemma:firstreduction}) 
we know that Theorem~\ref{thm: main} already holds for $\tilde\mu$.
In other words there exists a finite index subgroup $\Lambda<\ZZ^d$
and a decomposition $\tilde\mu=\frac1J(\tilde\mu_1+\cdots+\tilde\mu_J)$
of $\tilde\mu$ into $\alpha_\Lambda$-invariant and ergodic probability measures,
and there exist closed $\alpha_\Lambda$-invariant subgroup $G_j<X_m$
so that $\tilde\mu_j$ is invariant under translation by elements of $G_j$
for $j=1,\ldots,J$. Moreover, the entropy 
of $\alpha^\n_{X_m/\tilde G_j}$ with $\n\in\Lambda$
with respect to $\tilde\mu_j$ vanishes for all $j=1,\ldots,J$. 
Taking the quotient of $X_m$ by the invariant subgroup $K$ all of these
statements become the corresponding statements for the push forwards $\mu_j$
of $\tilde\mu_j$ under the quotient map $X_m\rightarrow X_m/K\cong X$. 
We also note that the $\alpha_\Lambda$-ergodic components of $\mu$
are either equal or singular to each other. Hence, if $\mu_j$
equals $\mu_k$ for some $j\neq k$ we may simply collect equal terms
and would again obtain a decomposition into mutually singular measures (of necessary
equal weight due to ergodicity with respect to $\alpha$). 
Together we obtain the conclusions of Theorem \ref{thm: main}
for the original measure $\mu$. This gives the lemma.
\end{proof}

\subsection{Choosing a good finite index subgroup $\Lambda<\ZZ^d$}

\begin{lemma}\label{irredequaltot} 
  Let $m,d\geq 1$ and let $\alpha$ be an adelic $\ZZ^d$-action on $X_m$.
 Then there exists a finite index subgroup $\Lambda<\ZZ^d$ with the 
 following property: Applying Lemma~\ref{factorlist} to the restriction
 $\alpha_\Lambda$ of $\alpha$ to $\Lambda$ we obtain finitely many
 $\AA$-irreducible adelic actions. Then each one of them remains $\AA$-irreducible
 if we restrict $\alpha_\Lambda$ further to a finite index 
 subgroup $\Lambda'<\Lambda$. 
 Moreover, $\Lambda$ may be chosen so that an adelic subgroup $Y<X$ is invariant under $\alpha_\Lambda$
 if and only if it is invariant under $\alpha_{\Lambda'}$ for a finite
 index subgroup $\Lambda'<\Lambda$. 
\end{lemma}

For the proof of the lemma and some of the following arguments we first recall
the Jordan decomposition. Given a matrix $A\in\GL_m(\QQ)$ there exist matrices $D,U\in\GL_m(\QQ)$ so that $A=DU=UD$, 
$D$ is semisimple (i.e.\ is diagonalizable over $\overline\QQ$),
and $U$ is unipotent (i.e.\ has only $1$ as eigenvalue). This decomposition is
unique and if $A$ commutes with a matrix $B$, then $D$ and $U$ as above also commute with $B$.
Moreover, a subspace $V<\QQ$ is invariant under $A$ if and only if $V$ is invariant under both $D$ and $U$.

\label{page of Jordan}
By applying this to each of the matrices $\alpha^{\mathbf{e}_i}$ for $i=1,\ldots,d$
we obtain two representations of $\ZZ^d$ on $\QQ^m$: The first representation
$\alpha_{\diag}$ is semisimple, and the 
second $\alpha_{\uni}$ is by unipotent matrices,
the two representations commute, and we have
$\alpha^\n=\alpha_{\diag}^\n\alpha_{\uni}^\n$
for all $\n\in\ZZ^d$.

\begin{proof}[Proof of Lemma \ref{irredequaltot}]
 We first consider an $\AA$-irreducible action $\alpha$. 
 As the proof of Lemma \ref{adelicirred} shows $\alpha$ 
 corresponds in this case to a global field $\KK$ 
 generated by $d$ elements $\zeta_1,\ldots,\zeta_d$
 (obtained directly from the matrix representations of $\widetilde\alpha^{\mathbf{e}_j}$).
 Restricting the action to a finite index subgroup results in replacing $\zeta_1,\ldots,\zeta_d$ by $d$ monomial expressions $\xi_1,\ldots,\xi_d$ 
 (corresponding to a basis of $\Lambda<\ZZ^d$)
 in the numbers $\zeta_1,\ldots,\zeta_d$. This in turn
 may result in $\xi_1,\ldots,\xi_d$ generating instead of $\KK$ 
 a subfield $\LL$ of $\KK$. 
 In this case the  $\AA$-irreducible representation of $\ZZ^d$ on $\KK$ obtained in the 
 proof of Lemma \ref{adelicirred} becomes, when restricted to $\Lambda$, isomorphic
 to a direct sum of $[\KK:\LL]$ many copies of 
 the  $\AA$-irreducible representation defined by multiplication
 by $\xi_1,\ldots,\xi_d$ on $\LL$. If this indeed happens we 
 may choose $\Lambda$ so that $\LL$
 is minimal in dimension. Hence for any finite index subgroup $\Lambda'<\Lambda$
 the monomial expressions in the variables $\xi_1,\ldots,\xi_d$
 corresponding to a basis of $\Lambda'$ will still generate
 the same field $\LL$ and so the $[\KK:\LL]$ many  $\AA$-irreducible representations
 for the restriction to $\Lambda$ will remain irreducible for the 
 restriction to $\Lambda'$. This proves the first part
 of the lemma in the $\AA$-irreducible case.
 
 Let now $\alpha$ be a general adelic action. 
Let 
 \begin{equation*}
  V_0=\{0\}<V_1<V_2<\cdots<V_r=\QQ^m
 \end{equation*}
 be as in Lemma~\ref{factorlist}, equation \eqref{eq:partialflag}, with the action induced by $\widehat \alpha$ on $V_i/V_{i-1}$ irreducible over $\QQ$,  and let $\KK_i$ be the corresponding finite extension of $\QQ$ as in Proposition~\ref{adelicirred} for $V_i/V_{i-1}$ (or more precisely, for the dual adelic solenoid) for $i=1,\ldots,r$.
 Applying the above discussion on each irreducible quotient by passing to a finite index subgroup $\Lambda < \ZZ^d$ we may assume that these quotients remain irreducible even if we pass to a further finite index subgroup of $\Lambda$,  establishing the first part
 of the lemma. 
 
 Now let $M$ be the least common multiple of the orders of 
 the (finitely many) roots of unity in some finite degree Galois extension of $\QQ$ containing the fields 
 $\KK_1,\ldots,\KK_r$.  We claim that if $\Lambda_1 = M \Lambda$ 
 and if $V < \QQ ^ m $ is invariant under $\widehat{\alpha}(\Lambda')$ 
 for some $\Lambda' < \Lambda_1$ then $V$ is invariant 
 under $\widehat{\alpha} (\Lambda _ 1)$. 

Indeed, as discussed above, for any $\Lambda ' \leq \ZZ ^ d$, the space $V=Y^\perp$ is $\widehat{\alpha} (\Lambda ')$-invariant if and only if it is invariant under both $\widehat{\alpha}_{\diag} (\Lambda ') $ and $\widehat{\alpha}_{\uni} (\Lambda ')$.

Since the map $\mathbf n \mapsto \widehat{\alpha}_{\uni} (\mathbf n)$ is polynomial, and since any finite index subgroup of $\ZZ ^ d$ is Zariski dense in affine $d$-dimensional space, a space $V$ is $\widehat{\alpha}_{\uni} (\Lambda ')$-invariant if and only if it is $\widehat{\alpha}_{\uni} (\ZZ ^ d)$-invariant (hence in particular $\widehat{\alpha}_{\uni} (\Lambda _1)$-invariant).

Suppose $V$ is  $\widehat{\alpha}_{\diag} (\Lambda ') $  invariant but not $\widehat{\alpha}_{\diag} (\Lambda _1) $ invariant. Let $\mathbf n = M \mathbf n ' \in \Lambda _ 1$ so that $\widehat{\alpha}_{\diag} (\mathbf n)$ does not fix $V$. Since $\widehat{\alpha}_{\diag} (\Lambda ') $ leaves $V$ invariant it follows that there is some $k  \in \NN$ so that $\widehat{\alpha}_{\diag} (k\mathbf n) $ leaves $V$ invariant.

By Proposition~\ref{adelicirred}, for every $i$ the action induced by $\alpha (\mathbf n ')$ on 
each $V_i/V_{i-1}$ can be identified with multiplication by some $\xi  \in \KK_i$ on $\KK_i$.
If $\widehat{\alpha}_{\diag} (\mathbf n)=\widehat{\alpha}_{\diag} (\mathbf n')^M$ does not fix $V$ but $\widehat{\alpha}_{\diag} (k \mathbf n)$ does this implies that there is an $j$ and $\xi'\in\KK_j$ as well as embeddings $\sigma,\sigma'$ of $\KK_i$ and $\KK_j$ respectively into $\CC$ so that $\sigma(\xi)^ M \neq \sigma'(\xi') ^ M$ but $\sigma(\xi) ^ {k M} = \sigma (\xi') ^ {k M}$. Then $\sigma(\xi) \sigma' (\xi')  ^{-1}$ is an element of the compositum of $\sigma(\KK_i)$, $\sigma'(\KK_j)$ that is a root of unity of order not dividing $M$ --- a contradiction.
\end{proof}

In particular Lemma \ref{irredequaltot} shows that it is possible
to restrict any adelic action to a finite index subgroup so that each 
of the $\AA$-irreducible adelic actions associated to its restriction 
are in fact totally $\AA$-irreducible.

\subsection{Reduction to Theorem \ref{thm:main2}}

Using the above preparations we are now ready to explain the following reduction
step.

\begin{proof}[Proof of Theorem \ref{thm: main} assuming Theorem \ref{thm:main2}]
 By Lemma \ref{lemma:firstreduction} it suffices to consider adelic actions
 for the proof of Theorem \ref{thm: main}. So let $d\geq 2$,
 let $\alpha$ be an adelic $\ZZ^d$-action without virtually cyclic factors,
 and let $\mu$ be an $\alpha$-invariant and ergodic probability measure. 
 
 By Lemma \ref{irredequaltot} there exists a finite index subgroup $\Lambda<\ZZ^d$
 so that the restriction of $\alpha$ to $\Lambda$ satisfies the assumptions to
 Theorem \ref{thm:main2}. Note however, that the measure $\mu$ might not be 
 ergodic with respect to $\alpha_\Lambda$. Hence we may have to apply the 
 ergodic decomposition. Since $\Lambda$ has finite index in $\ZZ^d$,
 this ergodic decomposition simply takes the form 
 \[
  \mu=\frac1J\left(\mu_1+\cdots+\mu_J\right),
 \]
 where the probability measures $\mu_j$ are $\alpha_\Lambda$-invariant 
 and ergodic for $j=1,\ldots,J$. Since $\mu$ is invariant under the full action $\alpha$,
 we also have that for every $\n\in\ZZ^d$ and every index $j\in\{1,\ldots,J\}$
 there exists an index $k$ with $\alpha^\n_*\mu_j=\mu_k$. Since ergodic measure
 are either equal or singular to each other we may also assume that the measure $\mu_1,\ldots,\mu_J$
 are mutually singular to each other. By ergodicity of $\mu$ with respect to $\alpha$
 we also have that for every pair of indices $j,k$ there exists some $\n\in\ZZ^d$
 so that $\alpha^\n\mu_j=\mu_k$. 
 
 We now apply Theorem \ref{thm:main2} to $\mu_1$ and the restriction $\alpha_\Lambda$. Therefore
 there exists a closed subgroup $G_1<X_m$ so that $\mu_1$ is invariant under translation by 
 elements in $G_1$ and for any $\n\in\Lambda$ we have $h_{\mu_1}(\alpha^\n_{X_m/G_1})=0$. 
 Applying the above transitivity claim we obtain the theorem.
\end{proof}

\subsection{Standing assumptions for the proof of Theorem \ref{thm:main2}}\label{sec:standing}
Since we have shown that Theorem \ref{thm:main2} implies Theorem \ref{thm: main} our aim for the
the next sections is to show the former. Hence we will assume that $\alpha$ is an adelic $\ZZ^d$-action
on $X_m$ satisfying the assumptions of Theorem \ref{thm:main2}. Furthermore we assume
that $\mu$ is an $\alpha$-invariant and ergodic probability measure. 

Suppose that $\mu$ is translation invariant under an adelic subgroup $Y<X_m$,
then by invariance of $\mu$ under $\alpha$ it is also translation invariant
under $\alpha^\n(Y)$. Taking the closed subgroup generated by these, it follows
that there exists a maximal adelic subgroup $Y<X_m$ so that $\mu$ is invariant
under translation by elements of $Y$ and moreover that $Y$ is $\alpha$-invariant.
We may replace $X_m$
by $X_m/Y$ for this maximal $\alpha$-invariant adelic subgroup $Y$
and consider the push forward of $\mu$ under the canonical 
projection map $X_m\rightarrow X_m/Y$. If this new measure has zero entropy,
Theorem~\ref{thm:main2} already holds for this action. Hence we may and will assume
for the proof of Theorem \ref{thm:main2} that
\begin{enumerate}
    \item[(i)] $\mu$ is not invariant under any adelic subgroup, and
    \item[(ii)] there exists some $\n\in\ZZ^d$ so that $h_\mu(\alpha^\n)>0$.
\end{enumerate}
Our aim is to derive a contradiction from these assumptions.

To be able to apply the method introduced in \cite{EL-ERA03} (relying on 
arithmetic properties of $\AA$-irreducible actions) we start by applying Lemma \ref{factorlist}
to the adelic action $\alpha$ on $X_m$. We now think of the chain of invariant subgroups
as defining a chain of factor maps 
\begin{align*}
 X=X_{(0)}=X_m&\rightarrow X_{(1)}=X/Y_1\rightarrow\cdots\\
 &\rightarrow X_{(j)}=X/Y_{j}\rightarrow X_{(j+1)}=X/Y_{j+1}\rightarrow\cdots
 \rightarrow X_{(r)}=\{0\}. 
\end{align*}
Applying the Rokhlin entropy addition formula inductively to these factors
we have
\[
 h_\mu(\alpha^\n)=\sum_{j=0}^{r-1}h_\mu(\alpha^\n_{X_{(j)}}\mid X_{(j+1)})
\]
for all $\n\in\ZZ^d$,
where we write $\mu$ for the invariant measure on all of the factors,
write $\alpha^\n_{X_{(j)}}$ for the induced action on the factor $X_{(j)}$,
and write $h_\mu(\alpha^\n_{X_{(j)}}\mid X_{(j+1)})$ for the conditional 
entropy of $\alpha^\n_{X_{(j)}}$ conditioned on the next factor $X_{(j+1)}$. 

By assumption (ii) above there exists some $\n\in\ZZ^d$
so that $h_\mu(\alpha^\n)>0$. Hence we may choose the minimal $s\in\{0,\ldots,r-1\}$
so that $h_\mu(\alpha^\n_{X_{(s)}}\mid X_{(s+1)})>0$ for some $\n\in\ZZ^d$. 
We define $Y_\base=Y_{s+1}$, 
$X_{\base}=X_{(s+1)}$, and will always consider conditional entropy
over the factor $X_{\base}$. 
We also define $Y_\posfact=Y_s$ (contained in $Y_\base$) and $X_\posfact=X/Y_\posfact$
(which factors on $X_\base$).
In this sense the factor $X_{\posfact}$
will be important for us as it is a `positive entropy extension' of $X_{\base}$
and at the same time an `adelic extension with $\AA$-irreducible fibers'. 
In fact by choice of $s$ there exists some $\n\in\ZZ^d$ so 
that $h_\mu(\alpha^\n_{X_{\posfact}}\mid X_{\base})>0$,   we 
have $X_{\posfact}=X/Y_{\posfact}$, $X_{\base}=X/Y_\base$, and
that the action induced on the fibers $Y_{\irred}=Y_\base/Y_{\posfact}$
is (totally) $\AA$-irreducible.
Finally as we have choosen $s$ minimally
the original system $X$ is a zero entropy extension of $X_{\posfact}$
in the sense that $\h_\mu(\alpha^\n\mid X_{\posfact})=0$ for all $\n\in\ZZ^d$.

To summarize the factor maps
\begin{equation}\label{x base equation}
X\rightarrow X_{\posfact}\rightarrow X_{\base} =X_{\posfact}/Y_{\irred}
\end{equation}
describe the original action as a zero-entropy extension
of $X_{\posfact}=X/Y_\posfact$, and $X_{\base}=X/Y_\base$ as the quotient of $X_{\posfact}$ 
by an $\alpha$-invariant $\AA$-irreducible adelic subgroup $Y_{\irred} < X_{\posfact}$ 
so that $X_{\posfact}$ is a positive entropy extension
over $X_{\base}$. More precisely, 
\[
 h_\mu(\alpha^{\mathbf n} \mid X_{\posfact})=0 \qquad\text{for all $\n\in\ZZ^d$}
\]
 but 
 \[
 h_\mu(\alpha^{\mathbf n}_{X _{\posfact}}\mid X_{\base})>0 \qquad\text{for some $\n \in \ZZ^d$.}
\]

To the exact sequence $0 \to Y _{\irred} \to X _{\posfact} \to X _{\base}  \to 0$ there is attached an exact sequence of $\QQ$-vector spaces $0 \to L _{\base}\to L _{\posfact} \to L _{\irred}  \to 0$ where 
$L_{\base}=\widehat{X_{\base}}$, $L_{\posfact}=\widehat{X_{\posfact}}$, and $L _{\irred}=\widehat{Y_{\irred}}$ can be identified as $\QQ$-vector space with a number field $\KK$.
Viewing $X _{\posfact}$ as a quotient of $\AA ^{m_1}$ amounts to choosing 
a $\QQ$-basis $v _ 1, \dots, v _ {m_1}$ for the vector space $L_\posfact^*$.

Identifying $\KK$ as a $\QQ$-vector space with its dual using the trace form, we have an embedding $\KK \to L _{\posfact} ^{*}$, and we will always choose our $\QQ$-basis so that $v _ 1, \dots, v _ {[\KK: \QQ]}$ will be a basis for $\KK < L ^{*} _{\posfact}$.

\section{A bound on the entropy contribution}\label{sec: bound}

In this section we prove the following theorem.

\begin{theorem}[cf.~{\cite[Thm.~4.1]{EL-ERA03}}]\label{thm:bound-contribution} Let $m, d \geq 1$, and let $\alpha$ be a $\ZZ ^ d$-action on an adelic solenoid $X = \AA ^ m / \QQ ^ m$. We also let $\alpha$ denote the corresponding action on~$\AA ^ m$. Let $Y _{\irred} \leq X$ be an $\alpha$-invariant $\AA$-irreducible adelic subspace, and set $X _{\base} = X / Y _{\irred}$. Let $\alpha _ \Omega$ be a $\ZZ ^ d$-action on a compact metric space $\Omega$, $\widetilde \alpha = \alpha \times \alpha _ \Omega$, and let $\widetilde \mu$ be a $\widetilde \alpha$-invariant measure on $\widetilde X = X \times \Omega$. Fix $\mathbf n \in \ZZ ^ d$, and let $V < U ^ - _ {\mathbf n} < \AA ^ m$ be a closed $\alpha$-invariant subspace. Let $V _ {\irred, \AA} < \AA ^ m$ be the rational subspace projecting modulo $\QQ ^ m$ to $Y _{\irred}$.
Then
\begin{equation}\label{eq: entropy inequality}
h _ {\widetilde \mu} (\widetilde \alpha ^ {\mathbf n}, V \mid X _{\base} \times \Omega) \leq \frac {h _ \lambda (\alpha _{\irred} ^ {\mathbf n}, V \cap V _{\irred,\AA}) }{ h _ \lambda (\alpha ^ {\mathbf n} _{\irred})} \cdot h _ {\widetilde \mu} (\widetilde \alpha ^ {\mathbf n} \mid X _ \base \times \Omega)
.\end{equation}
\end{theorem}

Notice that this estimate is sharp for a product measure ${\widetilde \mu}= \lambda \times \nu$ with $\lambda$ being the Haar measure on $X$.
Our treatment here follows closely in content (if not in notation) that of \cite[\S4]{EL-ERA03}.
A special case of this theorem appeared in \cite [Thm.~2.4]{ Lindenstrauss-p-adic}.

\begin{proof}
 We first note that
\begin{equation*}
h _ {\widetilde \mu} (\widetilde \alpha ^ \mathbf n, V \cap V _ {\irred, \AA} \mid X _{\base} \times \Omega) = h _ {\widetilde \mu} (\widetilde \alpha ^ \mathbf n, V \mid X _{\base} \times \Omega)
.\end{equation*}
Indeed, if $\pi: X \to X _{\base}$ is the natural projection, then for any $x \in X$ we have that
\begin{equation*}
(x + V) \cap \pi ^{-1} \circ \pi (x) = x + (V \cap V _ {\irred, \AA} )
,\end{equation*}
hence if $\mathcal{C} _ V$ is a decreasing $V$-subordinate $\sigma$-algebra of subsets of $X$, then $\mathcal{C} _ V \vee \mathcal{B} _{\base}$ is a decreasing $\sigma$-algebra subordinate to $V \cap V _ {\irred, \AA}$ (with $\mathcal{B} _{\base}$ denoting the $\sigma$ algebra of Borel subsets of $X _{\base}$, or more precisely the image under $\pi ^{-1}$ of this $\sigma$-algebra in $X$).
Therefore applying \peqref{eq: definition of entropy contribution} twice, once for the $V$-subordinate $\sigma$-algebra $\mathcal{C} _ V$ and once for the $V \cap V _ {\irred, \AA}$-subordinate $\sigma$-algebra $\mathcal{C} _ V \vee \mathcal B _{\base}$ we get
\begin{align*}
h _ {\widetilde \mu} (\widetilde \alpha ^ \n, V \mid X _{\base} \times \Omega)& = H _ {\widetilde \mu} (\mathcal{C} _ V \mid \widetilde \alpha ^ {- \n} \mathcal{C} _ V \vee \mathcal B _ {{\base}} \vee \mathcal B _ {\Omega}) \\
& = H _ {\widetilde \mu} (\mathcal{C} _ V \vee \mathcal{B} _ {{\base}}  \mid \widetilde \alpha ^ {- \n} \mathcal{C} _ V \vee \mathcal B _ {\base} \vee \mathcal B _ {\Omega}) \\
& = h _ {\widetilde \mu} (\widetilde \alpha ^ \n, V \cap V _ {\irred, \AA} \mid X _{\base} \times \Omega)
.\end{align*}
Since the right hand side of \eqref{eq: entropy inequality} depends only on $V \cap V _ {\irred, \AA}$, we may (and will, for the remainder of the proof) assume that $V \leq U_\bn^- \cap  V _ {\irred, \AA}$. We may also assume $V$ is a proper subgroup of $U_\bn^- \cap  V _ {\irred, \AA}$ since for $V=U_\bn^- \cap  V _ {\irred, \AA}$ by first applying the above discussion and then using the second part of Proposition~\ref{prop: contribution} twice
\begin{gather*}
  h _ {\widetilde \mu} (\widetilde \alpha ^ \n,U_\bn^- \cap  V _ {\irred, \AA}\mid X _{\base} \times \Omega) = h _ {\widetilde \mu} (\widetilde \alpha ^ \n,U_\bn^- \mid X _{\base} \times \Omega) = h _ {\widetilde \mu} (\widetilde \alpha ^ {\mathbf n} \mid X _ \base \times \Omega),   \\
 {h _ \lambda (\alpha _{\irred} ^ {\mathbf n}, U_\bn^- \cap V _{\irred,\AA}) }={ h _ \lambda (\alpha ^ {\mathbf n} _{\irred})}   ,
\end{gather*}
establishing \eqref{eq: entropy inequality} in this case.

\medskip

Let the rank of $V _ {\irred, \AA}$ as a free $\AA$-module be $k$. Since $Y _{\irred}$ is $\alpha$-invariant and $\AA$-irreducible,
by Proposition~\ref{adelicirred}, there is a global field $ \KK$ with $[\KK: \QQ] = k$, an injective hopmomorphism of $\QQ$-vector spaces $\phi:\KK \to \QQ ^ m$, and $d$ nonzero elements $\zeta _ 1, \dots, \zeta _ d \in  \KK ^ \times$  so that $\AA \otimes \phi (\KK) = V _ {\irred, \AA}$ and so that for any $\n = (n _ 1, \dots, n _ d)$ and $\xi \in \AA _ \KK = \AA \otimes \KK$
\begin{equation*}
\alpha ^ \mathbf n .\phi _ \AA (\xi) = \phi _ \AA (\zeta_\n \xi) \qquad \zeta_\n:=\zeta _ 1 ^ {n _ 1} \dots \zeta _ d ^ {n _ d}
\end{equation*}
with $\phi _ \AA$ the isomorphism of $\AA$-modules $ \AA_ \KK   \to V _ {\irred, \AA}$ induced from $\phi$.

Fix $\n \in \ZZ ^ d$.  Then $U ^ - _ {\n} $ is $S$-linear for a finite set $S$ of places of $\QQ$ (including $\infty$) as in \eqref{eq:define S}. Let $S_\KK$ be the (finite) set of places of $\KK$ lying over the places $S$ of $\QQ$. The $\AA$-irreducibility of the action of $\alpha$ on $Y _{\irred} \cong \AA _ \KK / \KK$ implies that every $\alpha$-invariant subspace $V \leq V _ {\irred, \AA} \cap U ^ - _ {\n}$ has the form
\begin{equation}\label{V and Greek phi}
V = \phi _ \AA \left (\prod_ {\sigma\in S ' _ \KK} \KK _ \sigma \right)
\end{equation}
with $S ' _ \KK \subseteq S_\KK$.
Similarly,
\[  V _ {\irred, \AA} \cap U ^ - _ {\n} = \phi _ \AA  \left (\prod_ {\sigma\in S ^- _ \KK} \KK _ \sigma \right)
\end{equation*}
with $S ^ - _ \KK \subseteq S_\KK$ a finite set of places of $\KK$ with $S'_\KK \subset S ^- _ \KK$.
It follows from the relation between entropy contribution and leafwise measures in Proposition~\ref{prop: contribution} that
\begin{align*}
h _ \lambda (\alpha _{\irred} ^ {\mathbf n}, V) & = \sum_ {\sigma \in S ' _ \KK} \delta _ \sigma  \log 1/\absolute {\zeta _ \n} _ \sigma \\
h _ \lambda (\alpha ^ {\mathbf n} _{\irred})& = h _ \lambda (\alpha _{\irred} ^ {\mathbf n}, U_\n^- \cap V _ {\irred, \AA}) = \sum _{\sigma \in S ^ - _ \KK} \delta _ \sigma  \log 1/\absolute {\zeta _ \n} _ \sigma;
\end{align*}
note that by definition of $U_\n^-$ and $\zeta _ \n$
we have that $\absolute {\zeta _ \n} _ \sigma < 1$ for every $\sigma \in S ^ - _ \KK$. Let
\[
\kappa=
\frac {\sum_ {\sigma \in S '_\KK} \delta _ \sigma  \log \absolute {\zeta_\n} _ \sigma}{ \sum_ {\sigma \in S_{\KK}^-} \delta _ \sigma \log \absolute {\zeta_\n} _ \sigma} = \frac {h _ \lambda (\alpha _{\irred} ^ {\mathbf n}, V ) }{ h _ \lambda (\alpha ^ {\mathbf n} _{\irred})} <1.
\]
Let $\cP$ be a sufficiently fine finite partition with small boundaries of $X$
as in \eqref{eq: bound-on-boundary} and \eqref{eq: sufficiently fine} (indeed, we will take it to be even finer), let $\cP_V$ be a corresponding $\sigma$-algebra of subsets of $X \times\Omega$  as in p.~\pageref{page of P_V}, and let $\cC_V = \bigcup_{i\geq 0} \widetilde\alpha^{-i\n}\cP_V$ be as in
\S\ref{sec:increasing sigma algebra}. Then $\cC_V$ and $\cP_V$ are both subordinate to $V$, and $\cC_V$ is in addition decreasing with respect to $\widetilde\alpha^{\n}$. Moreover,
\begin{equation}\label{decreasing equation}
\mathcal{C} _ V = \widetilde \alpha ^ {-\n} \mathcal{C} _ V \vee \mathcal{P} = \widetilde \alpha ^ {-\n} \mathcal{C} _ V \vee \mathcal{P} _ V
.\end{equation}
Setting $\widetilde T = \widetilde \alpha ^ \mathbf n$, it follows from \eqref{decreasing equation} that for any $j\in \NN$, \[\mathcal{C} _ V = \mathcal{P} ^ {(0, j-1)} \vee \mathcal{C} _ V  ^ j,\] where we recall that $\mathcal{C} _ V  ^ j = \widetilde T ^ {- j} \mathcal{C} _ V$ and $\mathcal{P} ^ {(0, j -1)} = \bigvee_ {0 \leq i \leq j -1} \widetilde T ^ {- i} \mathcal{P}$.
The key point in the proof of Theorem~\ref{thm:bound-contribution} is that \emph{for $\ell =\lceil \kappa j \rceil$, the atoms of $\mathcal{P} ^ {(0, \ell)} \vee \mathcal{C} _ V ^ j$ are already very close to being equal to the atoms of $\mathcal{C} _ V = \mathcal{P} ^ {(0, j-1)} \vee \mathcal{C} _ V ^ j$} (recall that $\kappa<1$). In some simple cases (e.g.\ that considered in \cite{Lindenstrauss-p-adic}) these $\sigma$-algebras literally coincide, though in general there may be a small disparity. What we now proceed to show (cf.~\cite[Lem.~4.2]{EL-ERA03}) is that there is a set $X_j$ with $\mu (X \setminus X_j ) \leq \exp (- c j)$ for appropriate $c > 0$ so that
\begin{equation}\label{main point equation}
[x]_{\mathcal{P} ^ {(0, \ell)}} \cap [x]_{\mathcal{C} _ V ^ j} = [x] _ {\mathcal{C} _ V} \qquad \text{for any $x \in X _ j$}
.\end{equation}

Since $\mathcal{C} _ V ^ j$ is $V$-subordinate for $V$ as in \eqref{V and Greek phi}, there is some $B \subset \prod_ {\sigma \in S ' _ \KK} \KK _ \sigma$ (depending on $x$ and $j$) so that
\begin{equation*}
[x] _ {\mathcal{C} _ V ^ j} = x + \phi _ \AA (B);
\end{equation*}
moreover for any $\eta > 0$, by choosing $\mathcal{P}$ sufficiently fine depending on $\eta$, one can ensure that
\begin{equation*}
B \subset \prod_ {\sigma \in S ' _ \KK} \left\{ t  \in \KK _ \sigma: \absolute {t} _ \sigma \leq \eta \absolute {\zeta _ \n}_\sigma ^ {-j} \right\}
.\end{equation*}
Notice that $[x] _ {\mathcal{P} ^ {(0, \ell)} \vee \mathcal B _{\base}} \subset x + V _ {\irred, \AA}$ hence there is some $D \subset \AA _ \KK$ so that
\begin{equation*}
[x] _ {\mathcal{P} ^ {(0, \ell)} \vee \mathcal B _{\base}} = x + \phi _ \AA (D)
.\end{equation*}
If the partition $\mathcal{P}$ was chosen to be sufficiently fine (again, depending on the parameter $\eta > 0$ introduced above) we may assume that
\begin{equation*}
D \subset \prod_ {\sigma \in S _ \KK} \left\{ t \in \KK _ \sigma: \absolute t _ \sigma \leq \eta \min(1,\absolute {\zeta _  \n} _ \sigma ^ {- \ell}) \right\} \times \prod_ {\sigma \not \in S _ \KK} \mathcal{O} _ {\KK, \sigma}
.\end{equation*}
Note that $\absolute{\zeta _  \n} _ \sigma<1$ for every $\sigma\in S'_\KK$ but may be $<1$, $=1$ or $>1$ for $\sigma\in S_\KK$; however by choice of $S$ (hence of $S_\KK$) \ $\absolute{\zeta _  \n} _ \sigma=1$ for $\sigma\not\in S_\KK$.
Since $ \QQ ^ m \cap V _ {\irred, \AA} = \phi _ \AA (\KK) $, and using the fact that $\mathcal B _{\base} \subseteq \mathcal{C} _ V^j$ (modulo $\widetilde \mu$), we see that
\begin{align*}
    [x]_{\mathcal{P} ^ {(0, \ell)}} \cap [x]_{\mathcal{C} _ V ^ j} &= [x]_{\mathcal{P} ^ {(0, \ell)}\vee\cB_\base} \cap [x]_{\mathcal{C} _ V ^ j} \\
    &= x + \phi_\AA \left(\bigcup_{\xi\in\KK}(B\cap(D+\xi))\right).
\end{align*}
However, if $B \cap (D + \xi) \neq \emptyset$, i.e.~$\xi \in B - D$ then
\begin{equation*}
\absolute {\xi} _ \sigma \leq\begin{cases}  2 \eta \absolute {\zeta _ \n} _ \sigma ^ {- j}& \text{if $\sigma \in S ' _ \KK$} \\
2 \eta & \text{if $\sigma \in S ^-_ \KK \setminus S ' _ \KK$} \\
2 \eta  \absolute {\zeta _ \n} _ \sigma ^ {- \ell} & \text{if $\sigma \in S _ \KK \setminus S ^- _ \KK$} \\
1& \text{otherwise.}
\end{cases}
\end{equation*}
By Proposition~\ref{prop: localprod}, if moreover $\xi \in \KK ^ \times$ and $\eta$ was chosen small enough ($<1/2$),
\begin{equation}\label{big product equation}
1 = \prod_ {\text{places $\sigma$ of $\KK$}} \absolute \xi _ \sigma ^{\delta(\sigma)}
< \prod_ {\sigma \in S ' _ \KK} \absolute {\zeta _ \n} _ \sigma ^ {-\delta(\sigma)j} \times \!\!\!\prod_ {\sigma \in S _ \KK \setminus S ^ - _ \KK} \absolute {\zeta _ \n} _ \sigma ^ {-\delta(\sigma)\ell}
.\end{equation}
Applying Proposition~\ref{prop: localprod} to $\zeta_\n$,
\[
1=\prod_ {\text{places $\sigma$ of $\KK$}} \absolute{ \zeta_\n} _ \sigma ^{\delta(\sigma)} = \prod_ {\sigma \in S ^ - _ \KK} \absolute{ \zeta_\n} _ \sigma ^{\delta(\sigma)} \times \prod_ {\sigma \in S _ \KK \setminus S ^ - _ \KK} \absolute{ \zeta_\n} _ \sigma ^{\delta(\sigma)}
\]
hence $\prod_ {\sigma \in S ^ - _ \KK} \absolute{ \zeta_\n} _ \sigma ^{\delta(\sigma)} = \prod_ {\sigma \in S _ \KK \setminus S ^ - _ \KK} \absolute{ \zeta_\n} _ \sigma ^{-\delta(\sigma)}$.
Thus \eqref{big product equation} implies
\begin{equation*}
0 < j \left (\sum_ {\sigma \in S ' _ \KK} \delta (\sigma) \log 1/\absolute {\zeta _ \n} _ \sigma \right) - \ell \left (\sum_ {\sigma \in S ^- _ \KK} \delta (\sigma) \log 1/\absolute {\zeta _ \n} _ \sigma \right)
.\end{equation*}
But this contradicts the definition of $\kappa$ and $\ell \geq \kappa j$.

\medskip

Thus we obtain the important conclusion that if $\ell =\lceil \kappa j \rceil$,
\begin{equation*}
[x]_{\mathcal{P} ^ {(0, \ell)}} \cap [x]_{\mathcal{C} _ V ^ j} = x + \phi _ \AA (B ')
\end{equation*}
where
\begin{equation*}
B ' = B \cap D \subset  \prod_ {\sigma \in S ' _ \KK} \left\{ t  \in \KK _ \sigma: \absolute {t} _ \sigma \leq \eta \right\}
.\end{equation*}
It follows that \[[x]_{\mathcal{P} ^ {(0, \ell)}} \cap [x]_{\mathcal{C} _ V ^ j} = [x]_{\mathcal{P} ^ {(0, j-1)}} \cap [x]_{\mathcal{C} _ V ^ j} = [x] _ {\mathcal{C} _ V}\] unless there is a $\ell < \ell ' < j$ so that
\begin{equation*}
\widetilde T ^ {\ell '} (x+ \phi _ \AA (B ')) \not\subset [x]_{\mathcal  {P}}
.\end{equation*}
If $s = \max_ {\sigma \in S ' _ \KK} \absolute {\zeta _ \n} _ \sigma < 1$ then $\widetilde T ^ {\ell '} (x+ \phi _ \AA (B ')) = \widetilde T ^ {\ell '} (x)+ \phi _ \AA (B '')$ where
\[
B '' \subset  \prod_ {\sigma \in S ' _ \KK} \left\{ t  \in \KK _ \sigma: \absolute {t} _ \sigma \leq \eta s^{\ell'} \right\}
.
\]
By \eqref{eq: bound-on-boundary}, the set of such $x$ has $\widetilde \mu$ measure which decays exponentially in $\ell '$, hence in $j$, establishing \eqref{main point equation}.
\medskip

Once \eqref{main point equation} has been established, establishing \eqref{eq: entropy inequality} is easy. Indeed,  if $\mathcal{A} = \left\{ X _ j, X _ j^ \complement\right\}$ for $X _ j$ as in \eqref{main point equation} then for $j$ large and $\ell=\lceil \kappa j \rceil$
\begin{align}
h _ {\widetilde \mu} (\widetilde \alpha ^ {\mathbf n}, V \mid X _{\base} \times \Omega) & = j^{-1} H _ {\widetilde \mu} (\cC_V^{(0,j-1)} \mid \cC_V^j \vee \cB _{\base} \vee \cB_\Omega) \notag\\
& = j^{-1} H _ {\widetilde \mu} (\cP^{(0,j-1)} \mid \cC_V^j \vee \cB _{\base} \vee \cB_\Omega ) \notag\\
& = j^{-1} \Bigl(  H _ {\widetilde \mu} (\cP^{(0,\ell)} \mid \cC_V^j \vee \cB _{\base} \vee \cB_\Omega ) +\notag\\
&\hspace{2cm}   +H _ {\widetilde \mu} (\cP^{(0,j-1)} \mid \cP^{(0,\ell)}\vee \cC_V^j \vee \cB _{\base} \vee \cB_\Omega )\Bigr )\notag\\
& \leq j^{-1}  \Bigl( H _ {\widetilde \mu} (\cP^{(0,\ell)}\mid X_\base \times \Omega) + H_{\widetilde\mu}(\cA) + \label{bound for entropy}\\
&\hspace{2cm} +H _ {\widetilde \mu} (\cP^{(0,j-1)} \mid \cP^{(0,\ell)}\vee \cC_V^j \vee \cB _{\base} \vee \cB_\Omega \vee \cA ) \Bigr )\notag.
\end{align}
By definition
\begin{multline*}
H _ {\widetilde \mu} (\cP^{(0,j-1)} \mid \cP^{(0,\ell)}\vee \cC_V^j \vee \cB _{\base} \vee \cB_\Omega \vee \cA ) =\\
\widetilde \mu (X _ j) H _ {\widetilde \mu | X_j} (\cP^{(0,j-1)} \mid \cP^{(0,\ell)}\vee \cC_V^j \vee \cB _{\base} \vee \cB_\Omega) \\
+ \widetilde \mu (X _ j ^\complement) H _ {\widetilde \mu | X_j ^\complement} (\cP^{(0,j-1)} \mid \cP^{(0,\ell)}\vee \cC_V^j \vee \cB _{\base} \vee \cB_\Omega)
.\end{multline*}
On $X_j$, the atom $[x]_{\cP^{(0,j-1)}} \cap [x]_{\cC_V^j} = [x]_{\cP^{(0,\ell)}} \cap [x]_{\cC_V^j}$ by \eqref{main point equation}, hence
\[
H _ {\widetilde \mu | X_j} (\cP^{(0,j-1)} \mid \cP^{(0,\ell)}\vee \cC_V^j \vee \cB _{\base} \vee \cB_\Omega) =0;
\]
On $X _ j ^\complement$ we use the trivial bound
\[
H _ {\widetilde \mu | X_j ^\complement} (\cP^{(0,j-1)} \mid \cP^{(0,\ell)}\vee \cC_V^j \vee \cB _{\base} \vee \cB_\Omega) \leq j \log (\#\cP).
\]
Plugging these back in \eqref{bound for entropy} and using $\widetilde\mu(X_j^\complement)\to 0$ as $j \to \infty$ we see that \[
\eqref{bound for entropy} \to \kappa h _ {\widetilde \mu} (\widetilde \alpha ^ {\mathbf n} \mid X _{\base} \times \Omega) \qquad\text{ as $j\to\infty$,}
\]
establishing \eqref{eq: entropy inequality}.
\end{proof}

\section{Coarse Lyapunov subgroups and the product structure}\label{sec:productstructure}

A crucial property of the leafwise measures
for our argument is their product structure for the coarse Lyapunov subgroups as 
obtained by the first named author and Anatole Katok 
\cite{Einsiedler-Katok} (see also \cite{Lindenstrauss-Quantum, Einsiedler-Katok-II}). 
For an introduction of the product structure we also recommend \cite[\S~8]{EL-Pisa}. 
However, all of these papers assumed that the $\ZZ^d$-action under consideration is semisimple. In our case we do not make this assumption, so our action is given by 
$\alpha=\alpha_{\diag}\alpha_{\uni}$, with both $\alpha_{\diag}$ and $\alpha_\uni$ are defined over $\QQ$, i.e.\ can be thought of as homomorphisms from $\zd$ to $\GL_m(\QQ)$; cf.\ p.~\pageref{page of Jordan}.
The purpose of this section
is to recall the relevant notions and overcome the problems
arising from the lack of semisimplicity in the cases of interest.

\subsection{Lyapunov subgroups over $\QQ_\sigma$}
Fix $\sigma$ a place of $\QQ$, i.e.\ either $\sigma=\infty$ or $\sigma=p$ a prime.  
We consider the linear
maps $\alpha^\n$ for $\n\in\ZZ^d$ on $\QQ_\sigma^m$ and are mostly interested
in the asymptotic behavior of its elements with respect to the norm defined by
\[
 \|v\|_\sigma=\max_{j=1,\ldots,m}|v_j|_\sigma
\]
for all $v\in\QQ_\sigma^m$. We will also consider this behaviour restricted to an $\alpha$-invariant $\QQ_\sigma$-linear subspace $V$ (even when not explicitly stated, $V$ will always be assumed to be $\alpha$-invariant). 
For this we will initially ignore $\alpha_{\uni}$ (with polynomial behavior) 
and focus on $\alpha_{\diag}$ (with exponential behavior). 

Since $\alpha_{\diag}$ is semisimple, $\QQ^m_\sigma$ (as well as any $\alpha$-invariant subspace $V$) is a direct sum of
$\QQ_\sigma$-irreducible linear subspaces.
On each of these irreducible subspaces the action of $\alpha_{\diag}$
is isomorphic to the action defined by multiplication by $d$ elements 
$\zeta_1,\ldots,\zeta_d\in\KK$ on a local field $\KK$ (similar to the discussion in the proof of Proposition~\ref{adelicirred}). 
Recall that we extend the the norm on $\QQ_\sigma$ to $\KK$; this extended norm is denoted by $|\cdot|_\sigma$. We
refer to the linear functional $\chi:\ZZ^d\to\RR$ given by
\[
 \chi : (n_1,\dots,n_d)\mapsto \sum_i n_i\log|\zeta_i|_\sigma\in\RR
\]
as the \emph{Lyapunov weight} associated to the invariant subspace $\KK$, and denote the pairing of a functional $\chi$ and a vector $\bn \in \ZZ^d$ (or $\RR^d$) by $\chi\cdot\bn$. 
We note that for a vector $v$ in this subspace we have
\begin{equation}\label{eq:asymptoticformula}
 \|\alpha_{\diag}^\n v\|_\sigma\asymp e^{\chi\cdot \n}\|v\|_\sigma 
\end{equation}
for all $\n\in\ZZ^d$, where we write $\asymp$ to indicate that we can bound
each of the two terms by a multiple of the other. Here the implicit constants
only depend on the action and not on the vector $v$ or on $\n$. We will call $\KK$
an $\QQ_\sigma$-\emph{irreducible eigenspace} (for $\alpha_{\diag}$) with Lyapunov weight $\chi$.

It follows that $\QQ_\sigma^m$ (resp.\ $V$) is isomorphic to a finite direct
product of local fields $\KK$ extending $\QQ_\sigma$ so that the 
linear maps $\alpha_{\diag}^{\mathbf{e}_1},\ldots,\alpha_{\diag}^{\mathbf{e}_d}$ 
are written in diagonal form using this isomorphism. 
Moreover, we obtain in this way finitely many Lyapunov weights 
arising from the action on $\QQ_\sigma^m$. 
Note however, that these Lyapunov weights are all functionals into $\RR$
(independent of the place $\sigma$). 

\subsection{Coarse Lyapunov weights and subgroups}
We now apply the above for each place $\sigma$ of $\QQ$. 
Let $S$ be a finite set of places containing $\infty$ so that 
$\alpha_{\diag}^{\mathbf{e}_1},\ldots,\alpha_{\diag}^{\mathbf{e}_d}$ all belong
to $\operatorname{GL}_m(\ZZ_\sigma)$ for $\sigma\notin S$.
For $\sigma\notin S$ the only Lyapunov weight for the action of $\alpha_\diag$ on $\QQ_\sigma$ is the zero weight.
Hence by varying $\sigma$ over 
all places of $\QQ$
we only obtain finitely many nonzero Lyapunov weights. 

We say that two nonzero Lyapunov weights $\chi,\chi'$ 
(possibly arising
from different places of $\QQ$) are equivalent
if there exists some $t>0$ so that $\chi'=t\chi$. We will denote the equivalence
class of a nonzero Laypunov weight $\chi$ by $[\chi]$ and will refer to $[\chi]$
as a coarse Lyapunov weight.

For a coarse Lyapunov weight $[\chi]$ we define the 
\emph{coarse Lyapunov subgroup} $W^{[\chi]}$
to be the $S$-linear subspace defined as the subgroup generated
by all irreducible eigenspaces with Lyapunov 
weight $\chi'$ equivalent to $\chi$. 
Alternatively we can use \eqref{eq:asymptoticformula}
to see that the
coarse Lyapunov subgroup could also be defined as an intersection of stable horospherical 
subgroups for $\alpha_{\diag}$, namely
\[
W^{[\chi]}=\bigcap_{\n\in\ZZ^d: \chi\cdot\n<0}U^-_{\alpha_{\diag}^\n};
\]
this relation to the horospherical 
subgroups is the reason
for the dynamical importance of the coarse Lyapunov subgroups.

Because of the polynomial nature of $\alpha_{\uni}$ we have that $U^-_{\alpha_{\diag}^\n}=U^-_{\alpha^\n}$
for all $\n\in\ZZ^d$ and hence $W^{[\chi]}$ can be defined
directly in terms of the action $\alpha$ by
\[
W^{[\chi]}=\bigcap_{\n\in\ZZ: \chi\cdot \n<0}U^-_{\alpha^\n}.
\]
In particular,
$W^{[\chi]}$ is invariant under $\alpha^\n$ 
for all $\n\in\ZZ^d$.

Conversely, any stable horosphericle group is a product of coarse Lyapunov subgroups: indeed, 
 for any $\n\in\ZZ^d$ we have  
\begin{equation}\label{eq:coarsedecomp}
U^-_{\alpha^\n}=U^-_{\alpha_\diag^\n}=\bigoplus_{[\chi]: \chi\cdot\n<0}W^{[\chi]},
\end{equation}
where the direct sum runs over all the coarse Lyapunov subspaces $W^{[\chi]}$
satisfying that $\chi\cdot\n<0$. 

Given an $\alpha$-invariant $S$-linear subgroup $V<U^-_{\alpha^\n}$
we define $V^{[\chi]}=V\cap W^{[\chi]}$ for any nonzero coarse Lyapunov weight 
$[\chi]$. These also satisfy that $V$ is the direct
sum of $V^{[\chi]}$ for all $[\chi]$ with $\chi\cdot\n<0$.

\subsection{Product structure of leafwise measures}
\label{secthm:prod}

Recall that we assume that $\mu$ is as in
Theorem \ref{thm:main2}. Let
 $\Omega$ be an arbitrary compact metric space as in \S\ref{sec:leafwiseXtilde}
 equipped with an action of $\zd$,
 let $\widetilde{X}=X\times\Omega$, and consider an invariant probability
 measure $\widetilde{\mu}$ on $\widetilde{X}$ projecting to $\mu$. 

\begin{theorem}[Entropy and Product Structure]\label{thm:prod}
 Let $X \to X_\posfact\to X_\base$ with $X_\posfact=X/Y_\posfact$ and $X_\base=X/Y_\base$ be as in \eqref{x base equation} and
 let $V_{\base,\AA}<\AA^m$ be the rational $\alpha$-invariant subspace
 so that $Y_\base$ is the image of $V_{\base,\AA}$ modulo $\QQ^m$. 
 Let $\n_0\in\zd$. 
 Then the leafwise measure on $\widetilde{X}$ for 
 \[
 V^-_{\n_0}=V_{\base,\AA}\cap U^-_{\alpha^{\n_0}}
 \]
 is up to proportionality the product of the leafwise measures
 for its coarse Lyapunov subgroups $V^{[\chi]}=V^-_{\n_0}\cap W^{[\chi]}$, i.e.
 \[
  {\widetilde{\mu}}_x^{V^-_{\n_0}}\propto\prod_{[\chi]: (\chi\cdot\n_0)<0}{\widetilde{\mu}}_x^{V^{[\chi]}}
 \]
 for a.e.\ $x\in \widetilde{X}$. 
 In particular, the relative entropy of $\alpha^{\n_0}$ conditional
 on the factor $X_\base\times\Omega$ is equal to the sum of the 
 entropy contributions
 of these coarse Lyapunov subgroups,
 i.e.
 \begin{equation}\label{eq: entropyaddition}
  h_{\widetilde{\mu}}\bigl(\widetilde{\alpha}^{\n_0}\mid X_\base\times\Omega\bigr)=
  h_{\widetilde{\mu}}\bigl(\widetilde{\alpha}^{\n_0},V^-_{{\n_0}}\bigr)
  =\sum_{[\chi]: \chi\cdot\n_0<0}h_{\widetilde{\mu}}\bigl(\widetilde{\alpha}^{\n_0},V^{[\chi]}\bigr).
  \end{equation}
\end{theorem}

We note that our proof relies on the assumption that $X$ is a 
zero entropy extension of $X_\posfact$, and does not extend
the product structure to other cases with Jordan blocks. In fact
our setup is used to show that
Jordan blocks of $\alpha^\n$
cannot appear within the subspace $V^{[\chi]}\cap\RR^m$
and for ``$\n$ in the kernel of $\chi$'' (in the sense described below). 
Moreover, in the non-Archimedean 
parts of $V_{\n_0}^-$ we use a different argument. 

\medskip
For any $\alpha$-invariant $S$-linear
subgroup $V$ 
we let $P<V$ 
be the minimal $S$-linear $\alpha$-invariant subspace so that
\[
 {\widetilde{\mu}}_x^{V}\bigl(V\setminus P\bigr)=0
\]
for almost every $x\in\widetilde{X}$. 
We will refer to $P$ as the \emph{supporting subgroup} of $V$.
We note that $\mu_x^V=\mu_x^P$ a.s.\ 
and if $W<P$ is another $S$-linear $\alpha$-invariant subspace with $W\neq P$, then  $\mu_x^V(W)=0$ a.s.
This follows from \cite[Lemma 5.2]{Einsiedler-Lindenstrauss-joinings-2}
(as we restrict here to $\alpha$-invariant subspaces the assumption of 
class $\mathcal{A}'$ can easily be avoided in the proof of that lemma).\label{explaining lemma from joinings paper}

We say that $[\chi]$ is an \emph{exposed coarse Lyapunov weight} for 
an $\alpha$-invariant $S$-linear subgroup $V<U^-_{\alpha^\n}$
if $V^{[\chi]}$ is nontrivial, $V=V^{[\chi]}+V'$ for a sum $V'$
of coarse Lyapunov subgroups, and there exists some $\n'\in\RR^d$
with $\chi\cdot\n'=0$ and $\chi'\cdot\n'<0$
for all Lyapunov weights $\chi'$ of $V'$. 

We also recall that $\alpha_\uni^\n$ is polynomial map from $\ZZ^d \to \GL_m(\QQ)$ and so extends
to a homomorphism from $\RR^d$ to $\GL_m(\RR)$ that will also be denoted by $\alpha_\uni$.

\subsection{No shearing and proving the product structure}\label{sec:noshearargument}

The folllowing lemma stands in stark
contrast to the non-Archimedean case
where $\alpha_\uni$ takes values in a 
compact group.

\begin{lemma}[Existence of logarithmic sequence 
for real Jordan blocks]\label{lemma:logsequence}
Let $\alpha$ be a representation
of $\ZZ^d$ on a real vector space $P$. 
Suppose that $\chi$ is a nonzero Lyapunov weight
for $\alpha$ and that all other Lyapunov
weights are equivalent to $\chi$.
Suppose that $\alpha_\uni^{\bm}$ is nontrivial
for some $\bm\in\RR^d$ with $\chi\cdot \bm=0$. 
Then there exists a sequence $\n_k\in\ZZ^d$
so that $\alpha^{\n_k}\in\operatorname{End}(P)$ 
converges to a non-zero non-invertible linear 
map $L\in\operatorname{End}(P)$.
\end{lemma}

We will refer to the sequence $\n_k$ as a logarithmic
sequence for $P$ since it can be defined using
the logarithm map in easy special cases.

\begin{proof}[Proof of Lemma \ref{lemma:logsequence}]
Let $\alpha_{|\cdot|}$ be the representation of $\ZZ^d$ on $P$,
which has every eigenvector of $\alpha_\diag$ with 
eigenvalue $\lambda$ also as eigenvector but
with eigenvalue $|\lambda|$. 
Note that $\alpha_{|\cdot|}$ extends continuously
to all $\n\in\RR^d$. Moreover
for $\n\in\ZZ^d$ the map $\alpha^{\n}_\diag\alpha_{|\cdot|}^{-\n}$
belongs to a fixed compact subgroup of $\GL(P)$.
Using nearest integer vectors
we see that it suffices to construct a sequence $\n_k\in\RR^d$
so that $\alpha_\uni^{\n_k}\alpha_{|\cdot|}^{\n_k}$
converges to a non-zero non-invertible linear map $L$.

For this we apply our assumption
and pick a direction $\bm\in\RR^d$ with $\chi\cdot \bm=0$
so that $\alpha^{\bm}_\uni$ nontrivial. Also let $\n_-\in\RR^d$
with $\chi\cdot \n_-<0$. As a non-zero real polynomial, $\alpha_\uni^{k\bm}$ diverges as $k\to\infty$ and $\alpha_{|\cdot|}^{t \n_-}\to 0$ as $t\to \infty$. 
Also note that $\alpha_{|\cdot|}^{t \n_-}\to 0$ converges
exponentially fast while $\alpha_\uni^{t\n_-}$ can only diverge polynomially fast
as $t\to \infty$. 

This shows that for each sufficiently large $k\in\NN$ we can
define $\n_k=k \bm+t_k \n_-$,
where $t_k>0$ is chosen (using the intermediate value 
theorem) minimally so that the Hilbert-Schmidt norms satisfy
\[
 \|\alpha_{\uni}^{\n_k}\alpha_{|\cdot|}^{\n_k}\|
 =\|\alpha_\uni^{k\bm+t_k\n_-}\alpha_{|\cdot|}^{t_k \n_-}\|=1.
\]
Since $\alpha_\uni^{k\bm}$ diverges as $k\to\infty$
we see that also $t_k\to\infty$. If now $v\in P$
is a common eigenvector for $\alpha$, then $\alpha_\uni^{k\bm}v=v$
and so
\[
 \alpha_{\uni}^{\n_k}\alpha_{|\cdot|}^{\n_k}v
 =\alpha_\uni^{t_k\n_-}\alpha_{|\cdot|}^{t_k \n_-}v\to 0
\]
as $k\to\infty$ (since $\alpha_\uni^{t_k\n_-}$
is polynomial and $\alpha_{|\cdot|}^{t_k \n_-}$
contracts $P$ at exponential rate).

From this and compactness
of the unit ball in finite dimensions it follows
that there exists a converging subsequence 
of $\alpha_{\uni}^{\n_k}\alpha_{|\cdot|}^{\n_k}$
whose limit is non-zero and non-invertible.
As indicated in the beginning of the proof
a subsequence of the integer vectors closest to 
$\n_k\in\RR^d$ will satisfy the conclusions
of the lemma. 
\end{proof}

We need the following upgrade to the above concerning
exposed coarse Lyapunov weights:

\begin{lemma}[Properties of logarithmic sequence for adelic action]\label{lemma:log2seq}
Let $\alpha$ be a linear representation of $\ZZ^d$
on $\QQ^m$ defining an adelic action on $X_m$. 
Let $\n_-\in\ZZ^d\setminus\{0\}$ and $V<U^-_{\alpha^{\n_-}}$
be an $S$-linear subspace. Let $[\chi]$
be an exposed  coarse Lyapunov weight $[\chi]$ of $V$.
Suppose that there exists some $\bm\in\RR^d$ with $\chi\cdot \bm=0$
so that $(\alpha_\uni|_{V\cap\RR^m})^{\bm}$ 
is nontrivial. 
Then there exists a sequence $\n_k\in\ZZ^d$ 
so that $(\alpha|_V)^{\n_k}$ converges 
(uniformly within compact subsets) to 
a non-zero map $L\in\operatorname{End}(V)$ 
that vanishes on all non-Archimedean
subspaces, vanishes on all coarse Lyapunov 
subgroups $V^{[\chi']}$ with $[\chi']\neq[\chi]$, 
and whose restriction to $V^{[\chi]}\cap\RR^m$ is non-invertible.
\end{lemma}

\begin{proof}
 This actually follows by the same argument as 
 Lemma \ref{lemma:logsequence} after choosing $\bm\in\RR^d$
 correctly. Indeed we first note that the kernel of the homomorphism
 $\n\in\chi^\perp\mapsto\alpha_\uni^\n\in\GL(V\cap\RR^m)$
 is a proper subspace $K<\chi^\perp$, 
 and hence our first constraint on $\bm$
 is simply $\bm\in\chi^\perp\setminus K$.

 By definition $[\chi]$ is an exposed coarse
 Lyapunov weight for $V$ if there exists some $\bm\in\chi^\perp$ with $\chi'\cdot \bm<0$
 for all Lyapunov weights $\chi'$ of $V$ inequivalent
 to~$\chi$. 
 The latter condition is clearly satisfied
 by all elements of an open subset of $\chi^\perp$.
 Hence we can find $\bm\in\chi^\perp\setminus K$ 
 with $\bm\cdot\chi'<0$
 for all Lyapunov weights $\chi'$ of $V$ inequivalent
 to $\chi$. 
 
 Using this $\bm$ together with 
 $\n_-$ as in the assumptions
 of the lemma we now go again through the construction
 in the proof of Lemma \ref{lemma:logsequence}.
 Let $\n_k'\in\ZZ^d$ be a nearest integer approximation
 of $k\bm$ and let $\n_k$ be the nearest integer approximation of $k\bm+t_k\n_-$. 
 For all non-Archimedean subspaces $V_\sigma$ we note that $\alpha_\uni(\ZZ^d)|_{V_\sigma}$ 
has compact closure.
 Hence the restriction $\alpha^{\n_k'}|_{V_\sigma}$
 belongs to a compact subset of $\operatorname{Hom}(V_\sigma)$,
 which implies together with $\n_-$ contracting $V$
 and $t_k\to\infty$
 that the restriction $\alpha^{\n_k}|_{V_\sigma}$
 converges to zero. Similarly,
 our choice of $\bm$ implies
 that for a coarse Lyapunov weight $[\chi']\neq[\chi]$
 both $\alpha^{\n_k'}|_{V^{[\chi']}}$ and 
 $\alpha^{\n_k}|_{V^{[\chi']}}$ converge to the trivial
 map. Finally our proof 
 of Lemma \ref{lemma:logsequence}
 ensures the claimed properties of 
 the limit map $L|_{V_\infty^{[\chi]}}$. 
\end{proof}

We will now combine the above 
with the setup of Section \ref{sec:standing}
to prove a restriction
concerning the supporting subgroups (cf.~\S\ref{secthm:prod}). 

\begin{proposition}[No shearing on 
supporting subgroup]\label{prop;noshearing}
 Let $P$ be the supporting
 subgroup of 
 $V^-_{\n_0}=V_{\base,\AA}\cap U^-_{\alpha^{\n_0}}$
 (with $V_{\base,\AA}$ as defined in Theorem~\ref{thm:prod}).
 Let $\chi$ be an exposed 
 Lyapunov weight of $P$. 
 Then $(\alpha_\uni|_{P^{[\chi]}\cap\RR^m})^{\n}$ 
  is trivial for all $\n\in\RR^d$ with $\chi\cdot\n=0$. 
\end{proposition}

\noindent
Here we identified $\RR^m$ with the corresponding subgroup of $\AA^m$, and hence $P^{[\chi]}\cap\RR^m$
 is the maximal real subspace
 of the supporting subgroup~$P$.
 
\begin{proof}
 We suppose in contradiction that
 $(\alpha_\uni|_{P^{[\chi]}\cap\RR^m})^{\bm}$
 is nontrivial for some $\bm\in\RR^d$ with $\chi\cdot\bm=0$.
 Applying Lemma \ref{lemma:log2seq} we find
 a logarithmic sequence $\n_k$ and the limit
 $L\in\operatorname{End}(P)$ of $(\alpha|_P)^{\n_k}$.
 
 Recall that $Y_\base<X$ is the adelic subgroup
 so that $X_\base=X/Y_\base$ and 
 that $V_{\base,\AA}$ is the rational subspace of $\AA^m$
 so that $Y_\base$ is the image of $V_{\base,\AA}$ modulo $\QQ^m$. 
 Recall also that $Y_\posfact<Y_\base$ is the adelic subgroup
 so that $X_\posfact=X/Y_\posfact$ and let $V_{\posfact,\AA}<V_{\base,\AA}$ be the
 rational subspace so that $Y_\posfact$ is the image of $V_{\posfact,\AA}$ modulo $\QQ^m$.
 
 By construction
 we have that the action on $Y_\base/Y_\posfact$ is $\AA$-irreducible,
 or equivalently that the linear representation of $\ZZ^d$
 on $V_{\base,\AA}/V_{\posfact,\AA}$ (defined over $\QQ$) is irreducible over $\QQ$. In particular,
 this representation is semisimple. 
 Unfolding the definitions and restricting to $P^{[\chi]}\cap\RR^m$
 it follows that 
 for any $v \in P^{[\chi]}\cap\RR^m$
 \[
 \alpha_\uni^\bm (v) \in v + V^{[\chi]}
 \]
  where $V^{[\chi]}=W^{[\chi]}\cap V_{\posfact,\AA}$. 
  Combining this information with the construction of the logarithmic
 sequence and its limit $L$ it follows that
 \[
   L(P^{[\chi]}\cap\RR^m)\subseteq V^{[\chi]}.
 \]

 Next recall from \eqref{x base equation} that $X$ is a zero entropy
 extension over  $X_\posfact=X/Y_\posfact$, i.e.\ 
 $h_{\widetilde{\mu}}(\alpha^\n\mid X_\posfact\times\Omega)\leq h_\mu(\alpha^n\mid X_\posfact)=0$
 for all $\n\in\ZZ^d$.
 This implies by the relationship between the leafwise measures
 and entropy that the leafwise measures $\widetilde{\mu}_x^{V^{[\chi]}}$
 must be trivial a.e.\ --- indeed otherwise
 there would be positive entropy contribution for the
 relative entropy over the factor $X_\posfact\times\Omega$.
 By the compatibility property \eqref{eq:shifting}
 it follows that there exists a set $X'\subseteq\widetilde{X}=X\times\Omega$ of full measure
 so that $x,\ x+w\in X'$ for some $w\in V^{[\chi]}$ implies $w=0$. 
 Using regularity of the Borel probability 
 measure we choose some compact $K\subseteq X'$ of measure ${\widetilde{\mu}}(K)>0.99$. 
 
 Our aim in the proof is to use the logarithmic sequence $\n_k$ and
 the limit $L$ to derive a contradiction to the properties of $K$.
 For this let $\mathcal{A}$ be a $\sigma$-algebra 
 that is subordinate to $P$.
 We define
 \[
  X_k=\left\{x\in \alpha^{-\n_k}K\mid {\widetilde{\mu}}_x^{\mathcal{A}}(\alpha^{-\n_k}K)>0.9\right\}
 \]
 and note that ${\widetilde{\mu}}(X_k)>0.89$. By the Lemma of Fatou (applied for the probability measure ${\widetilde{\mu}}$)
 it follows that
 \[
   {\widetilde{\mu}}\bigl(\textstyle{\limsup_{k\to\infty}}X_k\bigr)>0.89.
 \]
 Hence there exists some $x_0$ and some subsequence $\n_k'$
 of $\n_k$ so that $\alpha^{\n_k'}x_0\in K$
 and ${\widetilde{\mu}}_{x_0}^{\mathcal{A}}(\alpha^{-\n_k'}K)>0.9$.
 Applying Lemma of Fatou again -- but this time for the probability
 measure ${\widetilde{\mu}}_{x_0}^{\mathcal{A}}$ -- we obtain
 \begin{equation}\label{eq:probarg}  {\widetilde{\mu}}_{x_0}^{\mathcal{A}}\left(\textstyle{\limsup_{k\to\infty}}
            \alpha^{-\n_k'}K\right)>0.9.
 \end{equation}
 
 Also recall that $P$ is the supporting subgroup of $V_{\n_0}^-$
 and that $\ker L<P$ is a proper $S$-linear $\alpha$-invariant
 subgroup. Using \cite[Lemma 5.2]{Einsiedler-Lindenstrauss-joinings-2} (cf.\ p.~\pageref{explaining lemma from joinings paper} above) 
 this implies that ${\widetilde{\mu}}_x^P(\ker L)=0$, which
 in turns implies ${\widetilde{\mu}}_x^{\mathcal{A}}(x+\ker L)=0$ almost surely.
 We may assume that our $x_0$ constructed above has this property.
 Hence there exists some 
 $x_1\in[x_0]_{\mathcal{A}}\setminus(x_0+\ker L)$ 
 and another subsequence $\n_k''$ of $\n_k'$ so that
 $\alpha^{\n_k''}x_1\in K$. 
 
 To summarize we have found a subsequence $\n_k''$
 of the logarithmic sequence $\n_k$, some $x_0\in X$, and 
 some $x_1=x_0+v\in[x]_{\mathcal{A}}$ with $v\in P\setminus\ker L$ 
 so that $\alpha^{\n_k''}x_j\in K$ for $j=0,1$.
 Using compactness of $K$ we can find yet another subsequence 
 of $\n_k''$ so that $\alpha^{\n_k''}x_0 \to y_0$ and $\alpha^{\n_k''}x_1 \to y_1$
 with $y_0,y_1\in K$. Moreover, 
 since $\alpha^{\n_k}v\rightarrow Lv$ as $k\to\infty$
 we also have $y_1=y_0+Lv$ with $Lv\in V^{[\chi]}\setminus\{0\}$
 by the properties of $L$. However, this contradicts
 the properties of $K\subset X'$ and so concludes the proof.
\end{proof}

\begin{proof}[Proof of Theorem \ref{thm:prod}]
 In view of Proposition~\ref{prop;noshearing},
 the product structure of the leafwise measure follows
 from \cite[Thm.~8.2]{Einsiedler-Katok-II} (or more precisely
 its proof). 
 
 For this we first recall that if $P<V_{\n_0}^-$ is the corresponding supporting subgroup then by Proposition~\ref{prop: subspace}, the leafwise measure $\mu_x^{V_{\smash{\n_0}}^-}$
 coincides with $\mu_x^{P}$ for $\mu$-a.e.~$x$.
 Next we recall that by Proposition \ref{prop;noshearing} we have for any nontrivial 
 coarse Lyapunov weight $[\chi]$ that $\alpha_\uni^\bm|_{P^{[\chi]}\cap \RR^m}$
 is trivial for any $\bm \in \ker\chi$.
 We also note that $\alpha_\uni(\zd)$ restricted to 
 the $p$-adic subspaces $P^{[\chi]}\cap\QQ_p^m$ of the coarse Lyapunov 
 subspace $P^{[\chi]}$
 has compact closure in $\GL(P^{[\chi]}\cap\QQ_p^m)$.
 Hence we may assume that the metric on $P^{[\chi]}\cap\QQ_p^m$ is
 invariant under $\alpha_\uni(\ZZ^d)$.
 With these two observations the inductive argument for
 \cite[Thm.~8.2]{Einsiedler-Katok-II} applies and proves 
 the product structure for ${\widetilde{\mu}}_x^{V_{\smash{\n_0}}^-}={\widetilde{\mu}}_x^P$. 
 
 The product structure implies
 now quite directly using \eqref{eq:entropygrowthrate}
 that the entropy contribution for $V_{\n_0}^-$
 equals the sum of the entropy contributions of its coarse Lyapunov
 subgroups $V^{[\chi]}$, hence the second equality in \eqref{eq: entropyaddition}. 
 
 The first equality in \eqref{eq: entropyaddition}, i.e.\ the fact that $h(\widetilde\alpha^{\n_0},V^-_{\n_0})$
 equals the conditional entropy of $\widetilde\alpha^{\n_0}$ over the 
 factor $X_\base\times\Omega$ follows e.g.\ 
 from the proof of \cite[Thm.~7.6]{EL-Pisa}. Indeed by conditioning
 the calculation there on the factor $X_\base\times\Omega$
 the leafwise measure on the full stable horospherical is automatically
 supported on $V^-_{\n_0}$, since a displacement by some element
 of the stable horospherical not belonging to $V_{\base,\AA}$ would
 change the point within $X_\base$. 
\end{proof}

\section{Proof of Theorem \ref{thm:main2}}

\subsection{Rigidity of the entropy function}
\label{sec: entropy-function} 

In the following we will again consider entropy contributions 
for various coarse Lyapunov subgroups
with varying definitions of the second factor $\Omega$ in the framework of 
Section \ref{sec: leaf-entropy}. 
Consistent with our notations so far, we will use e.g.\ 
$h_{\widetilde{\mu}}(\widetilde\alpha^\n,W^{[\chi]}\mid X_\base\times\Omega)$
for the entropy contribution of a coarse Lyapunov subgroup $W^{[\chi]}$
on $X\times\Omega$ and similarly for other factors and foliations.

We now establish the following identity regarding the relation between the entropies of 
individual elements of the action. This identity is central to our approach.

\begin{theorem}\label{thm: shape}
Let $\alpha$ be a $\zd$-action on $X=X_m$ without cyclic factors as in Theorem \ref{thm:main2}. 
Let $\mu$ be an $\alpha$-invariant probability measure, 
let $X_\posfact,\ X_\base,\ Y_\irred$ be as \eqref{x base equation}
and denote the Haar measure on $Y_\irred$ by $\lambda$.
Moreover, let $\Omega$
be a compact metric space equipped with an action of $\zd$, let
$\widetilde{\alpha}$ be the corresponding $\zd$-action on $X\times\Omega$,
and let $\widetilde{\mu}$ be an invariant measure on $X\times\Omega$ projecting to $\mu$. 
Then there exists a
constant $\kappa_{\widetilde{\mu},\Omega}>0$ with
\begin{equation}\label{entropy identity}
 \h_{\widetilde\mu}(\widetilde\alpha^\n, W^{[\chi]}\mid X_\base\times\Omega)
=\kappa_{\widetilde{\mu},\Omega}\h_\lambda(\alpha_{Y_\irred}^\n, W^{[\chi]})   
\end{equation}
for every $\n \in\zd$. 
\end{theorem}

\noindent
We recall that we use $\lambda$ to denote the Haar measure on the appropriate adelic quotient (that should hopefully be clear from the context; e.g.\ in~\eqref{entropy identity}, \ $\lambda$ is the Haar measure on $Y_\irred$).
 While the proof of Theorem~\ref{thm: shape} is much more complicated than in the case considered by Rudolph, this theorem plays a similar role in our proof as \cite[Thm.~3.7]{Rudolph-2-and-3} did in Rudolph's proof in \cite{Rudolph-2-and-3}.

As a first step towards the theorem we consider just one coarse Lyapunov subgroup.
Note that unlike Theorem~\ref{thm: shape} which uses in an essential way the irreducibility of $Y _{\irred}$, 
the next lemma only uses the fact that $W ^ {[\chi]}$ is a coarse Lyapunov group.

\begin{lemma}\label{lem: linear-contribution}
Using the same notation as in Theorem \ref{thm: shape},
let $[\chi]$ be a coarse Lyapunov weight for $Y_\irred$.
Then there exists some $\kappa_{\widetilde{\mu},\Omega,[\chi]} \geq 0$ with
\[
\h_{\widetilde\mu}(\widetilde\alpha^\n, W^{[\chi]}\mid X_\base\times\Omega)
=\kappa_{\widetilde{\mu},\Omega,[\chi]}\h_\lambda(\alpha_{Y_\irred}^\n, W^{[\chi]})
\]
for every $\n \in\zd$.
\end{lemma}

\begin{proof} 
We first note that for $\mathbf n\in\zd$ with $\chi\cdot\n<0$ and $k\in\NN$,
Proposition~\ref{prop: contribution} implies 
\begin{equation} \label{eq: N-times}
\h_{\widetilde\mu}(\widetilde\alpha^{k\n}, W^{[\chi]}\mid X_\base\times\Omega)
=k \h_{\widetilde\mu}(\widetilde\alpha^\n, W^{[\chi]}\mid X_\base\times\Omega).
\end{equation} 
Moreover, since $W^{[\chi]}$ is a coarse Lyapunov subgroup, 
for all $\n,\bm\in\zd$ with $\chi\cdot\n < \chi\cdot\bm<0$, we have that
\[
\h_{\widetilde\mu}(\widetilde\alpha^{\n}, W^{[\chi]}\mid X_\base\times\Omega)
 \geq 
\h_{\widetilde\mu}(\widetilde\alpha^{\bm}, W^{[\chi]}\mid X_\base\times\Omega).
\]
In conjunction with \eqref{eq: N-times} this implies elementarily that there is a constant $c\geq 0$ depending only on $\tilde \mu$,
$W^{[\chi]}$, and $\alpha$ so that
that $\h_{\widetilde \mu}(\widetilde \alpha^\n, W^{[\chi]})=c
|\chi\cdot\n|$ for all $\n\in\zd$. 

Next notice that for similar reasons $\h_{\lambda}(\alpha^\n, W^{[\chi]})$
is given by a similar formula for a constant $c_\lambda\geq 0$. 
As $[\chi]$ is assumed to be a coarse Lyapunov weight for $Y_\irred$ we have $c_\lambda>0$
and obtain the lemma with $\kappa_{\widetilde{\mu},\Omega,[\chi]}=c/c_\lambda$.
\end{proof}

\begin{lemma}\label{lemma:entropycontvanishes}
 We again use the notation in Theorem \ref{thm: shape}. 
 For any coarse Lyapunov weight $[\chi]$ of $X$
 we have
 \[
  h_{\widetilde{\mu}}(\widetilde\alpha^\n, W^{[\chi]}\cap V_{\posfact,\AA}\mid X_\base\times\Omega)=0,
 \]
 where $V_{\posfact,\AA}<\AA^m$ is the rationally defined subspace so that $Y_\posfact$ is the image of $V_{\posfact,\AA}$ modulo $\QQ^m$.
 Moreover, if $[\chi]$ is \emph{not} a coarse Lyapunov weight for $Y_\irred$ 
 then
 \[
  \h_{\widetilde\mu}(\alpha^\n, W^{[\chi]}\mid X_\base\times\Omega)=0.
 \]
\end{lemma}

\begin{proof}
 We claim for the entropy contributions
 for $W^{[\chi]}\cap V_{\posfact,\AA}$ that
 \begin{equation}\label{eq:anotherdoubleconditioning}
       h_{\widetilde{\mu}}(\widetilde\alpha^\n, W^{[\chi]}\cap V_{\posfact,\AA}\mid X_\base\times\Omega)
  = h_{\widetilde{\mu}}(\widetilde\alpha^\n, W^{[\chi]}\cap V_{\posfact,\AA}\mid X_\posfact\times\Omega).
 \end{equation}
  To see this recall that $Y_\posfact<Y_\base$, and correspondingly $X_\base=X/Y_\base$ is a factor of $X_\posfact=X/Y_{\posfact}$. By \peqref{eq: definition of entropy contribution}, the conditional entropy contribution
  \[
  h_{\widetilde{\mu}}(\widetilde\alpha^\n, W^{[\chi]}\cap V_{\posfact,\AA}\mid X_\base\times\Omega)
  \]
  is given by
\[
h_{\widetilde{\mu}}(\widetilde\alpha^\n, W^{[\chi]}\cap V_{\posfact,\AA}\mid X_\base\times\Omega) = H _ {\widetilde \mu} (\mathcal{C} \mid \widetilde\alpha^{-\n} \mathcal{C}  \vee \mathcal{B} _ {{\base}} \vee \mathcal B _ {\Omega})
\]
for $\mathcal C$ a ${W^{[\chi]}\cap V_{\posfact,\AA}}$-subordinate $\sigma$-algebra for $X\times\Omega$ (and $\mathcal{B} _ {{\base}}$ and $\mathcal B _ {\Omega}$ the $\sigma$-algebras of Borel measurable sets on $X_\base$ and $\Omega$ respectively).  However, since each atom of $\mathcal C$ is contained in a single orbit of $Y_\posfact$, its image under the projection from $X$ to $X _{\posfact} = X / Y _{\posfact}$ consists of a single point, so modulo $\widetilde \mu$,
\[
\widetilde\alpha^{-\n} \mathcal{C}  \vee \mathcal{B} _ \base = \widetilde\alpha^{-\n} \mathcal{C}  \vee \mathcal{B} _ { \posfact}
\]
(with $\cB_\posfact$ the Borel $\sigma$-algebra on $X_\posfact$). It follows that
\begin{align*}
 h_{\widetilde{\mu}}(\widetilde\alpha^\n, W^{[\chi]}\cap V_{\posfact,\AA}\mid X_\base\times\Omega) &= H _ {\widetilde \mu} (\mathcal{C} \mid \widetilde\alpha^{-\n} \mathcal{C}  \vee \mathcal{B} _ \base \vee \mathcal B _ {\Omega})   \\
 &=H _ {\widetilde \mu} (\mathcal{C} \mid \widetilde\alpha^{-\n} \mathcal{C}  \vee \mathcal{B} _ \posfact \vee \mathcal B _ {\Omega})\\
 &=h_{\widetilde{\mu}}(\widetilde\alpha^\n, W^{[\chi]}\cap V_{\posfact,\AA}\mid X_\posfact\times\Omega)
\end{align*}
as claimed.

 Using \eqref{eq:anotherdoubleconditioning} 
 and the relation between entropy contributions and entropy in Proposition \ref{prop: contribution}
 we now obtain 
 \begin{align*}
  h_{\widetilde{\mu}}(\widetilde\alpha^\n, W^{[\chi]}\cap V_{\posfact,\AA}\mid X_\base\times\Omega)
  &\leq  h_{\widetilde{\mu}}(\widetilde\alpha^\n\mid X_\posfact\times\Omega) \\
  &\leq  h_{\mu}(\alpha_{X}^\n\mid X_\posfact)=0,
 \end{align*}
where the last inequality follows from the choice of $X _{\posfact}$ in \S\ref{sec:standing}.

So suppose now that $[\chi]$ is not a coarse Lyapunov weight for $Y_\irred$.
Let $V_{\base,\AA}<\AA^m$ be the rational subspace corresponding to $Y_\base$.
Since we consider entropy contribution conditional on $X_\base\times\Omega$ we may use
the above argument again and 
replace the coarse Lyapunov subgroup $W^{[\chi]}$ with intersection $W^{[\chi]}\cap V_{\base,\AA}$.
However, since $[\chi]$ is not a coarse Lyapunov
weight for $Y_\irred=Y_\base/Y_\posfact$ it follows 
that $W^{[\chi]}\cap V_{\base,\AA}=W^{[\chi]}\cap V_{\posfact,\AA}$.
Now the first part of lemma implies that the entropy contribution vanishes. Since
the entropy contribution for the Haar measure~$\lambda$ on $Y_\irred$ vanishes too
this proves the lemma.
\end{proof}

\begin{proof}[Proof of Theorem \ref{thm: shape}]
 By the Abromov-Rokhlin entropy addition formula we have
 \[
  h_{\widetilde{\mu}}(\widetilde\alpha^\n\mid X_\base\times\Omega)=
  h_{\widetilde{\mu}}(\widetilde\alpha^\n\mid X_\posfact\times\Omega)+
  h_{\widetilde{\mu}}(\alpha_{X_\posfact\times\Omega}^\n\mid X_\base\times\Omega)
 \]
 for all $\n\in\zd$.
 However, by properties of $X_\posfact=X/Y_\posfact$ in \eqref{x base equation}
 we have
 \begin{equation}\label{zero entropy equation}
        h_{\widetilde{\mu}}(\widetilde\alpha^\n\mid X_\posfact\times\Omega)
   \leq h_{{\mu}}(\alpha^\n\mid X_\posfact)=0.
 \end{equation}
Therefore, 
  \begin{equation}\label{eq:entropyequalityinproof}
        h_{\widetilde{\mu}}(\widetilde\alpha^\n\mid X_\base\times\Omega)=
  h_{\widetilde{\mu}}(\alpha_{X_\posfact\times\Omega}^\n\mid X_\base\times\Omega)\qquad\text{for all $\n\in\zd$.}
  \end{equation}

 We claim that the entropy contributions for all coarse Lyapunov subgroups
 $W^{[\chi]}$ satisfy a similar equation, namely
  \begin{equation}\label{eq:entropyeqinprooftimes2}
  h_{\widetilde{\mu}}(\widetilde\alpha^\n,W^{[\chi]}\mid X_\base\times\Omega)=
  h_{\widetilde{\mu}}(\alpha_{X_\posfact\times\Omega}^\n,W_\posfact^{[\chi]}\mid X_\base\times\Omega)
 \end{equation}
 for all $\n\in\zd$ with $\chi\cdot\n<0$, where $W^{[\chi]}$ and $W_\posfact^{[\chi]}$ denote the coarse Lyapunov
 subgroups for $X$ and $X_\posfact$ respectively.
 Applying Proposition~\ref{entropy inequality proposition} to $W^{[\chi]}$ and $W_\posfact^{[\chi]}$ (and with $X_\base\times \Omega$ playing the role of $\Omega$ in Proposition~\ref{entropy inequality proposition})
we conclude that
\begin{multline*}
 h_{\widetilde{\mu}}(\widetilde\alpha^\n,W^{[\chi]}\mid X_\base\times \Omega)\leq h_{\widetilde{\mu}}(\alpha_{X_\posfact\times\Omega}^\n,W_\posfact^{[\chi]}\mid X_\base\times \Omega) \\
+ h_{\widetilde\mu}(\widetilde\alpha^\n,W^{[\chi]}\cap V_{\posfact,\AA}\mid X_\base\times \Omega)
\end{multline*}
with $V_{\posfact,\AA} < \AA^m$ as in Lemma \ref{lemma:entropycontvanishes}. 
By Lemma \ref{lemma:entropycontvanishes},  \[h_{\widetilde{\mu}}(\widetilde\alpha^\n, W^{[\chi]}\cap V_{\posfact,\AA}\mid X_\base\times\Omega)=0,\]
hence
 \begin{equation}\label{ineqtobeimproved}
    h_{\widetilde{\mu}}(\widetilde\alpha^\n,W^{[\chi]}\mid X_\base\times\Omega)\leq
  h_{\widetilde{\mu}}(\alpha_{X_\posfact\times\Omega}^\n,W_\posfact^{[\chi]}\mid X_\base\times\Omega).
 \end{equation}
 
 Now fix some $\n\in\zd$ and take the sum over all coarse Lyapunov weights $[\chi]$
 with $(\chi\cdot\n)<0$. By the second claim in Theorem \ref{thm:prod} (applied both to $X$
 and to $X_\posfact$)  this leads to an inequality between the two
 terms in \eqref{eq:entropyequalityinproof}. However, since in \eqref{eq:entropyequalityinproof} equality holds, equality for the entropy contributions in \eqref{ineqtobeimproved} must hold
 as well. Varying $\n\in\zd$
 gives \eqref{eq:entropyeqinprooftimes2} for all coarse Lyapunov weights $[\chi]$.
 
 Next we are going to combine Theorem \ref{thm:bound-contribution} and Theorem \ref{thm:prod}. By  Theorem \ref{thm:bound-contribution}
\begin{equation}\label{eq:crucialineqbeforedone}
       h_{\tilde{\mu}}(\alpha^\n_{X_\posfact\times\Omega},W_\posfact^{[\chi]}\mid X_\base\times\Omega)\leq 
  \frac{h_{\tilde{\mu}}(\alpha^\n|X_\base\times\Omega)}{h_\lambda(\alpha_{Y_\irred}^\n)}
  h_{\lambda}(\alpha^\n_{Y_\irred},W_{Y_\irred}^{[\chi]}).
 \end{equation}
 Set $\kappa_{\n,\widetilde{\mu},\Omega} = h_{\tilde{\mu}}(\alpha^\n|X_\base\times\Omega)/h_\lambda(\alpha_{Y_\irred}^\n)$; note that it
 does not depend on the coarse Lyapunov weight $[\chi]$.
 Taking the sum over all coarse Lyapunov weights $[\chi]$ with $\chi\cdot\n<0$
 gives by \eqref{eq: entropyaddition} of Theorem~\ref{thm:prod} on the left hand side the conditional entropy $h_{\tilde{\mu}}(\alpha^\n|X_\base\times\Omega)$
 and on the right hand side we obtain
 $
 \kappa_{\n,\widetilde{\mu},\Omega}h_\lambda(\alpha_{Y_\irred}^\n)$, which in view of the definition of $\kappa_{\n,\widetilde{\mu},\Omega}$ also equals
$ h_{\tilde{\mu}}(\alpha^\n|X_\base\times\Omega)$.
 This shows that in fact
 \begin{equation*}\ptag{eq:crucialineqbeforedone}
       h_{\tilde{\mu}}(\alpha^\n_{X_\posfact\times\Omega},W_\posfact^{[\chi]}\mid X_\base\times\Omega)=\kappa_{\n,\widetilde{\mu},\Omega}\,
  h_{\lambda}(\alpha^\n_{Y_\irred},W_{Y_\irred}^{[\chi]}).
 \end{equation*}
  for all coarse Lyapunov weights
 $[\chi]$ with $\chi\cdot\n<0$.  
 
 We now choose $\n_0$ so that $\chi\cdot\n_0\neq 0$ for all coarse Lyapunov weights $[\chi]$.
 Since
 \[
\kappa_{\n_0,\widetilde{\mu},\Omega}=\kappa_{-\n_0,\widetilde{\mu},\Omega}.
 \]
 equation \peqref{eq:crucialineqbeforedone} together with \eqref{eq:entropyeqinprooftimes2}, 
implies that
 \[
   h_{\widetilde{\mu}}(\widetilde\alpha^{\n_0},W^{[\chi]}\mid X_\base\times\Omega)  
  =\kappa_{\n_0,\widetilde{\mu},\Omega}
  h_{\lambda}(\alpha^{\n_0}_{Y_\irred},W^{[\chi]})
 \]
 for all coarse Lyapunov weights $[\chi]$.
 It follows that the constants appearing in Lemma \ref{lem: linear-contribution}, that may depend on $[\chi]$ but not on $\n_0$,
 agree with $\kappa_{\n_0,\widetilde{\mu},\Omega}$, that may depend on $\n_0$ but not on $[\chi]$. This gives the theorem.
\end{proof}

\subsection{Rigidity of the entropy function implies invariance}\label{sec: rigidity implies invariance}

We are now ready to prove Theorem \ref{thm:main2}.
As before we work
with the setup explained in \S\ref{sec:standing}, specifically \eqref{x base equation}: $X_\posfact=X/Y_\posfact$
is a factor of $X$ so that $h_\mu(\alpha^\n|X_\posfact)=0$ for all $\n \in \ZZ^d$, \ $Y_\irred$ is an $\alpha$-invariant $\AA$-irreducible subgroup of $X_\posfact$, and $X_\base=X/Y_\base=X_\posfact/Y_\irred$ satisfies that for some $\n\in\ZZ^d$ we have that $h_\mu(\alpha^\n_{X_\posfact}|X_\base)>0$. 

Applying Theorem \ref{thm: shape}
with $\Omega$ being the trivial factor we obtain a constant $\kappa_\mu>0$ 
so that $h_\mu(\alpha^\n|X_\base)=\kappa_\mu h_\lambda(\alpha_{Y_\irred}^\n) >0$ for all $\n \in \ZZ^d$.

Next we choose a coarse Lyapunov weight $[\chi]$ of $Y_\irred$
and consider it as a coarse Lyapunov weight for $X$. 
Let $V_{\base,\AA}<\AA^m$ be the rational $\alpha$-invariant subspace
so that $Y_\base$ is the image of $V_{\base,\AA}$ modulo $\QQ$. 
We also set $W=W^{[\chi]}\cap V_{\base,\AA}$. Let $f _ W$ be a positive function on $W$ which is integrable with respect to $\mu _ x ^ W$ for every $x$ in a set of full measure as in \S\ref{sec:leafwiseXtilde}. We take
\[
\Omega = \left\{ [\nu]: \text{$\nu$ is a locally finite measure on $W$ with $\int f _ W \,d \nu < \infty$} \right\},
\]
where $[\nu]$ denotes the equivalence class of $\nu$ in the space of locally finite measures with respect to proportionality. One can equip $\Omega$ with the structure of a compact metric space in a standard way.
The map (defined for a.e.\ $x \in X$) that takes $x \in X$ and maps it to the proportionality class of its leafwise measure $[\mu_x^W]$ is, by \eqref{eq:conjugacy}, a factor map of the $\ZZ ^ d$-action $\alpha$ on $X$ to the action of $\ZZ ^ d$ on elements of~$\Omega$ 
by pushforward with respect to the linear action corresponding to $\alpha$ on $W$.
Taking the product of $\Omega$ with $X_\base$ we get a factor of $X$ 
and we apply Theorem~\ref{thm: shape} once
more over this factor $X_\base \times \Omega$ to obtain a constant $\kappa_{\widetilde{\mu},\Omega}\geq 0$ so that
\begin{equation}
    h_{\widetilde{\mu}}(\alpha^\n,W^{[\chi]}\mid X_\base \times\Omega)=
  \kappa_{\widetilde{\mu},\Omega}h_\lambda(\alpha_{Y_\irred}^\n,W^{[\chi]})>0 
\end{equation}
for all coarse Lyapunov weights $[\chi]$ and $\bn \in \ZZ^d$.

\begin{lemma}\label{two weights Lemma}
$Y_\irred$ has at least two linearly independent Lyapunov weights
\end{lemma}

\begin{proof}
  Since $Y_\irred$ is irreducible we may apply Proposition \ref{adelicirred}
 and describe the action on $Y_\irred$ using a global field $\KK$ and its elements.
 Also recall that the eigenspaces for $Y_\irred$ correspond to the completions of $\KK$. 
 
 Suppose in contradiction that $\alpha$ has no two linearly
 independent Lyapunov weights. Then every nonzero Lyapunov weight
 must be a multiple of $\chi$. We now define the hyperplane $H<\RR^d$ 
 as the kernel of $\chi$. 
 Suppose $\n\in\zd$ is close to $H$, i.e.\ satisfies
 $\chi'\cdot\n\in (-\epsilon,\epsilon)$ for all Lyapunov weights $\chi'$
 of $Y_\irred$ and some $\epsilon>0$ to be determined later. 
 For the algebraic number $\zeta_\n$ corresponding to $\n$ 
 this becomes the inequality 
 \begin{equation}\label{zeta inequality}
     e^{-\epsilon}<|\zeta_{\n}|_\sigma<e^\epsilon
 \end{equation}
 for all places $\sigma$ of $\KK$ (and $|\zeta_{\n}|_\sigma=1$ for places $\sigma$ lying over finite primes $p$ not in $S$, with $S$ as in \eqref{eq:define S}. 
 However, for small enough $\epsilon>0$ it follows from~\eqref{zeta inequality} applied to all finite places
 that $\zeta_\n$ must be an
 algebraic unit. Applying~\eqref{zeta inequality} also to all infinite places we get that this unit satisfies that all its real and complex embeddings have absolute
 value close to one. It follows from Dirichlet's unit theorem that
 $\zeta$ must be a root of unity and $\n\in H$. However, this implies that the action of $\alpha(\ZZ^d)$ on $Y_\irred$ is
 virtually cyclic. By the Jordan decomposition over $\QQ$, $X$ has a factor isomorphic to $Y_\irred$; the fact that $\alpha(\ZZ^d)$ on $Y_\irred$ is
 virtually cyclic now contradicts our standing assumption that $X$ has no virtually cyclic factors.
\end{proof}

\begin{proposition}\label{nexttolast}
 In fact $\kappa_{\widetilde{\mu},\Omega}=\kappa_\mu$ (hence $\kappa_{\widetilde{\mu},\Omega}>0$).
\end{proposition}

\begin{proof}
 Let $\chi'$ be a Lyapunov weight of $Y_\irred$
 that is linearly independent to $\chi$. The existence of $\chi'$ follows from Lemma~\ref{two weights Lemma}. Choose some $\n\in\zd$ 
 so that $\chi\cdot\n<0$ and $\chi'\cdot\n<0$. 
 The product structure of the leafwise measures in Theorem \ref{thm:prod} 
 (see also \cite[Cor.~8.6]{EL-Pisa})
 now implies that for a set $X'\subset X$ of full measure
 we have the following property: for any $x\in X'$ and any $w'\in W^{[\chi']}$
 with $x+w'\in X'$ we have $\left [\mu_{x+w'}^{W^{[\chi]}}\right ]=\left [\mu_{x}^{W^{[\chi]}}\right ]$.
 For a $\sigma$-algebra $\mathcal{A}$ subordinate to $W^{[\chi']}$ this means that,
 on the complement of a null set, all points in a given atom of $\mathcal{A}$ are contained
 in the same fiber of the factor map from $X$ to $\Omega$.
 In other words the map $x \mapsto \left (x,\left [\mu_x^{W^{[\chi]}}\right ]\right )$ from $X$ to $X\times\Omega$ maps $\sigma$-algebras
 subordinate to $W^{[\chi']}$ on $X$ to $\sigma$-algebras subordinate to $W^{[\chi']}$ on $X\times\Omega$.
 From this it follows that the leafwise measures for $X$ and for $X\times\Omega$
 with respect to the subgroup $W^{[\chi']}$ agree. In particular
 we have 
 \[
  h_{\widetilde\mu}(\alpha^\n,W^{[\chi']}\mid X_\base\times\Omega)
  = h_\mu(\alpha^\n, W^{[\chi']}\mid X_\base),
 \]
 which together with Theorem \ref{thm: shape} proves the proposition.  
\end{proof}

We continue working under the assumptions stated at the beginning of \S\ref{sec: rigidity implies invariance}.

\begin{corollary}\label{reallylast}
For any subset $X''\subset X$ of full measure 
there exist $x\in X''$ and a nonzero $v\in W$
with $x+v\in X''$ and $\mu_x^{W}\propto\mu_{x+v}^{W}$.
\end{corollary}

\begin{proof}
 By Proposition \ref{nexttolast} we have $\kappa_{\widetilde{\mu},\Omega}>0$.
 We now apply this to the coarse Lyapunov weight $[\chi]$ 
 and the subgroup $W=W^{[\chi]}\cap V_{\base,\AA}$ that was used to define
 the factor $\Omega$.
 Choose $\n\in\zd$ with $\chi\cdot\n<0$.
 It now follows from the definition of $\kappa_{\widetilde{\mu},\Omega}$
 in Theorem \ref{thm: shape} that
 \[
 h_{\widetilde{\mu}}(\alpha^\n,W^{[\chi]}\mid X_\base \times\Omega)=
  \kappa_{\widetilde{\mu},\Omega}h_\lambda(\alpha_{Y_\irred}^\n,W^{[\chi]})>0.
 \]
  We also note that $h_{\widetilde{\mu}}(\alpha^\n,W\mid X_\base \times\Omega)=h_{\widetilde{\mu}}(\alpha^\n,W^{[\chi]}\mid X_\base \times\Omega)$ (c.f.\ e.g.\ the first lines in the proof of Theorem~\ref{thm:bound-contribution}).
 We note that positive entropy contribution 
 as in \eqref{eq:entropygrowthrate} shows in particular that the leafwise
 measure $\widetilde\mu_x^W$ gives zero mass to $0\in W$. 
 
 By the characterizing properties of leafwise measures in terms
 of $W$-subordinate $\sigma$-algebras the fact that the leafwise
 measure $\widetilde\mu_x^W$ gives zero mass to $0\in W$ 
 implies that for any subset $X''\subset\widetilde{X}$ of full measure for $\widetilde{\mu}$ 
 \begin{equation}\label{recurrence equation}
     \text{there exist $x\in X''$
 and $v\in W\setminus\{0\}$ so that $x+v\in X''$}.
 \end{equation} 
 
 Since the action of $W$ 
 on $\widetilde{X}=X\times\Omega$
 was defined to be trivial on $\Omega$ and $\widetilde{\mu}$ was defined as the push
 forward of $\mu$ under the map  $x\mapsto(x,[\mu_x^W])$, equation \eqref{recurrence equation} translates to the following statement:
 for every subset $X''\subset X$ of full measure for $\mu$ there exists $x\in X''$
 and $v\in W\setminus\{0\}$ with $x+v\in X''$ and $[\mu_x^W]=[\mu_{x+v}^W]$. However, 
 this is precisely the claim in the corollary.
\end{proof}

\begin{proof}[Proof of Theorem \ref{thm:main2}]
We first show how the statement in Corollary \ref{reallylast}
implies invariance of $\mu$ under translation by all elements of a nontrivial adelic subgroup.
We note that by \eqref{eq:shifting}
the claim in Corollary \ref{reallylast} amounts to saying
that for any set $X''\subseteq X'$ of full measure
there exists some $x\in X''$ and $v\in W\setminus\{0\}$
with $\mu_x^W+v\propto\mu_x^W$. Moreover, by our discussion
concerning \eqref{eq:affineinvariance1}--\eqref{eq:affineinvariance2}
there exists $X''$ of full measure so that
for any $x\in X''$ there is some $v\in W$ with 
$\mu_x^W+v=\mu_{x}^W$
(see also \cite[Lemma 5.10]{Einsiedler-Katok-II}). 
Hence Corollary \ref{reallylast}
implies that for all $x\in X''$ the 
closed subgroup 
\[
 W_x=\bigl\{v\in W: \mu_x^W+v=\mu_{x}^W\bigr\}
\]
is nontrivial. We define $\widetilde{W}_x$ as the 
maximal $S$-linear subgroup of $W_x$. We will show
below (using the equivariance formula \eqref{eq:conjugacy} and Poincar\`e 
recurrence) that almost surely $\widetilde{W}_x$ is nontrivial. 
In fact \cite[Prop.~6.2]{Einsiedler-Katok-II}
shows that $W_x=\widetilde{W}_x$ is itself $S$-linear, 
at least if the action is semisimple. 

We define the dimension $D_x$ of $\widetilde{W}_x$ as the sum 
of the dimensions of the maximal subspaces over $\QQ_\sigma$
contained in $\widetilde{W}_x$ for all $\sigma\in S$. 
Even though $\widetilde{W}_x$ may not be normalized
by $\alpha$ the equivariance formula \eqref{eq:conjugacy}
implies that both $W_x$ and $\widetilde{W}_x$
are equivariant for the action. Hence the dimension 
of $\widetilde{W}_x$ is invariant under $\alpha$. 
Therefore $D_x$ is constant (say equal to $D$) for a.e.\ $x$. 

We claim that Corollary \ref{reallylast} implies almost surely 
that $\widetilde{W}_x$ is nontrivial, 
or equivalently that $D_x \geq 1$ a.s. 
For this we apply Luzin's theorem
and let $K\subset X'$ be a compact subset of measure close
to $1$, on which all almost
sure properties of the leafwise measures hold, 
and on which the map $x\in K\mapsto\mu_x^W$ is continuous.
To obtain the almost sure conclusion
we apply the following argument for an increasing sequence of 
such Luzin sets that cover almost all of the space.

By Poincar\'e recurrence we see that for almost every $x\in K$ 
there exists two increasing sequences $n^-_k,n^+_k\in\NN$
with $T^{-n^-_k}x, T^{n^+_k}x\in K$ converging to $x$ as $k\to\infty$. 
Suppose now $v\in W_x\setminus\{0\}$ for one such $x$ so 
that $\mu_x^W+v=\mu_x^W$.
If $v$ has a nontrivial real component,
we are going to use the sequence $n^+_k$. Indeed applying 
\eqref{eq:conjugacy} for these powers we see that
the leafwise measure at $T^{n^+_k}x$
has translation invariance under $\ZZ (T^{n^+_k}v)$.
Note that $T^{n^+_k}v$ converges 
to $0$ as $k\to\infty$. However, since the unit ball in $\RR^m$ is compact
we may choose a subsequence and assume in addition that 
the direction of $T^{n^+_k}v$ converges in projective space
to $\RR \widetilde{v}$ for some $\widetilde{v}\neq 0$. 
Note that this implies that the subgroups $\ZZ (T^{n^+_k}v)$
converge in the Chabauty topology to $\RR\widetilde{v}$. 
Combined
with continuity of the leafwise measures restricted to $K$
this now implies that $\mu_x^W$ is invariant under translation 
by $\RR\widetilde{v}$, which implies $\widetilde{v}\in\widetilde{W}_x$ and $D_x\geq 1$
as desired.

So suppose now $v$ has trivial real component and let us write $v_p$
for the $p$-adic component of $v$ for all $p\in S$.
In this case the invariance of $\mu_x^W$ under $\ZZ v$
implies invariance under $\overline{\ZZ v}$, which by the 
Chinese Remainder Theorem is the product of the compact 
subgroups $\ZZ_p v_p$ for all primes $p\in S$. To obtain a
nontrivial $S$-linear subgroup we fix a prime $p\in S$
with $v_p\neq 0$ and  are going to use the sequence $n^-_k$.
Indeed as above $T^{-n^-_k}x\in K$ has invariance under $\ZZ_p T^{-n^-_k}v_p$, where $T^{-n^-_k}v_p$ diverges to infinity
but projectively converges to some $\QQ_p \widetilde{v}$. By continuity
of the leafwise measures on $K$ this again implies $\widetilde{v}\in\widetilde{W}_x$ 
and $D_x\geq 1$ as desired.

We define the (measurable) factor map
\[
 \phi:x\mapsto \widetilde{W}_x,
\]
from $X$ to the space  $\mathcal F$ of closed subgroups of $W$
in the Chabauty topology.
We also decompose $\mu$ over $\phi$, i.e.\ we consider 
the conditional measures $\mu_x^{\phi^{-1}\mathcal{B}_{\mathcal{F}}}$.
By the compatibility condition \eqref{eq:shifting} and the definition 
of $\widetilde{W}_x$ we have $\phi(x)=\phi(x+v)$ whenever 
$v\in W$ and both $x,x+v$ belong
to $X''$. In particular this shows, for a $W$-subordinate
$\sigma$-algebra $\mathcal A$, that $\mathcal A$ 
contains $\phi^{-1}\mathcal{B}_{\mathcal{F}}$ modulo null sets. Hence, 
after fixing a choice of the conditional measures $\mu_x^{\mathcal{A}}$
we have that the (doubly) conditional measure $\left(\mu_{x_0}^{\phi^{-1}\mathcal{B}_{\mathcal{F}}}\right)^{\mathcal A}_{x}$ agrees with $\mu_x^{\mathcal{A}}$ for $\mu_{x_0}^{\phi^{-1}\mathcal{B}_{\mathcal{F}}}$-almost every $x$ and $\mu$-almost every $x_0$.
This in turn implies for the leaf-wise measures of $\mu$ 
and $\mu_{x_0}^{\phi^{-1}\mathcal{B}_{\mathcal{F}}}$ with respect to $W$
by the characterizing properties that
\begin{equation}\label{eq: conditional and leafwise}
  \bigl(\mu_{x_0}^{\phi^{-1}\mathcal{B}_{\mathcal{F}}}\bigr)_x^W=\mu_x^W
\end{equation}
for $\mu_{x_0}^{\phi^{-1}\mathcal{B}_{\mathcal{F}}}$-almost every~$x$
and for $\mu$-almost every $x_0$. 

Fix some $x_0$. 
Then essentially by definition of the factor $\phi$ we have that $\widetilde W_x=\widetilde W_{x_0}$
for $\mu_{x_0}^{\phi^{-1}\mathcal \mathcal{B}_{\mathcal{F}}}$-almost every $x$. In other words, by definition of $\widetilde W_x$ and \eqref{eq: conditional and leafwise}, we have that for $\mu_{x_0}^{\phi^{-1}\mathcal \mathcal{B}_{\mathcal{F}}}$-almost every $x$, the leafwise measures $\bigl(\mu_{x_0}^{\phi^{-1}\mathcal{B}_{\mathcal{F}}}\bigr)_x^W$ are invariant under the group $\widetilde W_{x_0}$, hence
by the standard properties of leaf-wise measures that
$\mu_{x_0}^{\phi^{-1}\mathcal \mathcal{B}_{\mathcal{F}}}$ is itself
invariant under the action of $\widetilde{W}_{x_0}$ by translations.
Note that, as far as we know at this point, the group $\widetilde{W}_x$
may depend on $x$, and hence we have not established the invariance of $\mu$ itself under any translations yet. 

As $X=\AA^m/\QQ^m$ is an abelian group and $\QQ^m$ acts trivially
this implies that $\mu_{x_0}^{\phi^{-1}\mathcal \mathcal{B}_{\mathcal{F}}}$
is in fact invariant under the closure $G_x$
of $\widetilde{W}_{x_0}+\QQ^m$ in $X$. 
Since $\widetilde{W}_{x_0}$, as an $S$-linear subgroup, is invariant under
multiplication by $\QQ$, we have that its anihilator $G_x^\perp$ in the Pontryagin dual $\QQ^m$ to $X$ is a vector space over $\QQ$. Hence $G_x$ is an adelic subgroup of $X$.

The equivariance formula \eqref{eq:conjugacy}
implies a similar equivariance formula for $W_x$, for $\widetilde{W}_x$,
and hence also for $G_x$.
In other words, $x\mapsto G_x$ is a (measurable) factor map for $\alpha$
with values in the countable set of all adelic subspaces of $X_m$.
Hence there exists an adelic subspace $G$ so that $G_x=G$
on a set of positive measure. By Poincar\'e recurrence we may conclude that
there is a finite index subgroup $\Lambda$ of $\ZZ^d$ so that $\alpha(\Lambda)$ normalizes $G$.
However, the assumption in Theorem \ref{thm:main2}
now implies that $G$ is actually normalized
by $\ZZ^d$. Ergodicity under $\alpha$ 
now implies that $G_x=G$ almost surely. Therefore, $\mu$
is invariant under $G$. 
\end{proof}

\section{Disjointness}\label{sec: disj} 
We will deduce in this section rigidity of joinings and in particular prove
Corollary \ref{thm: joining}.
We say a measure on a solenoid $X$ is {\em homogeneous} if it is the $Y$-invariant measure on a
coset $Y+y$ of a closed subgroup $Y<X$.

\begin{proposition}\label{pro: joining}
 Let $d,r\geq 2$ and let $\alpha_i$ be $\zd$-actions on solenoids $X_i$ (equipped with Haar measure $\lambda_{X_i}$)
 with no virtually-cyclic factors for $i=1,\ldots,r$.
 Suppose there exists a nontrivial joining $\mu$ between
 $\alpha_i$ for $i=1,\ldots,r$.
 Then there exists a finite index subgroup $\Lambda\subseteq\ZZ^d$
 such that there exists also a homogeneous nontrivial joining
 $\lambda_G$ between $\alpha_{i,\Lambda}$ for $i=1,\ldots,r$.
In fact, the subgroup $G$ can be chosen to be any of the groups in the conclusion of Theorem~\ref{thm: main}
 when applied to an ergodic component of the positive entropy measure $\mu$ on $X_1\times\dots\times X_r$.
\end{proposition}

\begin{proof}
 Let $X=X_1\times\dots\times X_r$
 and $\alpha=\alpha_1\times\cdots\times\alpha_r$ be the product
 group and action, and let $\mu$ be a nontrivial joining.
 Without loss of generality we can assume $\mu$ is $\alpha$-ergodic
 (since a.e.\ ergodic component of a joining is again a joining).
 We apply Theorem \ref{thm: main} and obtain a finite index subgroup $\Lambda\subset\zd$
 and an $\alpha_\Lambda$-invariant closed subgroup $G=G_1\subset
 X$. We claim that the Haar measure of $G$ is an homogeneous
 nontrivial joining for the $\Lambda$-action.

 To see this, let $\mu_j$ for $j=1,\ldots,J$ be as in Theorem \ref{thm: main} and let
 $\pi_i:X\rightarrow X_i$ for $i=1,\ldots,r$ be the coordinate projection map.
 First notice, that the Haar measure of $G$ cannot be the trivial
 joining. If it were, it would follow that $G=X$,
 $\mu_1=\lambda_X$, $\mu_j=(\alpha^\n)_*\mu_1=\lambda_X$ for all $j$ and some $\n$ (that depends on $j$),
 and so $\mu=\lambda_X$ would be the trivial joining. To show the claim,
 we only need to prove that $\pi_i(G)=X_j$ for $i=1,\ldots,r$.

 Fix some $i$. Clearly $\tilde\lambda_{j}=(\pi_i)_*\mu_j$
 defines a measure on $X_i$ that is invariant under $\alpha_{i,\Lambda}$.
 By assumption $\lambda_{X_i}=\frac{1}{M}(\tilde\lambda_{1}+\cdots+\tilde\lambda_{J})$.
 Note that $\alpha_{i,\Lambda}$ acts ergodically on $X_i$ with respect to $\lambda_{X_i}$ by
 the assumption that there are no virtually-cyclic factors. Therefore,
 $\tilde\lambda_{j}=\lambda_{X_i}$ for all $j$, in other words $\mu_1$ is a joining.
 Consider now the group $Y=X_i/\pi_i(G)\cong X/(G+\ker\pi_i)$
 endowed with the measure $\nu$ induced by $\mu_1$. However, by the above
 the projection of $\mu_1$ to $X_i$ is
 $\lambda_{X_i}$, and so $\nu=\lambda_Y$. Since $Y$ is a
 factor of $X/G$, the entropy (with respect to $\nu$)
 of every element on $Y$ of the action must vanish.
 However, the action on $X_i$
 contains elements with completely positive entropy. Therefore,
 $Y=\{0\}$ and $\pi_j(G)=X_i$ for all $i$.
\end{proof}

\begin{proof}[Proof of Corollary \ref{thm: joining}, simpler case]
 Assume first that $\alpha_1$ and $\alpha_2$ are
totally irreducible not virtually-cyclic  actions and let
 $\alpha=\alpha_1\times\alpha_2$.
 Suppose $\alpha_{1,\Lambda}$ and $\alpha_{2,\Lambda}$
 are algebraically weakly isomorphic for some finite index subgroup
 $\Lambda\subset\zd$. So there exists a finite-to-one algebraic factor map
 $\varphi:X_1\rightarrow X_2$ for the two subactions. Clearly, the
 graph $G$
 of $\varphi$ is $\alpha_\Lambda$-invariant and so is its Haar
 measure $\lambda_G$.
 The average $\mu$ over the elements $\alpha^\n_*\lambda_G$
 in the (finite) orbit of $\lambda_G$ under the action of $\alpha$
 defines a nontrivial joining between $\alpha_1$ and $\alpha_2$.

 Let $\mu$ be a joining between $\alpha_1$ and $\alpha_2$. We have
 to show that either $\mu=\lambda_{X_1}\times\lambda_{X_2}$ is the trivial
 joining, or find a finite index subgroup $\Lambda\subset\zd$
 such that $\alpha_{1,\Lambda}$ and $\alpha_{2,\Lambda}$ are algebraically weakly isomorphic.
 Assume that $\mu$ is not the trivial joining, then there exists a
 finite index subgroup $\Lambda$ and a homogeneous nontrivial joining $\lambda_G$ between
 $\alpha_{1,\Lambda}$ and $\alpha_{2,\Lambda}$ by Proposition
 \ref{pro: joining}. Here $G\subset X$ is a proper closed subgroup with
 $\pi_1(G)=X_1$ and $\pi_2(G)=X_2$.

 Next we study the factors $X'_1=X_1/\{x_1:(x_1,0)\in G\}$
  and $X'_2=X_2/\{x_2:(0,x_2)\in G\}$.
 Suppose $X'_1$ is trivial, then
 $X_1\times\{0\}\subset G$ and $\pi_2(G)=X_2$ implies that
 $G=X_1\times X_2$, a contradiction to $G$ being proper.

 So assume now $X'_1$ and $X'_2$ are nontrivial, and let
 $G'\subset X'_1\times X'_2$ be the subgroup defined by $G$.
 Then $G'\cap \bigl(X'_1\times\{0\}\bigr)=G'\cap\bigl(\{0\}\times X'_2\bigr)=\{0\}$,
 so $G'$ is the graph of an isomorphism between $X_1'$ and $X_2'$.
 Since $G'$ is closed, compactness shows the
 isomorphism is continuous.
 Since all of the above were invariant subgroups, it follows that
 $\alpha_{1,\Lambda}$ and $\alpha_{2,\Lambda}$ have a common
 nontrivial factor $X_1'\cong X_2'$.
 By assumption $\alpha_{1,\Lambda}$ and $\alpha_{2,\Lambda}$ are irreducible,
 so the kernel of the above factor maps must be finite. Let $N$ be the order
 of the kernel, then multiplication by $N$ defines a factor map $ \psi_N$ from
 $X_1$ to $X_1$ that can be extended to a factor map $ \psi_N$ from $X_1'$ to $X_1$,
 i.e.\ the two actions on $X_1$ and $X_1'$ are weakly algebraically isomorphic.
\end{proof}

\begin{proof}[Proof of Corollary \ref{thm: joining}, general case]
 Let $X=\prod_{=1}^r X_j$
 and $\alpha=\alpha_1\times\cdots\times\alpha_r$ be the product
 group and action.
 Suppose $j\neq k\in\{1,\ldots,r\}$ and
 $\alpha_{j,\Lambda}$ and $\alpha_{k,\Lambda}$
 have a common nontrivial factor $\beta$ on a solenoid $Y$,
 where $\Lambda\subset\zd$ is a finite index subgroup. Let
 $\varphi_j:X_j\rightarrow Y$ and $\varphi_k:X_k\rightarrow Y$
 be the corresponding group homomorphisms. Then $G=\{x\in
 X:\varphi_j(x_j)=\varphi_k(x_k)\}$ is a nontrivial closed
 $\alpha_\Lambda$-invariant subgroup of $X$ such that
 $\pi_i(G)=X_i$ for $i=1,\ldots,r$. The Haar measure
 $\lambda_G$ on $G$ has finite orbit under $\alpha$, and the average $\mu$
 over the elements in the (finite) orbit of $\lambda_G$ is a nontrivial joining.

 Suppose now that $\mu$ is a nontrivial joining
 between $\alpha_i$ for $i=1,\ldots,r$, and apply
 Proposition \ref{pro: joining}.
 We obtain a finite index subgroup $\Lambda\subset\zd$
 and an $\alpha_\Lambda$-invariant proper closed subgroup $G<
 X$ that satisfies $\pi_i(G)=X_i$ for $i=1,\ldots,r$.

 Next we factor $X_i$ by the subgroup
 $$H_i=\pi_i\Bigl( G \cap \bigl(\{0\}^{i-1}\times X_i\times \{0\}^{r-i}\bigr)\Bigr)$$
 to get $X_i'=X_i/H_i$ for $i=1,\ldots,r$ and the factor $X'=\prod_iX_i'$ of $X$.
 If $X'=\{0\}$, then $G=X$ which contradicts $G$ being a proper subgroup.

 So assume that $X'$ is nontrivial, and therefore infinite.
 Let $G'< X'$ be the image of $G$.
 Clearly $\pi_i(G')=X_i'$ for $i=1,\ldots,r$. Let $i$ be minimal such that
 $H=G'\cap Z_i$ is infinite,
 where $Z_i=X_1'\times\cdots\times X_i'\times\{0\}^{r-i}$.
 Then the kernel $\{x\in H:x_i=0\}$ of $\pi_i|_{H}$
 is finite, and so $\pi_i(H)$ is an infinite closed $\alpha_\Lambda$-invariant subgroup of
 $X_i'$. Let $\beta$ be an irreducible component of $\alpha_\Lambda|_H$.
 Since $\pi_i|_H$ is finite-to-one, $\beta$ is also an irreducible component of
 $\alpha_{i,\Lambda}|_{\pi_i(H)}$. Since an irreducible component
 of a subgroup is also an irreducible component of the whole
 group, we see that $\alpha_{\Lambda}|_H$ and
 $\alpha_{i,\Lambda}$ share $\beta$ as irreducible component.
 By construction of $X'$ we have
 $\{z\in H: z_k=0$ for all $k\neq i\}=\{0\}$. Therefore $H$ is
 isomorphic to a subgroup of $X_1'\times\cdots\times X_{i-1}'$.
 We conclude that there exists $j<i$ such that
 $\alpha_{i,\Lambda}$ and $\alpha_{j,\Lambda}$ have $\beta$ as a
 common factor.
\end{proof}

\section{Invariant \texorpdfstring{$\sigma$}{sigma}-algebras and measurable factors}\label{sec: inv_algebra}

In addition to the characterization of factors stated in Corollary \ref{thm: algebra},
we prove in this section a generalization of the isomorphism rigidity \cite{Katok-Katok-Schmidt} for
higher rank actions. Indeed we characterize when two actions by automorphisms of
solenoids have a common measurable factor.

Suppose $\alpha_1,\alpha_2 $ are $\zd $-actions by automorphisms
of the solenoids $X_1$ and $X_2$. Let $\Gamma_1$ respectively
$\Gamma_2$ be finite groups of affine automorphisms of $X_1$ and
$X_2$ that are normalized by the respective actions. We say that the two
factors of $X_1 $ and $X_2 $ arising from $\Gamma_1$ respectively
$\Gamma_2$ are {\em isomorphic} if there exists an affine isomorphism
$\Phi:X_1\rightarrow X_2$ such that
$\Phi\circ\alpha_1^\n\circ\Gamma_1=\alpha_2^\n\circ\Gamma_2\circ\Phi$
for all $\n\in\zd$.

We claim that this is essentially the only way common
measurable factors of higher rank actions on solenoids can arise.

\begin{corollary}[Classification of common factors]\label{theorem about common factors}
 Let $d\geq 2$ and let $\alpha_1,\alpha_2$ be $\zd $-actions by automorphisms
 of the solenoids $X_1$ and $X_2$ without virtually cyclic factors.
 Suppose $\alpha_1 $ and $\alpha_2$ have a common
 measurable factor. Then there exist
 closed invariant subgroups $X'_1\subseteq X_1$
 and $X_2'\subseteq X_2$, finite groups $\Gamma_1$ of affine automorphisms of
 $X_1/X_1'$ 
 and $\Gamma_2$ of affine automorphisms of $X_2/X_2'$ that are normalized by the
 corresponding actions such that this common measurable factor
 can be described alternatively as the factor of $X_1/X_1'$ by
 the orbit equivalence relation of $\Gamma_1$, or similarly
 using $X_2/X_2'$ and $\Gamma_2$.
 Moreover, the isomorphism between these two realizations
 of the factor is algebraic in the following sense:
 there exists an affine isomorphism
 $\Phi:X_1/X_1'\rightarrow X_2/X_2'$ such that
\begin{equation}\label{factor classification equation}
      \Phi\circ\alpha^\n_{X_1/X_1'}\circ\Gamma_1=
    \alpha^\n_{X_2/X_2'}\circ\Gamma_2\circ\Phi
\end{equation}
 for all $\n\in\zd$.
\end{corollary}

Note that we have more rigidity for the factors then we had for joinings in Corollary \ref{thm: joining}. In particular, there is no need to consider finite index subactions in order to classify when two $\zd$ actions on tori have a common factor; we illustrate this point with the following example.

\begin{example}\label{example:twisted}
Let $\alpha_1 $ be the $\ZZ^2$-action by automorphisms of the
solenoid $X_1$ dual to ${\ZZ[1/2,1/3]}$ generated by multiplication by $2$ and by $3$ on $X_1$.\footnote{I.e.\ the $\alpha_1$ action on $X_1$ is the the invertible extension of the $\times 2, \times 3$ action on $\TT$.} We define a $\ZZ^2$-action $\alpha_2 $ on
$X_2=X_1^2 $ by
\begin{eqnarray*}
\alpha_2^{\be_1}(x_1,x_2)&=&(-2x_2,2x_1),\\
\alpha_2^{\be_2}(x_1,x_2)&=&(3x_1,3x_2)\mbox{ for }(x_1,x_2)\in X_2.
\end{eqnarray*}
Then $\alpha_2^{4\be_1}(x_1,x_2)=(16x_1,16x_2)$ for $(x_1,x_2)\in
X_2$ and so the restriction $\alpha_{2,\Lambda}$ of $\alpha_2$ to
$\Lambda=(4\ZZ)\times\ZZ$ is identical to the action
$\alpha_{1,\Lambda}\times\alpha_{1,\Lambda}$ on $X_2=X_1^2$. By
Theorem~\ref{thm: joining} $\alpha_1$ and $\alpha_2 $ are not
disjoint. In fact, let
\[
 Z=\{(x_1,(x_1,x_2):x_1,x_2\in X\}\subseteq X_1\times X_2.
\]
Then $\pi_1(Z)=X_1$ and $\pi_2(Z)=X_2$. Therefore, the Haar
measure $m_Z$ of $Z$ satisfies $(\pi_1)_*m_Z=m_{X_1}$ and
$(\pi_2)_*m_Z=m_{X_2}$. Since $Z$ is only invariant
under $(\alpha_1\times\alpha_2)_\Lambda$, $m_Z$ is not a
joining. However, it is easy to check that
$\mu=\frac{1}{4}\sum_{j=0}^3(\alpha_1\times\alpha_2)^{j\be_1}m_Z$
is a nontrivial joining between $\alpha_1 $ and $\alpha_2 $.

We note that $\alpha_1$ and $\alpha_2$ are both irreducible
actions.
Suppose now that $\alpha_1 $ and $\alpha_2 $ have a common
measurable factor. By Corollary~\ref{theorem about common factors}
there exist closed invariant subgroups $X'_1\subseteq X_1$ and
$X_2'\subseteq X_2$ such that $X_1/X'_1$ and $X_2/X_2'$ are
isomorphic as groups. By irreducibility $X'_1$ and $X_2'$ are
either finite or everything. However, if $X'_1$ and $X_2'$ are
finite, then $X_1/X_1'$ is still one dimensional while $X_2/X_2'$
is still two dimensional, and so these groups cannot be
isomorphic. Therefore, $X'_1=X_1$, $X_2'=X_2$, and the common
measurable factor has to be the trivial factor.
\end{example}

Before we start with the proofs of Corollary \ref{thm:
algebra} and Corollary \ref{theorem about common factors} we recall some basic
facts about conditional measures and the
construction of the relatively independent joining.

\subsection{The relatively independent joining}
Let $\alpha_1,\alpha_2,\beta$ be $\zd$-actions on the
standard Borel probability spaces $(X_1,\cB_1,m_1)$, $(X_2,\cB_2,m_2)$, and $(Y,\cB_Y,\rho)$
respectively, and let $ \psi_1:X_1\rightarrow Y$ and
$ \psi_2:X_2\rightarrow Y$ be factor maps.  Then
$\cA_1= \psi_1^{-1}\cB_Y$ and $\cA_2= \psi_2^{-1}\cB_Y$ are
invariant $\sigma $-algebras with conditional measures
$m_{1,x}^{\cA_1}$ for $x\in X_1$ and $m_{2,x}^{\cA_2}$
for $x\in X_2$. By the basic properties of conditional measures, there is some null set $N_1 \subset X_1$ so that
$m_{1,x}^{\cA_1}=m_{1,x'}^{\cA_1}$ for every $x,x'\in
X_1\setminus N_1$ with $ \psi_1(x)= \psi_1(x')$, and so we can remove a nullset from
$X_1$ and $ Y $ and define
$m_{ \psi^{-1}_1y}=m_{1,x}^{\cA_1} $ for
$x\in \psi^{-1}_1y$. We do this similarly for $\cA_2$ to define
$m_{ \psi^{-1}_2y} $.  The relatively independent joining
$m_1\times_{(Y,\rho)}m_2$ of $m_1 $ and
$m_2 $ over the common factor $(Y,\rho)$ is defined by
\begin{equation}\label{definition of joining}
 m_1\times_{(Y,\rho)}m_2=
   \int_Y m_{ \psi^{-1}_1y}\times m_{ \psi^{-1}_2y}\od\!\rho(y).
\end{equation}
It is well known (and easy to verify directly) that $m_1\times_{(Y,\rho)}m_2 $ projects to
$m_1$ on $X_1$ and to $m_2$ on $X_2$ and that $m_1\times_{(Y,\rho)}m_2 $ is invariant under
$\alpha_1\times\alpha_2$ (hence is a joining between
$\alpha_1$ and $\alpha_2 $).  Furthermore, the relatively independent joining $ m_1\times_{(Y,\rho)}m_2$ gives full measure to the set
\begin{equation}
    \label{definition of D_Y}
    D_{Y}=\{(x_1,x_2) : \psi_1(x_1)=\psi_2(x_2)\}\subset X_1 \times X_2
\end{equation}
and moreover
\[
   \psi_1^{-1}C\times X_2,\  X_1\times  \psi_2^{-1}C,\mbox{ and }
  \psi_1^{-1}C\times  \psi_2^{-1}C
\]
are equal up to a $m_1\times_{(Y,\rho)}m_2$-nullset for any $C\in\cB_Y$.

\subsection{Proofs of Corollary \ref{thm:
algebra} and Corollary \ref{theorem about common factors}}

Let $\alpha_1,\alpha_2, X_1, X_2$ be as in Corollary~\ref{theorem
about common factors}, and suppose the $\zd$-action $\beta$ on
$(Y,\cB_Y,\rho)$ is a common factor of $\alpha_1 $ and $\alpha_2$. Let
$ \psi_1$ and $ \psi_2$ be the corresponding factor maps,
$\cA_1= \psi_1^{-1}\cB_Y$ and $\cA_2= \psi_2^{-1}\cB_Y$ the
corresponding invariant $\sigma$-algebras, and let
$\nu=\lambda_{X_1}\times_{(Y,\rho)}\lambda_{X_2}$ be the
relatively independent joining.

The main idea for the proof is to use Theorem \ref{thm: main} to
study $\nu$. In the following example we see how the algebraic
construction of the factor is encoded in the relatively
independent joining, and that the latter does not have to be
ergodic.

\begin{example} 
 Let $\alpha $ be a $\zd$-action on a solenoid $X $,
 and let $\cA=\cB_X^{\{\Id,-\Id\}}$ be the
 $\alpha$-invariant $\sigma$-algebra of measurable subsets $A$ satisfying $A=-A$.
 Then the relatively independent joining
 $\nu=\lambda_X\times_\cA \lambda_X$ of $\lambda_X$
 over the factor described by $\cA$
 is $\nu=\frac{1}{2}(\lambda_D+\lambda_{D_-})$,
 where $D=\{(x,x):x\in X\}$ and $D_-=\{(x,-x):x\in X\}$.
 The ergodic components of $\nu$ are
 $\lambda_D$ and $\lambda_{D_-}$.
\end{example}

This example shows that for the relatively independent joining
over a factor we have to take ergodic components in order to
apply Theorem \ref{thm: main}. However, in general we also
have to take ergodic components as in Theorem \ref{thm: main}
with respect to a finite index subgroup to
obtain measures invariant under translation by elements of certain
subgroups.

\begin{lemma}\label{construction of X_1'}
The set
\[
 X_1'=\{x'\in X_1: \psi_1(t)= \psi_1(t+x')
  \text{ for $\lambda_{X_1}$-a.e.~$t\in X_1$}\}
\]
is a closed $\alpha$-invariant subgroup of $X_1$. Furthermore,
$ \psi_1$ descends (on the complement of a nullset) to a well-defined
factor map from $X_1/X_1'$ to $Y$.
\end{lemma}

\begin{proof}
Since $\cA_1= \psi_1^{-1}\cB_Y$ is countably generated, it is
easily seen that 
\[
X_1=\{x'\in X_1: \lambda_{X_1}(A \vartriangle (A+x'))=0\quad\text{for all $A\in\cA_1$}\}.
\]
Furthermore,
it is well known that $\lambda_{X_1}(A\vartriangle (A+x'))$
depends continuously on $x'\in X$ for any $A\in\cB_{X_1}$. From
this it follows that $X_1'$ is a closed subgroup.  Moreover,
$X_1'$ is $\alpha_1$-invariant since $\cA_1$ is invariant (equivalently, since $ \psi_1$ is a
factor map).

Note that $ \psi_1(t)= \psi_1(t+x')$ for
$\lambda_{X_1}\times \lambda_{X_1'}$-a.e.\ $(t,x')\in X_1\times
X_1'$ by definition of $X_1'$ and Fubini's theorem. Therefore, for
$ \lambda_{X_1}$-a.e.\ $t\in X_1$ we know that
$ \psi_1(t)= \psi_1(t+x')$ for $ \lambda_{X_1'}$-a.e.\ $x'\in X_1'$.
Therefore, $ \psi_1(t)$ is (outside some nullset) independent of
the representative $t\in X_1$ of the coset $t+X_1'$ which implies
the second part of the lemma.
\end{proof}

By Lemma \ref{construction of X_1'} we can replace $X_1$ by
$X_1/X_1'$ and similarly $X_2$ by $X_2/X_2'$ for the remainder of
the proof. Hence we assume that
\begin{equation}\label{no translation}
 X'_1=\{x'\in X_1: \psi_1(t)= \psi_1(t+x')
  \text{ for $\lambda_{X_1}$-a.e.~$t\in X_1$}\}=\{0\}
\end{equation}
and similarly for $X_2$, and will show the existence of finite groups
$\Gamma_1$ and $\Gamma_2$ of affine automorphisms of $X_1$ and $X_2$, and the
existence of the affine isomorphism $\Phi:X_1\rightarrow X_2$ satisfying \eqref{factor classification equation}.

Since Theorem \ref{thm: main} assumes ergodicity, we need to study
the ergodic components of the relatively independent joining. The
following easy consequence
of the definition of the ergodic decomposition
gives all the properties we will need for the ergodic components.

\begin{lemma}\label{ergodic component}
 Almost every ergodic component $\mu$ of the relatively
 independent joining $\nu$ of $X_1$ and $X_2$ over $Y$ is still a joining between
 $\alpha_1$ and $\alpha_2$ that satisfies $\mu(D_Y)=1$ where $D_Y$
 is defined as in \eqref{definition of D_Y}.
\end{lemma}

In the next lemma we analyze what translation invariance for a
measure $\mu$ as above tells us about the factor maps.

\begin{lemma}\label{second invariance}
Let $\mu $ be a measure on $X_1\times X_2 $ that projects to the Haar measure
$ \lambda_{X_1}$ respectively $ \lambda_{X_2}$ and satisfies
$\mu(D_Y)=1$ where $D_Y$
 is defined in \eqref{definition of D_Y}. Suppose $\mu $ is translation invariant under elements
of $G\subset X_1\times X_2 $ and that $G$ projects surjectively to
$X_1$ and $X_2$.
 Then $G$ is the graph of a group isomorphism $\phi:X_1\rightarrow X_2$.
 Moreover, there exists some $w_\phi\in X_2$
 such that the affine isomorphism $\Phi(x)=\phi(x)+w_\phi$
 satisfies
 \begin{equation}\label{definition of Phi}
   \psi_1(x)= \psi_2(\Phi(x))\mbox{ for a.e.\ }x\in X_1
 \end{equation}
 and that $\mu $ is the Haar measure of the graph of $\Phi$.
 The element $w_\phi\in X_2$ is uniquely determined by $\phi$
 and \eqref{definition of Phi}.
\end{lemma}

\begin{proof}
 Suppose $(x',0)\in G$, \  $C\in\mathcal{B}_Y$, and $A= \psi_1^{-1}C\in\cA_1$.
 Then $A\times X_2= X_1\times \psi_2^{-1}C$ (modulo $\mu$) and
 \begin{align*}
   (A+x')\times X_2&=A\times X_2+(x',0)\\
   &=X_1\times \psi_2^{-1}C+(x',0)&&\text{(modulo $\mu$)}\\
   &=X_1\times \psi_2^{-1}C\\
   &=A\times X_2&&\text{(modulo $\mu$)},
 \end{align*}
 where we used invariance of $\mu$ under translation by
 $(x',0)$ in the transition from the first to the second line.
 It follows that
 \[
   \lambda_{X_1}\bigl((A+x')\vartriangle A\bigr)=
   \mu\left (\left ((A+x')\times X_2\right )\vartriangle \left (A\times X_2\right ) \right )=0
 \]
 for any $A \in \cA_1$ and so $x'\in X_1'$. Therefore, $x'=0$ by assumption
 \eqref{no translation}. The same holds for elements
 of the form $(0,x')\in G$. Together with our assumption
 that $G$ projects onto $X_1$ and onto $X_2$ this implies that
 $G$ is the graph of a group isomorphisms $\phi:X_1\rightarrow
 X_2$.

 We show next that $w_\phi$ is uniquely determined. So suppose
 $w,w'\in X_2$ are such that \eqref{definition of Phi} holds 
 independently of whether $\Phi$ is defined using $w$ or using $w'$.
 Let $v=\phi^{-1}(w-w')$. Then
 \[
   \psi_1(x)= \psi_2\bigl(\phi(x)+w\bigr)=
      \psi_2\bigl(\phi(x+v)+w'\bigr)= \psi_1(x+v)
 \]
 for $ \lambda_{X_1}$-a.e.\ $x\in X$. Therefore, $v\in X_1=\{0\}$
 (by \eqref{no translation} again) and $w=w'$.

 It remains to show the existence of $w_\phi$ and that
 $\mu$ is the Haar measure of the graph of $\Phi$.
 Since $\mu $ is invariant under translation by elements of $G$
 and since $\mu(D_Y)=1$, it follows that
 \[
   \mu\bigl(D_Y-(x,\phi(x))\bigr)=1\qquad\text{for every $x\in X_1$}.
 \]
 In other words, we know for $\mu\times \lambda_{X_1}$-a.e.\
 $((z_1,z_2),x)\in (X_1\times X_2)\times X_1$ that
 $(z_1,z_2)+(x,\phi(x))\in D_Y$. By Fubini's theorem
 this shows for $\mu$-a.e.\ $(z_1,z_2)\in X_1\times X_2$
 that
 \[
  (z_1,z_2)+(x,\phi(x))\in D_Y\qquad\text{for $\lambda_{X_1}$-a.e.\ $x\in X_1$.}
 \]
 However, by definition of $D_Y$ this is equivalent to
 \[
   \psi_1(z_1+x)= \psi_2\bigl(z_2+\phi(x)\bigr)\mbox{ for $\lambda_{X_1}$-a.e.\ }x\in X_1.
 \]
 We define $w=z_2-\phi(z_1)$, then \eqref{definition of Phi}
 follows for $\Phi(x)=\phi(x)+w$. However, by uniqueness
 $w=w_\phi$ is independent of $(z_1,z_2)$. Therefore,
 $z_2=\phi(z_1)+w_\phi$ and $(z_1,z_2)$
 belong to the graph of $\Phi(x)=\phi(x)+w_\phi$.
 This holds for $\mu$-a.e.\ $(z_1,z_2)$, and together
 with invariance of $\mu$ under translation by elements of $G$
 it follows that $\mu$ is the Haar measure of the graph of $\Phi$.
\end{proof}

Finally, we can describe the structure of the relatively
independent joining.

\begin{lemma}\label{structure}
The relatively independent joining is a convex combination
\[
 \nu=\sum_{j\in J}a_j\lambda_{\Phi_j}\mbox{ with }a_j>0
\]
of at most countably many Haar measures $\lambda_{\Phi_j}$ on graphs of affine
isomorphisms $\Phi_j$ that satisfy \eqref{definition of Phi}.
\end{lemma}

\begin{proof}
 Let $\mu $ be an ergodic component of $\nu$ as in
 Lemma \ref{ergodic component}. By Theorem~\ref{thm: main}
 there exist $\mu_1,\ldots,\mu_M$ and $G_1,\ldots,G_M$
 such that $\mu=\frac{1}{M}\sum_{i=1}^M \mu_i$ and
 $\mu_i$ is invariant under translation by elements of $G_i$.
 By Proposition \ref{pro: joining} each $G_i$
 projects surjectively to $X_1$ and $X_2$.
 Therefore, Lemma \ref{second invariance} shows that
 $\mu_i$ is the Haar measure $ \lambda_{\Phi}$
 of the graph of an affine isomorphism $\Phi(x)=\phi(x)+w_\phi$
 that satisfies  \eqref{definition of Phi}.
 
 We claim that there
 are at most countably many group isomorphisms
 $\phi:X_1\rightarrow X_2$.
 Since $\phi$ uniquely determines $w_\phi$, $\Phi$, $\mu_i$,
 and by ergodicity also $\mu$, the above claim implies that
 there are at most countably many ergodic components $\mu$
 each of which is a convex combinations of Haar measures.

 To prove the claim, it is enough to notice that
 every $\phi$ as above is uniquely determined by its dual
 $\hat\phi:\hat X_2 \rightarrow\hat X_1$ that is the restriction
 of a $\QQ$-linear map from $\QQ^{n_2}\supseteq\hat X_2$
 to $\QQ^{n_1}\supseteq\hat X_1$.
\end{proof}

\begin{proof}[Proof of Corollary \ref{thm: algebra}]
 It is well known that every invariant $\sigma $-algebra
 $\cA $ can be realized as $ \psi^{-1}\cB_Y$ for some
 factor map $ \psi:X\rightarrow Y$ and some $\zd$-action
 on a standard Borel probability space $(Y,\cB_Y,\rho)$.
 Recall that by Lemma \ref{construction of X_1'} we may assume
 that \eqref{no translation} holds.

 Let $\nu $ be the relatively independent joining of
 $\lambda_X$ and $\lambda_X$ over $(Y,\rho)$.
 Let
 \begin{multline}\label{definition of Gamma}
  \Gamma=\Bigl\{\gamma:\gamma \mbox{ is an affine automorphism }\\
  \mbox{ such that } \psi(\gamma(x))= \psi(x)\mbox{ for $\lambda_X$-a.e.\ }x\in X\Bigr\}.
 \end{multline}
 Then $\Gamma$ is a group normalized by $\alpha $
 that is at most countable (see proof of Lemma~\ref{structure})
 and satisfies
 \[
  \nu=\sum_{\gamma\in\Gamma}a_\gamma \lambda_\gamma\qquad\text{with $a_\gamma\geq 0$}
 \]
 by Lemma \ref{structure}. 
 From the construction
 \eqref{definition of joining} of the relatively independent joining
 it follows
that the conditional measures of $\nu $ with respect to the
$\sigma $-algebra $\cC=\cB_X\times\{\emptyset,X\}$ are
\[
 \nu_{(x,x')}^\cC=\delta_{x}\times (\lambda_{X})_x^\cA\qquad\text{for $\nu$-a.e.\ $(x,x')\in X\times X$.}
\]
However, from the above decomposition of $\nu$ it is also easy to
calculate the conditional measures of $\nu$ with respect to $\cC$,
which shows that
\[
 (\lambda_{X})_x^\cA=\sum_{\gamma\in\Gamma}a_\gamma\delta_{\gamma(x)}
\qquad\text{$\lambda_X$-a.e.}
\]
Since every $\gamma_0\in\Gamma$ preserves $ \lambda_X$ and $\cA$, it
follows that $(\lambda_{X})_{\gamma_0 x}^\cA=(\gamma_0)_*(\lambda_{X})_x^\cA$ a.e., and hence
\[
  (\lambda_{X})_{\gamma_0 x}^\cA=
 (\gamma_0)_*\sum_{\gamma\in\Gamma}a_\gamma\delta_{\gamma(x)}
 =\sum_{\gamma\in\Gamma}a_\gamma\delta_{\gamma_0(\gamma(x))}\mbox{ a.e.}
\]
However, by definition $\gamma_0$ also preserves a.e.\ atom of
$\cA$ and so $(\lambda_{X})_{\gamma_0 x}^\cA=(\lambda_{X})_x^\cA$ a.e., hence
\[
 (\lambda_{X})_{\gamma_0x}^\cA=\sum_{\gamma\in\Gamma}a_\gamma\delta_{\gamma(x)}
\qquad\text{$\lambda_X$-a.e.}
\]
By comparing the above two displayed formul\ae\ we conclude that $\Gamma$ is finite, and all the coefficients 
$a_{\gamma_0}$ are equal to each other, that is to say
\[
   (\lambda_{X})_x^\cA=\frac{1}{|\Gamma|}\sum_{\gamma\in\Gamma}\delta_{\gamma(x)}.
\]
Since the conditional measures determine the $\sigma $-algebra
(modulo $ \lambda_X$), the corollary follows.
\end{proof}

\begin{proof}[Proof of Corollary \ref{theorem about common factors}]
 We already constructed $X_1'$ and $X_2' $ in Lemma \ref{construction of X_1'}.
 Applying Corollary \ref{thm: algebra} we find $\Gamma_1$
 and $\Gamma_2 $.  Finally, let $\Phi:X_1\rightarrow X_2$
 be an affine isomorphisms as in \eqref{definition of Phi}
 that exists by Lemma \ref{structure}. Then $\Phi^{-1}\circ\gamma_2\circ\Phi$
 belongs to $\Gamma_1$ (defined as in \eqref{definition of Gamma})
 for any $\gamma_2\in\Gamma_2$. By symmetry $\Gamma_2\circ\Phi=\Phi\circ\Gamma_1$
 which concludes the proof.
\end{proof}

\bibliographystyle{amsplain}

\def\cprime{$'$} \providecommand{\bysame}{\leavevmode\hbox
to3em{\hrulefill}\thinspace}
\providecommand{\MR}{\relax\ifhmode\unskip\space\fi MR }
\providecommand{\MRhref}[2]{
  \href{http://www.ams.org/mathscinet-getitem?mr=#1}{#2}
} \providecommand{\href}[2]{#2}

\end{document}